\documentclass[11pt]{article}
\usepackage{amsfonts,mathrsfs}
\usepackage{amssymb,amsthm}
\usepackage{amsmath}
\usepackage{hyperref}
\usepackage{color}

=0pt
\def\sqr#1#2{{\vcenter{\vbox{\hrule height.#2pt
\hbox{\vrule width.#2pt height#1pt \kern#1pt \vrule
width.#2pt}
\hrule height.#2pt}}}}
\def\dbE{{\mathbb{E}}}
\def\dbF{{\mathbb{F}}}

\def\dbH{{\mathbb{H}}}

\def\dbN{{\mathbb{N}}}

\def\dbP{{\mathbb{P}}}

\def\dbR{{\mathbb{R}}}

\def\d{\delta}
\def\e{\varepsilon}

\def\f{\varphi}

\def\om{\omega}

\def\bd{{\mathbf d}}
\def\3n{\negthinspace \negthinspace \negthinspace }
\def\2n{\negthinspace \negthinspace }
\def\1n{\negthinspace }
\def\ns{\noalign{\smallskip} }

\def\ds{\displaystyle}
\def\G{\Gamma}
\def\D{\Delta}

\def\Om{\Omega}
\def\om{\omega}

\def\cC{{\cal C}}

\def\cF{{\cal F}}

\def\cJ{{\cal J}}

\def\cL{{\cal L}}

\def\cP{{\cal P}}

\def\cS{{\cal S}}

\def\cU{{\cal U}}

\def\cX{{\cal X}}

\def\mE{{\mathbb{E}}}

\def\no{\noindent}

\def\ss{\smallskip}
\def\ms{\medskip}
\def\bs{\bigskip}
\def\q{\quad}
\def\qq{\qquad}
\def\hb{\hbox}

\def\liminf{\mathop{\underline{\rm lim}}}

\def\lan{\big\langle}
\def\ran{\big\rangle}

\def\essinf{\mathop{\rm essinf}}

\def\pa{\partial}

\def\wt{\widetilde}
\def\cd{\cdot}

\def\dim{\hbox{\rm dim$\,$}}

\def\ae{\hbox{\rm a.e.{ }}}

\def\deq{\mathop{\buildrel\D\over=}}

\def\({\Big (}
\def\){\Big )}
\def\[{\Big[}
\def\]{\Big]}

\def\={\buildrel \triangle \over =}

\def\be{\begin{equation}}
\def\bel{\begin{equation}\label}
\def\ee{\end{equation}}
\def\bea{\begin{eqnarray}}
\def\eea{\end{eqnarray}}
\def\bt{\begin{theorem}}
\def\et{\end{theorem}}
\def\bc{\begin{corollary}}
\def\ec{\end{corollary}}
\def\bl{\begin{lemma}}
\def\el{\end{lemma}}
\def\bp{\begin{proposition}}
\def\ep{\end{proposition}}
\def\br{\begin{remark}}
\def\er{\end{remark}}
\def\ba{\begin{array}}
\def\ea{\end{array}}
\def\bd{\begin{definition}}
\def\ed{\end{definition}}

\newtheorem{lemma}{Lemma}[section]
\newtheorem{remark}{Remark}[section]

\newtheorem{theorem}{Theorem}[section]
\newtheorem{corollary}{Corollary}[section]

\newtheorem{definition}{Definition}[section]
\newtheorem{proposition}{Proposition}[section]

\DeclareMathOperator*\lowlim{\underline{lim}}
\DeclareMathOperator*\uplim{\overline{lim}}

\allowdisplaybreaks
\makeatletter

\@addtoreset{equation}{section}
\makeatother
\begin{document}

\title{\bf Relationships Between the Maximum Principle and Dynamic Programming for Infinite Dimensional Non-Markovian Stochastic Control Systems\thanks{This work is partially supported by the NSF of China under grant 12025105.}}

\author{
Qi L\"{u}
\footnote{School of Mathematics, Sichuan University, Chengdu, P. R. China. Email: lu@scu.edu.cn. }
~~~ and ~~~
Dingqian Gao
\footnote{School of Mathematics, Sichuan University, Chengdu, P. R. China. Email: dingqiangao@outlook.com.}
}

\date{}

\maketitle

\begin{abstract}
This paper investigates the relationship between Pontryagin's maximum principle  and dynamic programming principle in the context of stochastic optimal control systems governed by stochastic evolution equations with random coefficients in separable Hilbert spaces. Our investigation proceeds through three contributions:
(1). We first establish the formulation of the dynamic programming principle for this class of infinite-dimensional stochastic systems, subsequently deriving the associated stochastic Hamilton-Jacobi-Bellman equations that characterize the value function's evolution. (2). For systems with smooth value functions, we develop explicit correspondence relationships between Pontryagin's maximum principle and dynamic programming principle, elucidating their fundamental connections through precise mathematical characterizations. (3). In the more challenging non-smooth case, we employ tools in  nonsmooth  analysis and  relaxed transposition solution techniques to uncover previously unknown sample-wise relationships between the two principles.
\end{abstract}

\bs

\no{\bf 2020 Mathematics Subject
Classification}. 93E20.

\bs

\no{\bf Key Words:}  Pontryagin type maximum principle; dynamic programming; optimal control; infinite dimensional non-Markovian stochastic control systems. 

\section{Introduction}

Pontryagin's maximum principle (PMP) and Bellman's dynamic programming principle (DPP) represent two cornerstone methodologies in optimal control theory, each providing distinct but complementary approaches to deriving optimal controls (see, e.g., \cite{Bardi1997,Bellman2010,Coron2013,Fabbri2017,Fleming2006,Lu2014,Lu2021,Pontryagin1962,Yong1999}).  This duality naturally gives rise to:

\ss

{\bf Problem (R).} What is the  mathematical relationship between PMP and DPP?

\ss

The seminal work by Pontryagin et al. \cite[Chapter 1, Section 9]{Pontryagin1962} first established the relationship between PMP and DPP for control systems governed by ordinary differential equations (ODEs). Their analysis required the critical assumption that the value function remains smooth throughout the domain of interest. For decades, these theoretical connections remained largely formal (with the exception of certain special cases) due to the inherent nonsmooth nature of value functions in optimal control problems. 
Significant theoretical advances had been made in \cite{Barron1986,Clarke1987,Zhou1990} by employing tools from non-smooth analysis. These work successfully extended the PMP-DPP relationship to cases where the value function lacks smoothness, thereby substantially broadening the applicability of the theory to more general control problems  governed by ODEs.  

Subsequent developments have systematically extended these foundational results to control systems governed by both partial differential equations and stochastic differential equations (see, e.g., \cite{be1982,Bismut1978,Cannarsa1992,Cannarsa1996,Zhou1991}). Following these pioneering works, {\bf Problem (R)} has been thoroughly investigated across various classes of control systems. Given the extensive literature in this field, we refer readers to   \cite{Backhoff2017,Bokanowski2021,Cernea2005,Cannarsa2015,Dong2024,Frankowska2013,Hermosilla2023,Hu2020,Nie2017,Pakniyat2017,Subbotina2006,Wu2025} and their comprehensive bibliographies for a complete overview of existing results.

To our knowledge, \cite{ChenLu} remains the only published work addressing \textbf{Problem (R)} for systems governed by stochastic evolution equations (SEEs) in infinite-dimensional spaces. In their study, the authors examine the relationship between PMP and DPP for Markovian control systems described by SEEs, where control variables appear in both drift and diffusion terms.

The present work advances this research direction by investigating \textbf{Problem (R)} for non-Markovian SEEs in separable Hilbert spaces. As noted in \cite{ChenLu}, the infinite-dimensional setting presents unique challenges: establishing PMP-DPP relationships for non-smooth cases necessitates the transposition solution concept from \cite{Lu2014} to circumvent limitations in stochastic integration theory for general Banach spaces, particularly for second-order adjoint equations. Moreover, the non-Markovian nature of our system introduces additional complexities in both deriving the DPP and connecting it with PMP, going beyond the challenges addressed in \cite{ChenLu}. These technical aspects are developed in detail in Sections \ref{sec-DPP}.

\ss

Before giving the mathematical formulation of our problem,  we first introduce some notations.

\ss

Let $T > 0$ be fixed. Consider a complete filtered probability space $(\Omega, \mathcal{F}, \{\mathcal{F}_t\}_{t\in [0,T]}, \mathbb{P})$, on which a
separable Hilbert space $\wt H$-valued cylindrical Brownian motion
$W(\cdot)$ is defined and $\mathbf{F}\deq \{\cF_t\}_{t\in [0,T]}$ is
the natural filtration generated by $W(\cdot)$. Denote by $\mathbb{F}$ or $\mathscr{P}$ the progressive $\sigma$-algebra with respect to
$\mathbf{F}$.

Let $\cX$ be a Banach space. For any $t\in[0,T]$ and $p\in
[1,\infty)$,   write $L_{\cF_t}^p(\Om;\cX)$ for the Banach space of
all $\cF_t$-measurable, $\cX$-valued $p$th power integrable random variables with the canonical norm. Denote by 
$L^{p}_{\dbF}(\Om;C([t,T];\cX))$ the space of $\mathbf{F}$-adapted continuous processes with the norm 
$$
|\phi(\cd)|_{L^{p}_{\dbF}(\Om;C([t,T];\cX))} \=
\[\mE\sup_{\tau\in
[t,T]}|\phi(\tau)|_\cX^p\]^{1/p};$$
by 
$C_{\dbF}([t,T];L^{p}(\Om;\cX))$ the space of $\mathbf{F}$-adapted processes with continuous paths in $L^p$, equipped with  the norm 
$$
|\phi(\cd)|_{C_{\dbF}([t,T];L^{p}(\Om;\cX))} \= \sup_{\tau\in
[t,T]}\left[\mE|\phi(\tau)|_\cX^p\right]^{1/p};
$$
by
$D_{\dbF}([0,T];L^{p}(\Om;\cX))$ the space of c\`adl\`ag, (i.e., right
continuous with left limits) $\mathbf{F}$-adapted processes with the same norm  as $C_{\dbF}([t,T];L^{p}(\Om;\cX))$.

For any fixed exponents $p_1,p_2,p_3,p_4 \in [1,\infty]$, we define the following function spaces:\vspace{-2mm}
$$
\begin{array}{ll}
\ds L^{p_1}_\dbF(\Om;L^{p_2}(t,T;\cX)) =\Big\{\f:(t,T)\times\Om\to
\cX\;\Big|\;\f(\cd)\hb{
is $\mathbf{F}$-adapted and }\dbE\(\int_t^T|\f(\tau)|_\cX^{p_2}d\tau\)^{\frac{p_1}{p_2}}<\infty\Big\},\\
\ns\ds
L^{p_2}_\dbF(t,T;L^{p_1}(\Om;\cX)) =\Big\{\f:(t,T)\times\Om\to
\cX\;\Big|\;\f(\cd)\hb{ is $\mathbf{F}$-adapted and
}\int_t^T\(\dbE|\f(\tau)|_X^{p_1}\)^{\frac{p_2}
{p_1}}d\tau<\infty\Big\}.
\end{array}
$$
When $p_1 = p_2$, we use the simplified notation $L^{p_1}_\mathbb{F}(t,T;\mathcal{X})$. 
For the scalar case $\mathcal{X} = \mathbb{R}$, we omit $\mathcal{X}$ from the notation. 
The standard modifications apply when any $p_j = \infty$. 
Both $L_{\mathbb{F}}^{p_1}(\Omega;L^{p_2}(t,T;\mathcal{X}))$ and $L_{\mathbb{F}}^{p_2}(t,T;L^{p_1}(\Omega;\mathcal{X}))$ are complete Banach spaces when equipped with their natural norms.

For any $t \in [0,T]$ and $f \in L^1_{\mathcal{F}_T}(\Omega;\mathcal{X})$, we denote by 
$\mathbb{E}(f|\mathcal{F}_t)$ the conditional expectation with respect to $\mathcal{F}_t$ and by 
$\mathbb{E}f$ the standard mathematical expectation.

\ss

Let $\mathcal{Y}$ be another Banach space. We write 
$\mathcal{L}(\mathcal{X}; \mathcal{Y})$ for the Banach space of bounded linear operators from $\mathcal{X}$ to $\mathcal{Y}$ equipped with the standard operator norm and 
$\mathcal{L}(\mathcal{X})$ for $\mathcal{L}(\mathcal{X}; \mathcal{X})$ when $\mathcal{Y} = \mathcal{X}$. If $\mathcal{Y}$ is a Hilbert space, we denote by
$\mathcal{S}(\mathcal{Y})$   the closed subspace of self-adjoint operators in $\mathcal{L}(\mathcal{Y})$.

For any operator-valued process/random variable $M$, we denote by $M^*$ its pointwise adjoint. For instance: If $M\in L^{r_1}_\dbF(0,T;
L^{r_2}(\Om; \cL(\cX)))$, then  $M^*\in L^{r_1}_\dbF(0,T; L^{r_2}(\Om;
\cL(\cX')))$ and  $|M|_{L^{r_1}_\dbF(0,T; L^{r_2}(\Om; \cL(\cX)))}=|M^*|_{L^{r_1}_\dbF(0,T; L^{r_2}(\Om; \cL(\cX')))}$.

Let $ H$ be a separable Hilbert space, and $K$ be a Banach space.

Let $ k \in \mathbb{N} $. We say that a stochastic process $ \phi $ belongs to $ \mathbb{S}_{\mathbb{F}}^{k}([0,T] \times H,K) $ if the following hold:
{\it i)} $ \phi: [0,T] \times \Omega \times H \to K$ is $\mathscr{P}\otimes \mathcal{B}(H)/\mathcal{B}(K)$ measurable;
{\it ii)} For almost all $ (t,\omega) \in [0,T] \times \Omega $, the mapping $ \phi(t,\cdot) $ is $ k $-times continuously Fr\'{e}chet differentiable, and for each fixed $ x \in H $, $\phi(\cdot,x),\dots,\partial_x^k\phi(\cdot,x)$ are $\mathbf{F}$-adapted processes. Additionally, if for each fixed $x\in H$, $\phi(\cdot,x),\dots,\partial_x^k\phi(\cdot,x)$ are continuous, we say that $\phi\in \mathbb{C}_{\mathbb{F}}^{0,k}([0,T]\times H,K)$. If $K=\mathbb{R}$, we write $ \mathbb{S}_{\mathbb{F}}^{k}([0,T] \times H),\mathbb{C}_{\mathbb{F}}^{0,k}([0,T]\times H)$ for simplicity.

Let $A:D(A)\subset H\to H$ be
a linear operator, which generates a generalized contractive $C_{0}$ -semigroup (i.e., $|S(s)|_{\cL(H)}\le e^{cs}$ for some constant $c\in \mathbb{R}$ and any $s\ge 0$).
Write $ \mathcal{L}_2^0$ for the space
of all Hilbert-Schmidt operators from $\wt H$ to $H$, which is also
a separable Hilbert space.  

Let $U$ be a closed and bounded subset of some separable Hilbert space with norm $|\cdot|_U$, and define $U_m \deq \sup_{u \in U} |u|_U$. For $0\le s\le r\le T$, put\vspace{-2mm}
\begin{equation*}
\mathcal{U}[s,r]\deq\big\{u:[s,r]\times\Omega\to U\big| u \mbox{ is $\mathbf{F}$-adapted}\big\}.
\end{equation*}

Now we can introduce the control system to be studied in this paper:\vspace{-2mm}
\begin{equation}\label{system1}
\begin{cases}\ds
dX(t)=\big(AX(t)+a(t,X(t),u(t))\big)dt+b(t,X(t),u(t))dW(t), &t \in (0,T],\\
\ns\ds X(0)=X_0 \in H.
\end{cases}
\end{equation}
The cost functional is\vspace{-2mm}
\begin{equation}\label{cost1}
\mathcal{J}(x;u(\cdot))=\mathbb{E}\Big(\int_0^T
f(t,X(t),u(t))dt+h(X(T))\Big).
\end{equation}

We make the following assumptions for the control system
\eqref{system1} and the cost functional \eqref{cost1} (In this paper, $\mathcal{C}$
is a generic constant which may vary from line to line).

\ss

{\bf (S1)} {\it  Suppose that: i) $a: [0,T]\times \Omega\times H\times U \to H$ is $\mathscr{P}\otimes \mathcal{B}(H)\otimes \mathcal{B}(U)/\mathcal{B}(H)$-measurable and $b:  [0,T]\times \Omega\times H\times U \to \mathcal{L}_2^0$ is $\mathscr{P}\otimes \mathcal{B}(H)\otimes \mathcal{B}(U)/\mathcal{B}(\mathcal{L}_2^0)$-measurable; ii) For any $x\in H$ and $a.e.$ $\omega\in \Omega$, the maps $a(\cdot,x,\cdot):[0,T]\times U\to H $ and
$b(\cdot,x,\cdot):[0,T]\times U\to \mathcal{L}_2^0$ are continuous; and iii)
For any $(t,x_1,x_2,u_1,u_2)\in [0,T]\times H \times H\times U\times U$,\vspace{-2mm}
\begin{equation*}
\begin{cases}
|a(t,x_1,u_1)-a(t,x_2,u_2)|_H+|b(t,x_1,u_1)-b (t,x_2,u_2)|_{\mathcal{L}_2^0}\leq \mathcal{C}(|x_1-x_2|_H + |u_1-u_2|_U), \\
|a(t,x_1,u_1)|_H +|b(t,x_1,u_1)|_{\mathcal{L}_2^0}\leq
\mathcal{C}(1+|x_1|_{H}+|u_1|_U).
\end{cases}
\end{equation*}}

\ss

{\bf (S2)} {\it  Suppose that: i) $f:[0,T]\times\Omega\times H\times U\to \mathbb{R}$
is $\mathscr{P}\otimes \mathcal{B}(H)\otimes
\mathcal{B}(U)/\mathcal{B}(\mathbb{R})$-measurable and
$h:\Omega\times H\to \mathbb{R}$ is
$\mathcal{F}_T\otimes\mathcal{B}(H)/\mathcal{B}(\mathbb{R})$-measurable; ii) For any
$x\in  H$ and $a.e.$ $\omega\in \Omega$, the functional $f(\cdot,x,\cdot): [0,T]\times U\to
\mathbb{R}$ is continuous; and iii) For almost all $(t, \omega) \in [0,T] \times \Omega$ and $(x_1, u_1),  (x_2, u_2) \in H \times U$, \vspace{-4mm} 
\begin{eqnarray*}
\begin{cases}\ds
|f(t,x_1, u_1)-f(t,x_2, u_2)|+|h(x_1)-h(x_2)|\\
\ns\ds\leq \mathcal{C} \big(1+|x_1|_H+|x_2|_H+|u_1|_U+|u_2|_U\big)
\big(|x_1-x_2|_H+|u_1-u_2|_U\big),
\\ \ns\ds
|f(t,x_1, u_1)|+|h(x_1)|\leq \mathcal{C}\big(1+|x_1|_H^2+|u_1|_U^2\big).
\end{cases}
\end{eqnarray*}
}

\ss

{\bf (S3)} {\it For any $(t,u)\in [0,T]\times U$ and $a.e.$ $\omega\in \Omega$, the maps
$a(t,\cd,u)$ and $b (t,\cd,u)$, and the functionals $f(t,\cd,u)$ and
$h(\cd)$ are $C^2$, such that for $\phi(t,x,u)=a(t,x,u),\ b(t,x,u)$
and $\Psi (t,x,u)=f(t,x,u),\ h(x)$, it holds that $\phi_x(t,x,\cd),\
\Psi_x(t,x,\cd),\ \phi_{xx}(t,x,\cd),$ and $\Psi_{xx}(t,x,\cd)$ are
continuous.  Moreover, there is a  modulus of continuity
$\bar{\omega}:[0,\infty)\to [0,\infty)$ such that for any
$(t,x,x_1,x_2,u,u_1,u_2)\in [0,T]\times H\times H\times
H\times U\times U\times U $,\vspace{-2mm}
\begin{equation*}
\begin{cases}
|a_{xx}(t,x,u)|_{\mathcal{L}(H,H;H)}+|b_{xx}(t,x,u)|_{\mathcal{L}(H,H;\mathcal{L}_2^0)}+|\Psi_{xx}(t,x,u)|_{\mathcal{L}(H)}\leq \cC;\\
|a_{xx}(t,x_1,u_1)-a_{xx}(t,x_2,u_2)|_{\mathcal{L}(H,H;H)}+|b _{xx}(t,x_1,u_1)-b_{xx}(t,x_2,u_2)|_ {\mathcal{L}(H,H;\mathcal{L}_2^0)}\\
\  + |\Psi_{xx}(t,x_1,u_1)-\Psi_{xx}(t,x_2,u_2)|_{\mathcal{L}(H)} \leq \bar{\omega}(|x_1-x_2|)+ \cC |u_1-u_2|_U.
\end{cases}
\end{equation*}}

Under {\bf (S1)}, for any $u(\cdot)\in \mathcal{U}[0,T]$,   the control system
\eqref{system1} has a unique mild solution $X(\cdot) \in L^{p}(\Omega;C([0,T];H)),p\ge 2$  (see \cite[Theorem 3.21]{Lu2021}).

\ss

Consider the following optimal control problem:

\textbf{Problem (OP)}.  For any given $x\in
H$, find a
$\bar{u}(\cdot)\in \mathcal{U}[0,T]$ such that\vspace{-2mm}
\begin{equation}\label{OP1}
\mathcal{J}(x,\bar{u}(\cdot))=\inf\limits_{u(\cdot)\in \mathcal{U}[0,T]}\mathcal{J}(x,u(\cdot)).
\end{equation}

Any $\bar{u}(\cdot)\in \mathcal{U}[0,T]$ satisfying \eqref{OP1} is called an {\it optimal control} (of
\textbf{Problem (OP)}). The corresponding state
$\overline{X}(\cdot)$ is called an {\it optimal state}, and
$(\overline{X}(\cdot),\bar{u}(\cdot))$ is called an {\it optimal pair}.

The paper proceeds as follows:  

Section \ref{sec-DPP} develops the DPP for \textbf{Problem (OP)}, providing the theoretical foundation for our analysis. 

In Section \ref{sec-relation}, we establish the connections between the PMP and DPP for \textbf{Problem (OP)}, presenting the main theoretical contributions of this work.

Section \ref{sec-exam} demonstrates two concrete examples that satisfy all assumptions required by our main theorems (Theorems \ref{thm-sdy1}, \ref{th5.1}, and \ref{th6.1}).

\section{Dynamic programming principle for \textbf{Problem (OP)}}\label{sec-DPP}

This section is devoted to establishing the  DPP for \textbf{Problem (OP)}. In the literature, two principal approaches exist for deriving DPP in the context of stochastic optimal control problems governed by SEEs:

The first one is the relaxed control framework approach. In this method,   
one considers not only the control process $u$ but also incorporates the probability space $(\Omega, \mathcal{F}, \{\mathcal{F}_t\}_{t\in [0,T]}, \mathbb{P})$ and Wiener process $W(\cd)$ as part of the controls. It works well for Markovian control systems (see \cite[Chapter 2]{Fabbri2017}). However, it cannot be applied to non-Markovian systems due to the inherent dependence of coefficients on the fixed probability space structure.

The second one follows the idea in \cite[Chapter 2]{Peng1997} (see also \cite{Peng2024}) and establish the DPP without using the relaxed framework (see \cite{ChenLu}).

All coefficients in both the control system \eqref{system1} and the cost functional \eqref{cost1} are stochastic processes or random variables. This feature makes \textbf{Problem (OP)} inherently non-Markovian in nature. Hence we adopt the second approach to handle the technical challenges posed by random coefficients in infinite dimensions and develop a suitable DPP formulation for \textbf{Problem (OP)}. 


This section first recalls some  notations and basic results. Subsequently,  we formulate and establish the DPP for \textbf{Problem (OP)} and derive respective stochastic HJB equations.

\subsection{Preliminaries}\label{ssec-pre}

Let us first recall the definition of the essential infimum of a family of nonnegative random variables. 

\begin{definition}\label{def1}
Let $\mathscr{X}$ be a nonempty family of nonnegative random variables defined on $(\Omega, \mathcal{F}, \mathbb{P})$. The essential infimum of $\mathscr{X}$, denoted by $\operatorname*{essinf}\mathscr{X}$  or $\essinf\limits_{X\in \mathscr X}X$, is a random variable $X^{*}$ satisfying: i) for all $X\in \mathscr X$, $X\geq X^*$, $\dbP$-a.s.;
ii) if $Y$ is another random variable satisfying $ X\geq Y $, $\dbP$-a.s.,
for all $ X\in \mathscr X$, then $X^*\geq Y$, $\dbP$-a.s.
\end{definition}
\begin{lemma}\label{lemma1}\cite[Theorem A.3]{karatzas1998methods}
	Let $\mathscr{X}$ be a nonempty family of nonnegative random variables. Then $X^{*} = \operatorname*{essinf} \mathscr{X}$ exists. Furthermore, if $\mathscr X$ is closed  under pairwise minimum
	(i.e., $X, Y\in \mathscr X$ implies $X\wedge Y\in \mathscr X$), then
	there exists a nonincreasing  sequence  $\{Z_n\}_{n\in \mathbb N}$
	of random variables in $\mathscr X$ such that $X^*=
	\lim\limits_{n\rightarrow\infty} Z_n$, $\dbP$-a.s. Moreover, for any
	sub-algebra $\mathcal G$ of $\mathcal F$, the $\mathcal
	G$-conditional expectation is interchangeable with the essential
	infimum, that is,\vspace{-2mm}
	\begin{eqnarray}
		\mathbb E \big(\essinf\limits_{X\in \mathscr X}X \big |\mathcal G
		\big)= \essinf\limits_{X\in \mathscr X} \mathbb E\big(X
		\big|\mathcal G\big), \q\dbP\text{-a.s.}
	\end{eqnarray}
\end{lemma}

By applying Lemma
\ref{lemma1} to the family $\{ X-Y| X \in \mathscr X\}$, we get the following result:
\begin{corollary}\label{cor1}
Let $\mathscr X$ be a family of random variables that are uniformly
bounded from below by another random variable $Y$, i.e., for each $
X \in \mathscr X$, $X \ge Y$. Then the conclusions in Lemma
\ref{lemma1} also hold.
\end{corollary}

For any given initial condition $(t,\xi) \in [0,T] \times L^p_{\mathcal{F}_t}(\Omega; H),p\ge 2$ and control process $u(\cdot) \in \mathcal{U}[t,T]$, we consider the controlled stochastic evolution system:\vspace{-2mm}
\begin{eqnarray} \label{system2}
\left\{
\begin{array}{ll}\ds
dX(s)= AX(s) + a(s,X(s),u(s))ds+b(s,X(s),u(s))d W(s) &\mbox{ in } (t,T],\\
\ns\ds X(t)=  \xi.
\end{array}
\right.
\end{eqnarray}
Under {\bf (S1)}, for any $u(\cdot)\in \mathcal{U}[t,T]$, the control system 
\eqref{system2} admits a unique mild solution $X(\cdot;t,\xi,u) \in
L^{p}(\Omega;C([0,T];H)),p\ge 2$  (see \cite[Theorem 3.21]{Lu2021}).
 In the rest of this paper,  when the dependence on $(t,\xi,u)$ is clear from context or not essential to emphasize, we shall simply write $X(\cdot)$ to denote the solution of \eqref{system2} for notational brevity.


We have the following fundamental estimates:

\begin{lemma}\label{lemma_esti_sde}
Let Assumption  {\bf (S1)} hold. For $t\in [0,T]$, $u(\cdot),\tilde
u(\cdot)\in\cU[t, T], p\ge 2$, and $\xi,\tilde \xi \in L^p_{\cF_t}(\Omega;
H)$, there exists a positive constant $\cC$ such that\vspace{-2mm}
\begin{equation}\label{eq:3.15-1}
\mathbb E\(\sup_{t\le s \le T}|X(s;t,\xi,u)|_H^p \big |\mathcal F_t
\)\le \cC\big(1+|\xi|_H^p\big),\q\dbP\mbox{-a.s.}\vspace{-2mm}
\end{equation}
and\vspace{-4mm}
\begin{equation} \label{eq:3.15-2}
\begin{array}{ll}\ds
\mathbb E\(\sup_{t\leq s\leq T}
|X(s;t, \xi,u)-X(s;t, \tilde \xi, \tilde  u)|_H^p\big|\cF_{t}\)  \leq \cC \[|\xi-\tilde\xi|^p_H+\mathbb
E\(\int_t^T|u(s)-\tilde u(s)|_{U}^pds\Big| \cF_t\)\],\q\dbP\mbox{-a.s.}\vspace{-2mm}
\end{array}\vspace{-2mm}
\end{equation}
\end{lemma}

Although Lemma \ref{lemma_esti_sde} represents standard results in stochastic analysis, we could not locate a precise reference combining both estimates under our specific assumptions.   Hence, we provide a proof below.

\begin{proof}[Proof of Lemma \ref{lemma_esti_sde}] 
Using the mild solution representation:\vspace{-2mm}
$$
\begin{array}{ll}\ds
X(s) = S(s-t)\xi + \int_t^s S(s-r) a(r,X(r),u(r))d\tau +  \int_t^s S(s-r) b(r,X(r),u(r))dW(r),
\end{array}
$$
we obtain the preliminary estimate:\vspace{-2mm}
\begin{equation}\label{4.11-eq1}
\begin{array}{ll}\ds
\mE\big(\sup_{s\in [t,t_1]}\big|X(s)\big|_H^p\big|\cF_t\big)  \leq\3n&\ds
\cC\bigg[\big|\xi\big|_H^p + \mE\Big(\sup_{s\in [t,t_1]}\Big|\int_t^s
S(s-r) a(r,X(r),u(r))dr\Big|_H^p\Big|\cF_t\Big)\\
\ns&\ds \q+ \mE\Big(\sup_{s\in [t,t_1]}\Big|\int_t^s S(s-r)
b(r,X(r),u(r))dW(r)\Big|_H^p\Big|\cF_t\Big)\bigg]$$.
\end{array}\vspace{-2mm}
\end{equation}

From Assumption {\bf (S1)}, the drift term estimate is\vspace{-2mm}
\begin{equation}\label{4.11-eq2}
\begin{array}{ll}\ds
\mE\Big(\sup_{s\in [t,t_1]}\Big|\int_t^s
S(s-r) a(r,X(r),u(r))dr\Big|_H^p\Big|\cF_t\Big)\\
\ns\ds \leq 
\mE\Big(\sup_{s\in [t,t_1]}\Big\{\Big|\int_t^{s}1dr\Big|^{\frac{1}{q}}\int_t^{s}\big|S(s-r) a(r,X(r),u(r))\big|_H^pdr\Big\}\Big|\cF_t\Big)\\
\ns\ds \leq \mE\Big(|t_1-t|^{1+\frac{1}{q}}\sup_{r\in [0,T]}\big|S(r)|_{\cL(H)}^p\sup_{r\in [t,t_1]}\big|a(r,X(r),u(r))\big|_H^p\Big|\cF_t\Big)\\
\ns\ds \leq \cC\sup_{r\in [0,T]}\big|S(r)|_{\cL(H)}^p|t_1-t|^{1+\frac{1}{q}}\mE\Big(\sup_{r\in [t,t_1]}\big(1+\big|X(r)\big|_H^p+|u|_U^p\big)\Big|\cF_t\Big)\\
\ns\ds \leq \cC\sup_{r\in [0,T]}\big|S(r)|_{\cL(H)}^p|t_1-t|^{1+\frac{1}{q}}\Big(1+\mE \big(\sup_{r\in [t,t_1]}\big|X(r)\big|_H^p\Big|\cF_t\big)+U_m^p\Big),\\
\end{array}\vspace{-2mm}
\end{equation}
where $\frac{1}{p}+\frac{1}{q}=1$.
Similarly, by the assumption on $A$ and using a Burkholder-Davis-Gundy type inequality (see \cite[Theorem 3.18 ]{Lu2021}), we obtain an estimate for the diffusion term:\vspace{-2mm}
\begin{equation}\label{4.11-eq3}
\begin{array}{ll}\ds
\mE\Big(\sup_{s\in [t,t_1]}\Big|\int_t^{s} S(s-r)
b(r,X(r),u(r))dW(r)\Big|_H^p\Big|\cF_t\Big)\\
\ns\ds \leq \cC_p\mE\bigg\{\Big(\int_t^{t_1}\big|b(r,X(r),u(r))\big|_{\cL_2^0}^2 dr\Big)^{\frac{p}{2}}\Big|\cF_t\bigg\}\\
\ns\ds \leq\cC|t_1-t|^{\frac{p}{2}}\Big(1+\mE \big(\sup_{r\in [t,t_1]}\big|X(r)\big|_H^p\Big|\cF_t\big)+U_m^p\Big).\\
\end{array}\vspace{-2mm}
\end{equation}
Denoting \vspace{-2mm}
$$C(t_1):=\cC|t_1-t|^{\frac{p}{2}}+\sup_{r\in [0,T]}\big|S(r)|_{\cL(H)}^p|t_1-t|^{1+\frac{1}{q}},$$
and combining \eqref{4.11-eq1}--\eqref{4.11-eq3},\vspace{-2mm}
\begin{equation}\label{eq4}
\mE\big(\sup_{s\in [t,t_1]}\big|X(s)\big|_H^p\big|\cF_t\big)\le C(t_1)\mE\big(\sup_{s\in [t,t_1]}\big|X(s)\big|_H^p\big|\cF_t\big)+\cC(1+ |\xi|_H^p).\vspace{-2mm}
\end{equation}
Choosing $t_1>0$ such that $C(t_1)=\frac{1}{2}$ yields \vspace{-2mm}
$$
\mE\big(\sup_{s\in [t,t_1]}\big|X(s)\big|_H^p\big|\cF_t\big)\le \cC(1+ |\xi|_H^p).\vspace{-2mm}
$$
Recursively, \eqref{eq:3.15-1} follows.

For the continuous dependence estimates, let $X(\cdot) \equiv X(\cdot;t,\xi,u)$ and $\widetilde{X}(\cdot) \equiv X(\cdot;t,\tilde{\xi},\tilde{u})$: \vspace{-2mm}
\begin{equation*}
\begin{array}{ll}\ds
\mE\big(\sup_{s\in [t,r]}\big|X(s)-\widetilde{X}(s)\big|_H^p\big|\cF_t\big)\\
\ns\ds\leq
\cC\bigg[\big|\xi\big|_H^p + \mE\Big(\sup_{s\in [t,t_1]}\int_t^s\big|
S(s-r) \big[a(r,X(r),u(r))-a(r,\widetilde{X}(r),\tilde{u}(r))\big]dr\big|_H^p\Big|\cF_t\Big)\\
\ns\ds \q+ \mE\Big(\sup_{s\in [t,t_1]}\int_t^s \big|S(s-r)
\big[b(r,X(r),u(r))-b(r,\widetilde{X}(r),\tilde{u}(r))\big]dW(r)\big|_H^p\Big|\cF_t\Big)\bigg].
\end{array}\vspace{-2mm}
\end{equation*}
Using Assumption {\bf (S1)}:\vspace{-2mm}
\begin{equation*}
\begin{array}{ll}\ds
\mE\Big(\sup_{s\in [t,t_1]}\Big|\int_t^s
S(s-r)\big[a(r,X(r),u(r))-a(r,\widetilde{X}(r),\tilde{u}(r))\big]dr\Big|_H^p\Big|\cF_t\Big)\\
\ns\ds \leq 
\mE\Big(\sup_{s\in [t,t_1]}\Big\{\Big|\int_t^{s}1dr\Big|^{\frac{1}{q}}\int_t^{s}\big|S(s-r) \big[a(r,X(r),u(r))-a(r,\widetilde{X}(r),\tilde{u}(r))\big]\big|_H^pdr\Big\}\Big|\cF_t\Big)\\
\ns\ds \leq \cC\sup_{r\in [0,T]}\big|S(r)|_{\cL(H)}^p|t_1-t|^{1+\frac{1}{q}}\Big(\mE\big(\sup_{r\in [t,t_1]}\big|X(r)-\widetilde{X}(r))\big|_H^p\Big|\cF_t\big)+\mE\Big(\int_t^{t_1}|u(r)-\widetilde{u}(r)|_U^pdr\Big|\cF_t\Big)\Big).
\end{array}\vspace{-2mm}
\end{equation*}
Analogously to \eqref{4.11-eq3}, diffusion term estimates follow. Applying the same method as \eqref{eq4} recursively yields \eqref{eq:3.15-2}.
\end{proof}

For convenience and completeness, let us recall the concept of regular conditional probability:
\begin{lemma}\label{lm71}
Let $\mathcal{G}$ be a sub-$\sigma$-algebra of $\mathcal{F}$. Then there exists a map $p:  \Omega\times \mathcal{F}\to [0,1]$, called a regular conditional probability given $\mathcal{G}$, such that 

\ss

(i)  for each $\omega\in \Omega$, $p(\omega,\cdot)$ is a probability measure on $\mathcal{F}$;

\ss

(ii)  for each $A\in \mathcal{F}$, the function $p(\cdot,A)$ is $\mathcal{G}$-measurable;

\ss

(iii)  for each $B\in \mathcal{F}$, $p(\omega,B)=\mathbb{P}(B|\mathcal{G})(\omega)=\mathbb{E}(1_B|\mathcal{G})(\omega),\ \mathbb{P}${\rm-a.s.} 

\ss

We write $\mathbb{P}(\cdot|\mathcal{G})(\omega)$ for $p(\omega,\cdot)$.
\end{lemma}

\subsection{A family of auxillary optimal control problems}\label{scp}

In this subsection, we introduce a family of auxillary optimal control problems, which plays a key role in establishing the DPP for \textbf{Problem (OP)}.

For any given initial data $(t,\xi) \in [0,T]\times L^2_{\mathcal{F}_t}(\Omega; H)$ and control process $u(\cdot)\in\mathcal{U}[t,T]$, we consider the backward stochastic evolution equation (BSEE): \vspace{-2mm}
\begin{equation}\label{bsystem1}
Y(s)=h(X(T)) +\int_s^T f(\tau,X(\tau),u(\tau))d\tau -\int_s^T Z(\tau)dW(\tau), \quad
t\le s\le T,\vspace{-2mm}
\end{equation}
where $X(\cdot)$ is the solution to the forward system \eqref{system2} with the initial data $X(t)=\xi$ and $u(\cdot)$ is the corresponding control.

Under Assumption {\bf (S2)}, the classical well-posedness theory for BSEEs  guarantees that equation \eqref{bsystem1} admits a unique solution pair $(Y,Z) \in L^2_{\mathbb{F}}(\Omega;C([0,T];\mathbb{R})) \times L^{2}_{\mathbb{F}}(0,T;\widetilde{H})$, as established in \cite[Section 4.2]{Lu2021}. To explicitly track the dependence on the initial state $\xi$ and control process $u$, we adopt the notation $(Y(\cdot;t,\xi,u), Z(\cdot;t,\xi,u))$.

Furthermore, the solution to BSEE \eqref{bsystem1} satisfies the following  properties: 
\begin{lemma}\label{lem_bsde}
Let Assumptions  {\bf (S1)}--{\bf (S2)} hold. For any $s\in [t,T]$, the solution
$(Y(\cd;t,\xi,u),Z(\cd;t,\xi,u))$ of BSEE \eqref{bsystem1} satisfies\vspace{-3mm}
$$
|Y (s)|^2 \le \cC\mathbb E\(|h(X (T))|^2+\int_s^T|f(\tau,X
(\tau),u(\tau))|^2d\tau \Big|\mathcal F_s\)\vspace{-2mm}
$$
and\vspace{-2mm}
$$
|Y(s)|\le \cC\big(1+|\xi|_H^2\big).\vspace{-2mm}
$$
\end{lemma}
\begin{proof} 
The first estimate follows from direct application of energy estimates to \eqref{bsystem1}.  
The second one is obtained by combining the growth conditions in Assumption {\bf (S2)} and Lemma
\ref{lemma_esti_sde}.
\end{proof}

For any given initial condition $(t,\xi) \in [0,T]\times L^2_{\mathcal{F}_t}(\Omega; H)$ and control process $u(\cdot)\in\mathcal{U}[t,T]$, we consider the control  system \eqref{system2} with the associated cost functional:
\begin{eqnarray}\label{eq:b12}
\cJ(t,\xi; u(\cdot))=Y(t;t,\xi,u),
\end{eqnarray}
where $(Y(\cdot;t,\xi,u),Z(\cdot;t,\xi,u))$ solves the BSEE \eqref{bsystem1}.

The optimal control problem parameterized
by $(t,\xi)\in [0,T]\times  L^2_{\cF_t}(\Omega; H)$ is formulated as follows:

\ss

{\bf Problem (OP)$_{t,\xi}$} 
Find an admissible control process ${\bar u} (\cdot) \in \cU[t, T]$ such that
\begin{eqnarray}\label{eq:3.3}
\cJ (t, \xi; {\bar u}(\cdot) ) = \essinf_{u (\cdot) \in
{\cU[t, T]}} \cJ (t,\xi; u (\cdot) ).
\end{eqnarray}

A control $\bar{u}(\cdot)\in\cU[t, T]$ satisfying \eqref{eq:3.3} is called an \textit{optimal control} for {\bf Problem (OP)}$_{t,\xi}$. The corresponding state process $\overline{X}(\cdot)$ is called the \textit{optimal state process}. The pair $(\overline{X}(\cdot),\bar{u}(\cdot))$ is called an \textit{optimal pair}.

Clearly, when $t=0$ and $\xi=X_0$, {\bf Problem (OP)$_{0,X_0}$}  reduces to our original {\bf Problem (OP)}.

For each pair $(t,\xi) \in [0,T] \times L^2_{\mathcal{F}_t}(\Omega; H)$, we define the \textit{value mapping}:\vspace{-2mm}
\begin{eqnarray}\label{eq-value1}
\mathbb V(t, \xi) \triangleq \mathop{\textrm{ess}\inf}_{u (\cdot)
\in \cU[t, T]} \cJ (t,\xi; u(\cdot) ) .
\end{eqnarray}
By Lemma \ref{lemma1} and noting that the cost functional $\mathcal{J}(t,\xi;u(\cdot))$ is $\mathcal{F}_t$-measurable for each $u(\cdot) \in \mathcal{U}[t,T]$, for any $(t, \xi) \in [0,T]\times  L^2_{\cF_t}(\Omega; H)$,
$\mathbb V(t, \xi)$ is an $\cF_t$-measurable random variable.

With the help of Lemmas \ref{lemma_esti_sde} and \ref{lem_bsde},   we immediately get the following result.
\begin{lemma}\label{lem:3.6}
Let Conditions {\bf (S1)}--{\bf (S2)} hold. For $t \in [0,T)$, $u(\cdot), \tilde u(\cdot)\in \cU[t, T]$, and $\xi, \tilde \xi \in L^2_{\cF_t}(\Omega; H)$, \vspace{-2mm} 
\begin{equation*} 
|\cJ(t, \xi; u)|\leq \cC(1+|\xi|_H^2),
\end{equation*}
\begin{eqnarray*}
\begin{array}{ll}\ds
|\cJ(t,\xi; u)-\cJ(t,\tilde\xi; \tilde u)|\\
\ns\ds \le \cC \bigg\{\big(1+|\xi|_H+|\tilde
\xi|_H\big)|\xi-\tilde\xi|_H+\mathbb
E\bigg[\int_t^T\big(1+|u(s)|_{U}+|\tilde u(s)|_{U}\big)|u(s)-\tilde u(s)|_{U}ds\bigg| \cF_t\bigg]\bigg\},
\end{array}
\end{eqnarray*}
\begin{equation*}
|\mathbb V(\tau, \xi)|\leq \cC\big(1+|\xi|_H^2\big),\vspace{-2mm}
\end{equation*}
and\vspace{-2mm}
\begin{eqnarray}\label{eq:3.16}
|\mathbb V(\tau,\xi)-\mathbb V(\tau,\tilde \xi)|\leq
\cC\big(1+|\xi|_H+|\tilde \xi|_H\big)\big|\xi-\tilde\xi\big|_H.
\end{eqnarray}
\end{lemma}

\subsection{Dynamic Programming Principle}

In this subsection, we establish the dynamic programming principle for {\bf Problem (OP)}.

Given initial data $(t,\xi) \in [0,T] \times L^2_{\mathcal{F}_t}(\Omega;H)$,  a control $u(\cdot) \in \mathcal{U}[t,r]$, and a terminal condition $\eta \in L^2_{\mathcal{F}_r}(\Omega;\mathbb{R})$, we define the operator \vspace{-4mm}
$$
G^{t,\xi;u}_{s,r}[\eta]\=\widetilde Y (s), \quad s\in [t,r],\vspace{-2mm}
$$
where $(\widetilde X,\widetilde Y,\widetilde Z)$ is the solution to
the following forward-backward system:\vspace{-2mm}
\begin{equation}\label{3.28-eq5}
\left\{
\begin{array}{ll}\ds
d\wt X(s) = \big(A\wt{X}(s) + a(s,\wt X(s),u(s))\big)ds+b(s,\wt X(s),u(s))d W(s) & \mbox{in }(t,r], \\
\ns\ds d\wt Y(s)= -f(s,\wt X (s),u(s))ds+\wt Z(s)dW(s) & \mbox{in }[t,r),\\
\ns\ds \wt X(t) =  \xi, \quad \wt Y(r)=\eta.
\end{array}
\right.\vspace{-2mm}
\end{equation}

We have the following result.

\begin{theorem}\label{thm_dpp_1}
Under Assumptions {\bf (S1)}--{\bf (S2)}, the value mapping $\mathbb{V}$
satisfies the following dynamic programming principle 
for any $t \in [0,T]$, $r \in [t,T]$, and $\xi \in
L^2_{\cF_t}(\Omega;H)$:\vspace{-2mm}
\begin{eqnarray}\label{eq-value2}
\mathbb V(t,\xi)= \essinf_{u(\cdot)\in\cU[t, r]}
G^{t,\xi;u}_{t,r}[\mathbb V(r,X(r;t,\xi,u(\cdot)))].
\end{eqnarray}
\end{theorem}
\begin{remark}
As the classical Bellman's dynamic programming principle, Theorem \ref{thm_dpp_1} reveals that the optimal cost from time $t$ can be obtained by considering: first, the cost of running the system with any admissible control up to an intermediate time $r$; and second, the optimal cost from the resulting state at time $r$. The essential infimum over all admissible controls on $[t,r]$ captures this two-stage optimization.
\end{remark}

Before proceeding with the proof of Theorem \ref{thm_dpp_1}, we establish several key technical results.

\begin{lemma}\label{lemma2}
Under Assumptions {\bf (S1)}--{\bf (S2)}, for any initial data $(t, \xi) \in [0,T] \times L^2_{\mathcal{F}_t}(\Omega;H)$, the set\vspace{-2mm}
\begin{equation*}
\big\{ \mathcal{J}(t, \xi; u(\cdot)) \big| u(\cdot)\in \mathcal{U}[t, T]\big\}\vspace{-2mm}
\end{equation*}
is closed under pairwise minimization. That is, for any two admissible controls, there exists another admissible control whose cost equals the pointwise minimum of the original costs.
\end{lemma}
\begin{proof} 
Consider any pair of admissible controls $u_1(\cdot), u_2(\cdot)\in \mathcal{U}[t, T]$. Define the measurable set\vspace{-2mm}
$$
\Upsilon\triangleq \big\{\omega\in\Omega \big|\cJ(t, \xi;
u_1(\cdot)) \leq \cJ(t, \xi; u_2(\cdot))\big\}.\vspace{-2mm}
$$
Since both $\mathcal{J}(t, \xi; u_1(\cdot))$ and $\mathcal{J}(t, \xi; u_2(\cdot))$ are $\mathcal{F}_t$-measurable, we have $\Upsilon\in \mathcal{F}_t$.

Now construct the switching control\vspace{-2mm}
\begin{equation*}
v(\cdot) \triangleq u_1(\cdot)\chi_{\Upsilon} + u_2(\cdot)\chi_{\Omega\setminus\Upsilon},\vspace{-2mm}
\end{equation*}
which remains admissible in $\mathcal{U}[t, T]$. By the uniqueness of solutions to BSEE \eqref{bsystem1}, we obtain that\vspace{-2mm}
\begin{eqnarray}
\begin{array}{ll}\ds
\cJ(t, \xi; v(\cdot))\3n&\ds= \cJ(t, \xi;
u_1(\cdot)\chi_\Upsilon+u_2(\cdot)\chi_{\Om\setminus\Upsilon})
\\ \ns&\ds=
\cJ(t, \xi; u_1(\cdot))\chi_\Upsilon+\cJ(t, \xi;
u_2(\cdot))\chi_{\Om\setminus\Upsilon}
\\ \ns&\ds= \cJ(t, \xi; u_1(\cdot))\wedge \cJ(t, \xi; u_2(\cdot)).
\end{array}\vspace{-1mm}
\end{eqnarray}
This establishes that the cost functional values are closed under pairwise minimization. 
\end{proof} 

As an immediate corollary of Corollary \ref{cor1} and Lemma
\ref{lemma2}, we have the following result.
\begin{corollary}\label{cor2}
Under Assumptions {\bf(S1)}--{\bf (S2)}, the following properties hold:

1. Approximation by Admissible Controls: There exists a sequence of controls $\{u_k(\cdot)\}_{k=1}^\infty \subset \mathcal{U}[t,T]$ such that:
\begin{itemize}
\item The corresponding costs $\{\mathcal{J}(\tau,\xi; u_k(\cdot))\}_{k=1}^\infty$ form a non-increasing sequence.
\item The sequence converges pointwise to the value function:\vspace{-1mm}
\begin{equation}\label{eq:3.5}
\lim_{k \to \infty} \mathcal{J}(\tau,\xi; u_k(\cdot))(\omega) = \mathbb{V}(\tau, \xi)(\omega), \quad \mathbb{P}\text{-a.s.}\vspace{-1mm}
\end{equation}
\end{itemize}

2. Interchangeability of Conditional Expectation: For any sub-$\sigma$-algebra $\mathcal{G} \subset \mathcal{F}_t$, the conditional expectation commutes with the essential infimum:\vspace{-1mm}
\begin{equation}\label{eq:3.6}
\mathbb{E}\big[\mathbb{V}(t, \xi)|\mathcal{G}\big] = \essinf_{u(\cdot)\in\mathcal{U}[t,T]} \mathbb{E}\big[\mathcal{J}(t,\xi; u(\cdot))|\mathcal{G}\big], \quad \mathbb{P}\text{-a.s.}\vspace{-2mm}
\end{equation} 
\end{corollary}

The following lemma guarantees the existence of arbitrarily approximate optimal controls.

\begin{lemma}[Existence of $\varepsilon$-Optimal Controls]\label{lem_approximate_optimal}
For any initial data $(t,\xi) \in [0,T]\times L^2_{\mathcal{F}_t}(\Omega;H)$ and any tolerance $\varepsilon>0$, there exists a  control $u^{\varepsilon}(\cdot) \in \mathcal{U}[t,T]$ satisfying the $\varepsilon$-optimality condition:\vspace{-1mm}
\begin{equation*}
\mathcal{J}(t,\xi;u^{\varepsilon}(\cdot)) \leq \mathbb{V}(t,\xi) + \varepsilon, \quad \mathbb{P}\text{-a.s.}\vspace{-1mm}
\end{equation*}
\end{lemma}

\begin{proof} By Corollary \ref{cor2}, there exists a non-increasing sequence of controls $\{u_k(\cdot)\}_{k=1}^{\infty} \subset \mathcal{U}[t,T]$ such that $\mathcal{J}(t,\xi;u_k(\cdot))$ converges almost surely to $\mathbb{V}(t,\xi)$.

Define the measurable sets:\vspace{-1mm}
$$
\Upsilon_0=\emptyset,\q \Upsilon_k\=\big\{\om\in\Om\big|
\cJ(t,\xi;u_k(\cdot))\le \mathbb V(t,\xi)+\varepsilon\big\},\;
k\in\dbN.\vspace{-1mm}
$$
These sets form an increasing sequence with $\bigcup_{k=1}^{\infty} \Upsilon_k = \Omega$.

Construct the $\varepsilon$-optimal control by patching:\vspace{-1mm}
$$
u^{\varepsilon}(\cdot)= \sum_{k=1}^{\infty} u_{k}(\cdot) \chi_{\Upsilon_{k}
\setminus \Upsilon_{k-1}}.\vspace{-1mm}
$$
From the uniqueness for the solution to the BSEE \eqref{bsystem1},
we see that\vspace{-1mm}
$$
\cJ(t,\xi;u^{\varepsilon}(\cdot))\chi_{\Upsilon_{k}
\setminus \Upsilon_{k-1}}=\cJ(t,\xi;u_k(\cdot))\chi_{\Upsilon_{k}
\setminus \Upsilon_{k-1}}\le
\left( \mathbb V(t,\xi)+\varepsilon\right)\chi_{\Upsilon_{k}
\setminus \Upsilon_{k-1}},\q k \in\dbN.\vspace{-1mm}
$$
Summing over all $k$ establishes the desired inequality almost surely.
\end{proof}

\ss

\begin{proof}[Proof of Theorem \ref{thm_dpp_1}]
The proof consists of two main inequalities establishing the dynamic programming principle.

\textbf{Part 1: First Inequality}

By the uniqueness of solutions to the forward-backward system \eqref{3.28-eq5}, for any $t\in [0,T] ,\xi\in L_{\mathcal{F}_t}^2(\Omega;H)$, $u(\cdot)\in\mathcal{U}[t,T]$, and $r\in[t,T]$, we have the decomposition property:\vspace{-1mm}
\begin{equation*}
G^{t,\xi;u}_{s,T}[h(X(T;t,\xi,u))] = G^{t,\xi;u}_{s,r}[Y(r;r,X(r;t,\xi,u),u)], \quad t\leq s\leq r.\vspace{-1mm}
\end{equation*}
This leads to the first key inequality:\vspace{-1mm}
\begin{equation}\label{3.28-eq8}
\begin{aligned}
\mathbb{V}(t,\xi) &= \essinf_{u\in\mathcal{U}[t,T]} G^{t,\xi;u}_{t,T}[h(X(T;t,\xi,u))] \\
&= \essinf_{u\in\mathcal{U}[t,r]} G^{t,\xi;u}_{t,r}[Y(r;r,X(r;t,\xi,u),u)] \\
&\geq \essinf_{u\in\mathcal{U}[t,r]} G^{t,\xi;u}_{t,r}[\mathbb{V}(r,X(r;t,\xi,u))].
\end{aligned}\vspace{-1mm}
\end{equation}

\textbf{Part 2: Second Inequality}

For any $\varepsilon>0$ and $u(\cdot)\in\mathcal{U}[t,r]$, Lemma \ref{lem_approximate_optimal} guarantees the existence of $u^\varepsilon(\cdot)\in\mathcal{U}[r,T]$ satisfying:\vspace{-1mm}
\begin{equation*}
\mathbb{V}(r,X(r;t,\xi,u)) \geq Y(r;r,X(r;t,\xi,u),u^\varepsilon) - \varepsilon, \quad \mathbb{P}\text{-a.s.}\vspace{-1mm}
\end{equation*}
Construct the concatenated control:\vspace{-1mm}
\begin{equation*}
\tilde{u}(s) = \begin{cases}
u(s), & t\leq s\leq r, \\
u^\varepsilon(s), & r\leq s\leq T.
\end{cases}\vspace{-1mm}
\end{equation*}
This yields the second key inequality:\vspace{-1mm}
\begin{equation}\label{3.28-eq7}
\begin{aligned}
\mathbb{V}(t,\xi) &\leq G^{t,\xi;\tilde{u}}_{t,T}[h(X(T;t,\xi,\tilde{u}))] = G^{t,\xi;u}_{t,r}[Y(r;r,X(r;t,\xi,u),u^\varepsilon)] \\
&\leq G^{t,\xi;u}_{t,r}[\mathbb{V}(r,X(r;t,\xi,u))+\varepsilon] \leq G^{t,\xi;u}_{t,r}[\mathbb{V}(r,X(r;t,\xi,u))] + \cC\varepsilon.
\end{aligned}\vspace{-1mm}
\end{equation}
The final inequality follows from the well-posedness of BSEE \eqref{3.28-eq5}. Since $\varepsilon>0$ was arbitrary, combining both parts establishes the dynamic programming principle.
\end{proof}

\ms

The value mapping $\mathbb{V}(t,\cdot)$ defines a mapping from the space $L^2_{\mathcal{F}_t}(\Omega; H)$ to $L^2_{\mathcal{F}_t}(\Omega; \mathbb{R})$. We can associate with $\mathbb{V}$ a random field $V:[0,T]\times\Omega\times H\to \mathbb{R}$ through the identification:\vspace{-1mm}
\begin{equation}\label{eq:3.13}
V(t,\omega,x) \triangleq \mathbb{V}(t,x)(\omega), \quad \text{for } (t,x)\in [0,T]\times H.\vspace{-1mm}
\end{equation}

From Lemma \ref{lem:3.6}, we know that $V$ is continuous with
respect to $x\in H$. Hence, for $(t,\xi)\in [0, T]\times
L^2_{\cF_t}(\Omega;H)$ and a.e. $\omega\in \Omega$,
$V(t,\omega,\xi(\omega))$ is well defined. 

The relationship between the abstract value mapping and its pointwise representation is established in the following proposition:
\begin{proposition}[Consistency of Value Representations]\label{prop3}
For any initial data $(t,\xi)\in [0,T]\times L^2_{\mathcal{F}_t}(\Omega;H)$, the random field $V$ satisfies:
\begin{equation}\label{eq:3.13-1}
V(t,\omega,\xi(\omega)) = \mathbb{V}(t,\xi)(\omega), \quad \mathbb{P}\text{-a.s.}
\end{equation}
\end{proposition}

By Proposition \ref{prop3}, we know that the pointwise evaluation of $V$ captures the essential infimum property of $\mathbb{V}$. For each $(t,x) \in [0,T] \times H$, the random field $V(t,\omega,x)$ is called the \textit{value function} of \textbf{Problem (OP)$_{t,x}$}, representing the minimal achievable cost when starting from state $x$ at time $t$. Throughout this work, we will typically suppress the $\omega$ dependence in our notation unless explicit emphasis on the random parameter is required.

To prove Proposition \ref{prop3}, we first give the following
result.
\begin{lemma}[Decomposition Property of Solutions]\label{lm3.1}
For any $\xi\in L^2_{\mathcal{F}_t}(\Omega;H)$, control processes $\{u^j\}_{j=1}^N\subset \mathcal{U}[t,T]$, and $\mathcal{F}_t$-measurable partition $\{\Omega_i\}_{i=1}^N$ of $\Omega$ (i.e., $\bigcup_{i=1}^N \Omega_i = \Omega$ with $\Omega_i\cap \Omega_j=\emptyset$ for $i\neq j$), the solutions to systems \eqref{system2} and \eqref{bsystem1} satisfy the decomposition property:\vspace{-1mm}
\begin{equation}\label{lm3.1-eq1}
\begin{cases}\ds
X\bigg(\cdot;t,\xi,\sum\limits_{j=1}^N\chi_{\Omega_j}u^j\bigg) = \sum\limits_{j=1}^N\chi_{\Omega_j}X(\cdot;t,\xi,u^j),\\
\ns\ds Y\bigg(\cdot;t,\xi,\sum\limits_{j=1}^N \chi_{\Omega_j}u^j\bigg) = \sum\limits_{j=1}^N\chi_{\Omega_j} Y(\cdot;t,\xi,u^j),\\
\ns\ds Z\bigg(\cdot;t,\xi,\sum\limits_{j=1}^N\chi_{\Omega_j}u^j\bigg) = \sum\limits_{j=1}^N\chi_{\Omega_j}Z(\cdot;t,\xi,u^j).
\end{cases}\vspace{-1mm}
\end{equation}
\end{lemma}

\begin{proof}
For each $j=1,\ldots,N$, denote the solution triple as\vspace{-1mm}
\begin{equation*}
(X^j(s), Y^j(s), Z^j(s)) \equiv (X(s;t,\xi,u^j), Y(s;t,\xi,u^j), Z(s;t,\xi,u^j)).\vspace{-1mm}
\end{equation*}
These satisfy the forward equation:\vspace{-1mm}
\begin{equation}\label{12.13-eq1}
\begin{aligned}
X^j(s) &= S(s-t)\xi + \int_t^s S(s-r)a(r,X^j(r),u^j(r))dr \\
&\quad + \int_t^s S(s-r)b(r,X^j(r),u^j(r))dW(r), \quad s\in [t,T],
\end{aligned}\vspace{-1mm}
\end{equation}
and the backward equation:\vspace{-1mm}
\begin{equation}\label{12.13-eq2}
Y^j(s) = h(X^j(T)) + \int_s^T f(r,X^j(r),u^j(r))dr - \int_s^T Z^j(r)dW(r), \quad s\in [t,T].\vspace{-1mm}
\end{equation}

Multiplying both \eqref{12.13-eq1} and \eqref{12.13-eq2} by $\chi_{\Omega_j}$ and summing over $j$, we obtain\vspace{-1mm}
\begin{equation*}
\begin{aligned}
\sum_{j=1}^N \chi_{\Omega_j}X^j(s) &= S(s-t)\xi + \int_t^s S(s-r)a\bigg(r,\sum_{j=1}^N\chi_{\Omega_j}X^j(r),\sum_{j=1}^N\chi_{\Omega_j}u^j(r)\bigg)dr \\
&\quad + \int_t^s S(s-r)b\bigg(r,\sum_{j=1}^N\chi_{\Omega_j}X^j(r),\sum_{j=1}^N\chi_{\Omega_j}u^j(r)\bigg)dW(r),
\end{aligned}\vspace{-1mm}
\end{equation*} 
and\vspace{-1mm}
\begin{equation*}
\begin{aligned}
\sum_{j=1}^N \chi_{\Omega_j}Y^j(s) &= h\bigg(\sum_{j=1}^N\chi_{\Omega_j}X^j(T)\bigg) - \int_s^T \sum_{j=1}^N\chi_{\Omega_j}Z^j(r)dW(r) \\
&\quad + \int_s^T f\bigg(r,\sum_{j=1}^N\chi_{\Omega_j}X^j(r),\sum_{j=1}^N\chi_{\Omega_j}u^j(r)\bigg)dr.
\end{aligned}\vspace{-1mm}
\end{equation*}
The result follows from the uniqueness of solutions to \eqref{system2} and \eqref{bsystem1}, combined with the continuity of the coefficients.
\end{proof}

\ss

\begin{proof}[Proof of Proposition \ref{prop3}]
The proof proceeds in two steps, first establishing the result for simple random variables and then extending to general ones via approximation.

\ss

\textbf{Step 1: In this step, we prove \eqref{eq:3.13-1} for simple random variables}.

Consider $\xi = \sum_{i=1}^N \chi_{\Omega_i}x_i$, where $\{\Omega_i\}_{i=1}^N$ forms an $\mathcal{F}_t$-measurable partition of $\Omega$. For each $1\leq i \leq N$, select a sequence $\{u_{ij}\}_{j=1}^\infty \subset \mathcal{U}[t,T]$ such that\vspace{-1mm}
\begin{equation}\label{3.28-eq9}
\lim_{j\to\infty} \mathcal{J}(t,x_i;u_{ij}) = \mathbb{V}(t,x_i) = V(t,x_i).\vspace{-1mm}
\end{equation}

Applying Lemma \ref{lm3.1}, we obtain the upper bound:\vspace{-1mm}
\begin{equation}\label{3.28-eq10}
\mathbb{V}(t,\xi) \leq \mathcal{J}\left(t,\sum_{i=1}^N \chi_{\Omega_i}x_i; \sum_{i=1}^N \chi_{\Omega_i}u_{ij}\right) = \sum_{i=1}^N \mathcal{J}(t,x_i;u_{ij})\chi_{\Omega_i}, \quad \mathbb{P}\text{-a.s.}\vspace{-1mm}
\end{equation}
Taking $j\to\infty$ in \eqref{3.28-eq10} yields\vspace{-1mm}
\begin{equation}\label{3.28-eq11}
\mathbb{V}(t,\xi)(\omega) \leq \sum_{i=1}^N V(t,x_i)\chi_{\Omega_i}(\omega) = V(t,\omega,\xi(\omega)), \quad \mathbb{P}\text{-a.s.}\vspace{-1mm}
\end{equation}
For the lower bound, observe that for any $u\in \mathcal{U}[t,T]$:\vspace{-1mm}
$$
\begin{aligned}
\mathcal{J}(t,\xi;u) &= \sum_{i=1}^N \mathcal{J}(t,\xi;u)\chi_{\Omega_i}  = \sum_{i=1}^N \mathcal{J}(t,x_i;\chi_{\Omega_i}u)\chi_{\Omega_i} \\
&\geq \sum_{i=1}^N V(t,x_i)\chi_{\Omega_i} = V(t,\omega,\xi(\omega)),
\end{aligned}\vspace{-1mm}
$$
which implies\vspace{-1mm}
\begin{equation}\label{3.28-eq12}
\mathbb{V}(t,\xi)(\omega) = \essinf_{u \in \mathcal{U}[t,T]} \mathcal{J}(t,\xi;u) \geq V(t,\omega,\xi(\omega)), \quad \mathbb{P}\text{-a.s.}\vspace{-1mm}
\end{equation}

Combining \eqref{3.28-eq11} and \eqref{3.28-eq12} establishes the result for simple random variables:\vspace{-1mm}
\begin{equation}\label{3.28-eq13}
V(t,\omega,\xi(\omega)) = \mathbb{V}(t,\xi)(\omega), \quad \mathbb{P}\text{-a.s.}\vspace{-1mm}
\end{equation}

\ss

\textbf{Step 2: In this step, we prove \eqref{eq:3.13-1} for general random variables}.

For arbitrary $\xi\in L^2_{\mathcal{F}_t}(\Omega;H)$, construct a sequence of simple random variables $\{\xi_N\}_{N=1}^\infty$ with $\xi_N = \sum_{i=1}^N \chi_{\Omega_i}x_i$ converging to $\xi$ in $L^2_{\mathcal{F}_t}(\Omega;H)$. 

By using the local Lipschitz estimate in Lemma \ref{lem:3.6},
%
we obtain\vspace{-1mm} 
\begin{equation}\label{3.28-eq17}
\lim_{N\to\infty} V(t,\omega,\xi_N(\omega)) = V(t,\omega,\xi(\omega)), \quad \mathbb{P}\text{-a.s.}\vspace{-1mm}
\end{equation}

Similarly, from \eqref{eq:3.16} we obtain that\vspace{-1mm}
\begin{equation}\label{3.28-eq18}
\lim_{N\to\infty} \mathbb{V}(t,\xi_N)(\omega) = \mathbb{V}(t,\xi)(\omega), \quad \mathbb{P}\text{-a.s.}\vspace{-1mm}
\end{equation}
The conclusion follows by combining \eqref{3.28-eq13}, \eqref{3.28-eq17}  and \eqref{3.28-eq18}.
\end{proof}

The following result establishes the fundamental Dynamic Programming Principle (DPP) for {\bf Problem (OP)}, which follows directly from Theorem
\ref{thm_dpp_1} and Proposition \ref{prop3}.

\begin{theorem} \label{thm-sdy1}
Let Assumptions  {\bf (S1)}--{\bf (S2)}
hold. Then the value function
$V$ satisfies the following
dynamic programming principle: for any $(t,x)\in [0,T-\delta]\times H$,
\begin{eqnarray}\label{eq:3.14}
V(t,x)= \essinf_{u(\cdot)\in\cU[t, t+\delta]}  G^{t,x;u}_{t,t+\delta}[ V(t+\delta,X(t+\delta;t,x,u))].
\end{eqnarray}
\end{theorem}

\subsection{Stochastic Hamilton-Jacobi-Bellman equation}\label{sHJB}

We now derive the stochastic HJB equation associated with \textbf{Problem (OP)$_{t,\xi}$} through the dynamic programming approach. First, we introduce the generalized Hamiltonian functional:
\begin{definition}[Stochastic Hamiltonian]
The stochastic Hamiltonian $\mathbb{H}: [0,T]\times \Omega\times H \times U \times H \times \cL_2^0 \times \mathcal{L}(H)\to\mathbb{R}$ is defined as\vspace{-1mm}
\begin{equation}\label{eq:4.1}
\mathbb{H}(t,x,u,p,q,B) \triangleq f(t,x,u) + \langle p,a(t,x,u)\rangle_H + \langle q,b(t,x,u)\rangle_{\mathcal{L}_2^0} + \frac{1}{2}\langle B b(t,x,u), b(t,x,u)\rangle_{\mathcal{L}_2^0}.
\end{equation}
\end{definition}

Consider the following backward stochastic  evolution equation:\vspace{-1mm}
\begin{eqnarray}\label{eqHJB}
\left\{
\begin{array}{ll}\ds  d V + \lan A^* V_x,  x
\ran_H + \inf_{u\in U} \dbH(t,x,u, V_x, \Phi_x, V_{xx})
=\Phi dW(t) &\mbox{ in } [0,T)\times H,\\
\ns\ds V(T,x)=   h(x).
\end{array}
\right.
\end{eqnarray}
\begin{remark}
In the case that all coefficients in the state equation \eqref{system1} and cost functional \eqref{bsystem1} are deterministic, the value function undergoes a significant simplification. Specifically, the random field component $V(t,x)$ reduces to a deterministic mapping depending solely on the time-state pair $(t,x)$. This deterministic setting transforms the original stochastic HJB equation into a classical second-order nonlinear partial differential equation of Hamilton-Jacobi-Bellman type. The resulting deterministic HJB equation takes the form:\vspace{-1mm}
\begin{eqnarray}\label{eq:b16}
\left\{
\begin{array}{ll}\ds  V_t = \lan A^* V_x,  x
\ran_H + \inf_{u\in U} \mathbb{H}(t,x,u, V_x,0, V_{xx})  &\mbox{ in } [0,T)\times H,\\
\ns\ds V(T,x)= h(x).
\end{array}
\right.
\end{eqnarray}
where $V_x$ and $V_{xx}$ denote respectively the Fréchet derivative and Hessian operator of the value function. This deterministic version maintains the essential structure of the stochastic equation while eliminating the stochastic differential term, reflecting the absence of randomness in the system coefficients. The equation exhibits the characteristic features of HJB equations in infinite dimensions, combining a generator term involving the adjoint operator $A^*$ with a nonlinear Hamiltonian minimization. The terminal condition preserves the original cost function $h(x)$, completing the well-posed boundary value problem formulation.
\end{remark}

Now we give the definition of the {\it classical solution} to the
stochastic HJB equation  \eqref{eqHJB}.

\begin{definition}[Classical Solution to Stochastic HJB Equation]\label{defn:4.1}
A pair of random fields $(V, \Phi)$ is said to be a classical solution to the stochastic HJB equation \eqref{eqHJB} if the following conditions are satisfied:

\begin{enumerate}
\item $V\in \mathbb{C}_{\mathbb{F}}^{0,2}([0,T]\times H)$, and the adjoint composition $A^*V_x\in \mathbb{C}_{\mathbb{F}}^{0,0}([0,T] \times H,H)$.

\item $\Phi\in \mathbb{C}_{\mathbb{F}}^{0,1}([0,T]\times H,\wt{H})$.

\item The pair $(V, \Phi)$ satisfies the stochastic HJB equation in the integral form:\vspace{-1mm}
\begin{equation}\label{eq:4.3}
\begin{aligned}
V(t,x) = h(x) &+ \int_t^T \Big( \langle A^* V_x(s,x), x \rangle_H + \inf_{u \in U} \mathbb{H}(s,x,u, V_x(s,x), \Phi_x(s,x), V_{xx}(s,x)) \Big) ds \\
&- \int_t^T \Phi(s,x) dW(s),
\end{aligned}\vspace{-1mm}
\end{equation}
for all $(t,x) \in [0,T] \times H$, $\mathbb{P}$-almost surely.
\end{enumerate} 
%
%
\end{definition}

Assume that \textbf{Problem (OP)$_{t,\xi}$} admits an optimal control for all initial pairs $(t,\xi)\in [0,T]\times L^2_{\mathcal{F}_t}(\Omega;H)$. For a fixed $(t,x)\in [0,T]\times H$, let $(\overline{X}(\cdot),\overline{u}(\cdot))$ denote the corresponding optimal pair. Theorem \ref{thm-sdy1} yields the representation:\vspace{-1mm}
\begin{equation}\label{eq:value_rep}
V(t,x) = \mathbb{E}\left[\int_t^T f(r,\overline{X}(r),\overline{u}(r))dr + h(\overline{X}(T)) \bigg| \mathcal{F}_t\right], \quad \mathbb{P}\text{-a.s.}\vspace{-1mm}
\end{equation}
Define the conditional expectation process:\vspace{-1mm}
\begin{equation}\label{eq:cond_exp_process}
m(s,x) \triangleq \mathbb{E}\left[\int_t^T f(r,\overline{X}(r),\overline{u}(r))dr + h(\overline{X}(T)) \bigg| \mathcal{F}_s\right], \quad s\in [t,T].\vspace{-1mm}
\end{equation}
Under Assumption \textbf{(S2)}, we establish the following moment estimates:\vspace{-1mm}
$$
\begin{aligned}
&\int_t^T \mathbb{E}\left|\mathbb{E}\left[\int_t^T f(r,\overline{X}(r),\overline{u}(r))dr + h(\overline{X}(T)) \bigg| \mathcal{F}_s\right]\right|^2 ds \\
&\leq T \mathbb{E}\left[\int_t^T |f(r,\overline{X}(r),\overline{u}(r))|^2dr + |h(\overline{X}(T))|^2\right] \\
&\leq \cC\mathbb{E}\left[\int_t^T \big(1+|\overline{X}(r)|_H^2\big)dr + 1 + |\overline{X}(T)|_H^2\right] \\
&\leq \cC\big(1+|x|_{H}^2\big),
\end{aligned}\vspace{-1mm}
$$
which verifies that $m(\cdot,x) \in L_{\mathbb{F}}^{2}(t,T;\mathbb{R})$. By the martingale representation theorem (see \cite[Corollary 2.145]{Lu2021} for example), there exists a unique process $K(\cdot,x)\in L^{2}(t,T;\widetilde{H})$ such that \vspace{-1mm}
\begin{equation}\label{eq:martingale_rep}
m(s,x) = m(t,x) + \int_t^s K(r,x)dW(r), \quad s\in [t,T].\vspace{-1mm}
\end{equation}
Consequently, we obtain the following decomposition for the value function:\vspace{-1mm}
\begin{equation}\label{eq:value_decomp}
V(t,x) = m(t,x) = \int_t^T f(r,\overline{X}(r),\overline{u}(r))dr + h(\overline{X}(T)) - \int_t^T K(r,x)dW(r).\vspace{-1mm}
\end{equation} 
\begin{proposition} \label{prop:b1}
Let Assumptions {\bf (S1)}--{\bf (S3)} hold, and suppose the value function $V(\cdot,\cdot)$ of \textbf{ Problem (OP)$_{t,\xi}$}  admits the semimartingale decomposition:\vspace{-1mm}
\begin{equation}\label{eq611}
V(t,x) = h(x) + \int_t^T \Gamma(s,x) ds - \int_t^T \Phi(s,x) dW(s), \quad (t,x) \in [0,T] \times H,\vspace{-1mm}
\end{equation}
where $(V, \Phi)$ satisfies the regularity conditions (1)--(2) of Definition \ref{defn:4.1}.

Assume further the following regularity conditions hold:
\begin{enumerate}
\item[(a)] $\Gamma\in \mathbb{C}_{\mathbb{F}}^{0,1}([0,T]\times H)$;

\item[(b)] There exists a process $L \in L^4_{\mathbb{F}}(\Omega,L^2(0,T;\mathbb{R}))$ such that for almost all $(t,\omega)$:\vspace{-1mm}
$$
\begin{aligned}
&|\Gamma(t,x)| + |\Phi(t,x)|_{\widetilde{H}}+|\Gamma_x(t,x)|_H+|\Phi_x(t,x)|_{\cL_2^0} \leq L(t)(1 + |x|_H^2), \\
&|V(t,x)| + |V_x(t,x)|_H+|A^*V_x(t,x)|_H+|V_{xx}(t,x)|_{\cL(H)}  \le L(t)(1+|x|_H^2),\\
\end{aligned}
$$

%
\end{enumerate}
Then, if an optimal control $\overline{u}$ exists for each $(t,x)$, the pair $(V, \Phi)$ constitutes a classical solution to the stochastic HJB equation \eqref{eqHJB}.
\end{proposition} 

\begin{proof}[Proof of Proposition \ref{prop:b1}]
To show that $(V, \Phi)$ is a classical solution to \eqref{eqHJB}, we must verify that almost surely:\vspace{-1mm}
\begin{equation}\label{eq:main_identity}
\Gamma(t,x) = \langle A^*V_x(t,x), x \rangle_H + \inf_{u\in U} \mathbb{H}(t,x,u, V_x(t,x), \Phi_x(t,x), V_{xx}(t,x)).\vspace{-1mm}
\end{equation}
We divide this into two steps.

\ss

\textbf{Step 1:} In this step, we establish a lower bound on the Hamiltonian.

For any control $u(\cdot) \in \mathcal{U}[0,T]$ and $x\in H$, let $X(\cdot) \equiv X(\cdot;t,x,u)$. Applying It\^o-Kunita formula (see Lemma \ref{Ito-Kunita}) to $V(s,X(s))$ yields\vspace{-1mm}
\begin{equation}\label{eq:ito_decomp}
\begin{aligned}
&	V(t+\delta,X(t+\delta)) - V(t,X(t)) \\&=  -\int_t^{t+\delta} \Gamma(s,X(s)) ds + \int_t^{t+\delta} \Phi(s,X(s)) dW(s) \\
& \q + \int_t^{t+\delta} \Big[ \langle A^* V_x(s,X(s)), X(s) \rangle_H + \langle V_x(s,X(s)),a(s,X(s),u(s))\rangle_H \\
&\q + \frac{1}{2} \langle V_{xx}(s,X(s)) b(s,X(s),u(s)), b(s,X(s),u(s))\rangle_{\mathcal{L}_2^0} \Big] ds \\
&\q + \int_t^{t+\delta} b(s,X(s),u(s))^* V_x(s,X(s)) dW(s) + \int_t^{t+\delta} \langle \Phi_x(s,X(s)),b(s,X(s),u(s))\rangle_{\mathcal{L}_2^0} ds.
\end{aligned}\vspace{-1mm}
\end{equation}
Consider the BSEE:\vspace{-3mm}
\begin{equation*}
\begin{cases}
	dY(s) = -f(s,X(s),u(s)) ds + Z(s) dW(s), & s \in [t,t+\delta], \\
	Y(t+\delta) = V(t+\delta,X(t+\delta)).
\end{cases}\vspace{-1mm}
\end{equation*}
The dynamic programming principle gives $V(t,X(t)) \leq Y(t)$. Defining\vspace{-1mm}
$$
\begin{aligned}
F(t,x,u) &\triangleq -\Gamma(t,x) + \langle A^*V_x(t,x), x \rangle_H + \mathbb{H}(t,x,u, V_x(t,x),\Phi_x(t,x), V_{xx}(t,x)), \\
Z'(s) &\triangleq \Phi(s,X(s)) + b(s,X(s),u(s))^* V_x(s,X(s)),
\end{aligned}\vspace{-1mm}
$$
we have\vspace{-1mm}
$$dV(s,X(s))=[F(s,X(s),u(s))-f(s,X(s),u(s))]ds+Z'(s)dW(s),$$
and thus obtain\vspace{-1mm}
\begin{equation}\label{3.28-eq1}
0 \leq Y(t) - V(t,X(t)) = \mathbb{E}\left( \int_t^{t+\delta} F(s,X(s),u(s)) ds \Big| \mathcal{F}_t \right).\vspace{-1mm}
\end{equation}

For constant control $u(\cdot) \equiv \eta \in U$, since $X\in L^2(\Omega,C([t,T],H))$, then the state process $X(s)\rightarrow x$  in $H$ as $s\rightarrow t$ almost surely.
By assumptions on $V,\Gamma$ and $\Phi$, we have $F(\cdot,\cdot,\eta)\in \mathbb{C}_{\mathbb{F}}^{0,0}([0,T]\times H)$.
Noting  
$$
|F(s,X(s),\eta)|\leq \cC L(s)(1+|X(s)|_{H}+|\eta|_{U}),
$$
applying Dominated Convergence Theorem to \eqref{3.28-eq1} yields that
\begin{equation}\label{eq:pointwise_bound}
F(t,x,\eta) \geq 0, \quad \text{a.s. for all } (t,x,\eta) \in [0,T]\times H \times U.
\end{equation}
\ss

\textbf{Step 2}.  In this step, we prove that \vspace{-1mm}
\begin{equation}\label{3.29-eq1}
	\inf_{\eta\in U}  F (t,x,\eta) = 0,  \quad \text{ for all } (t,x)\in [0,T]\times H,\quad \mathbb{P}\mbox{-a.s.}\vspace{-1mm}
\end{equation}
Suppose that \eqref{3.29-eq1} does not
hold. Then there exist  $(t_0,x_0)\in [0,T)\times H$, $\d_0>0$,
$\d_1>0$, $\Omega_0\in \cF_{t_0}$ such that $\dbP(\Omega_0)>\d_1$
and  \vspace{-1mm}
\begin{equation}\label{3.29-eq2}
	\inf_{\eta\in U}  F (t_0,x_0,\eta) >\d_0, \quad \text{for a.e.  }
	\omega\in\Omega_0.\vspace{-1mm}
\end{equation}

Since $\inf_{\eta\in U}  F (\cd,\cd,\eta)$ is continuous, there exist a stopping time $\e_1>0$ and a positive constant $\e_2$ such that\vspace{-1mm}
\begin{equation}\label{3.29-eq3}
	\inf_{\eta\in U}  F (t,x,\eta) >\frac{\d_0}{2}, \quad \text{for a.e.
	} \omega\in\Omega_0,\q \forall (t,x)\in Q_{\e_1,\e_2},\vspace{-1mm}
\end{equation}
where\vspace{-1mm}
$$
Q_{\e_1,\e_2}\deq \big\{(t,x)\in [0,T]\times H\big|\, t\in
[t_0,t_0+\e_1], \; |x-x_0|_H\leq \e_2\big\}.\vspace{-1mm}
$$

For the optimal control $\bar{u}$ at $(t_0,x_0)$,  
from the dynamic programming principle \eqref{eq:3.14}, we see that
the equality holds in \eqref{3.28-eq1} when we replace the 
control $u$ by the optimal control $\bar u$. This, together with
\eqref{eq:pointwise_bound}, implies  that
\begin{equation}\label{eq:optimal_eq}
F(s,X(s;t_0,x_0,\bar{u}),\bar{u}(s)) = 0 \quad \text{a.e. } s \in [t_0,T], \text{ a.s.}
\end{equation}

On the other hand, by the continuity of the solution $X$ to \eqref{system2}, we know that there exists a stopping time $\e_3\in (0,\e_1]$ such that\vspace{-1mm}
$$
\big|X(s;t_0,x_0,\bar u)-x_0\big|_H<\e_2,\q \forall s\in
[t_0,t_0+\e_3], \q \dbP\mbox{-a.s.}\vspace{-1mm}
$$
This, together with \eqref{3.29-eq3}, implies that\vspace{-1mm}
\begin{equation}\label{3.29-eq4}
F(s,X(s;t_0,x_0,\bar u);\bar u(s)) >\frac{\d_0}{2}, \quad  \text{for
a.e.  }  \omega\in\Omega_0,\q \forall s\in [t_0,t_0+ \e_3].\vspace{-1mm}
\end{equation}
This contradicts \eqref{eq:optimal_eq}. Hence, we know that \eqref{3.29-eq1} holds.
Thus, \eqref{eq:main_identity} holds, completing the proof.
\end{proof} 

\vspace{-1mm} 

\section{Relationships between PMP and DPP for \textbf{Problem (OP)}}\label{sec-relation}

\vspace{-1mm}

In this section, we establish the relationships between PMP and DPP for \textbf{Problem (OP)}. For the readers' convenience, we first recall the known Pontryagin type maximum principle for \textbf{Problem (OP)}. Then we consider the case that the value function  enjoys appropriate regularity in Subsection \ref{ssec-smooth}. Then we handle the general case in Subsection \ref{ssec-nonsmooth}.


\subsection{Pontryagin type maximum principle for \textbf{Problem (OP)}} 
\label{ssec-PMP}

For the reader's convenience, we first state the PMP for \textbf{Problem (OP)}. A comprehensive treatment of this result, including detailed proofs and additional applications, can be found in \cite[Chapter 12]{Lu2021}.

We begin by introducing the adjoint equations that are fundamental to the Pontryagin Maximum Principle. The first-order adjoint equation is given by\vspace{-1mm}
\begin{equation}\label{ad-eq1}
\begin{cases}
dp(t)=-A^*p(t)dt-\big(a_x(t,\overline{X}(t),\bar{u}(t))^*p(t)+b_x(t,\overline{X}(t),\bar{u}(t))^*q(t)\\
\ \ \ \ \ \ \ \ \q  -f_x(t,\overline{X}(t),\bar{u}(t))\big)dt+q(t)dW(t), \qq \q \q \qq t\in[0,T),\\
p(T)=-h_x(\overline{X}(T)).
\end{cases}\vspace{-1mm}
\end{equation}
The second-order adjoint equation takes the form:\vspace{-1mm}
\begin{eqnarray}\label{ad-eq2}
\begin{cases}
dP(t)=-\big[ \big(A^*+a_x(t,\overline{X}(t),\bar{u}(t))^*\big)P(t)+ P(t)\big(A+a_x(t,\overline{X}(t),\bar{u}(t))\big) \\
\ \ \ \ \ \ \ \ \q +b_x(t,\overline{X}(t),\bar{u}(t))^*P(t)b _x(t,\overline{X}(t),\bar{u}(t))+b_x(t,\overline{X}(t),\bar{u}(t))^*Q(t) \\
\ \ \ \ \ \ \ \ \ \q  +Q(t)b_x(t,\overline{X}(t),\bar{u}(t))\!+\!\mathscr{H}_{xx}(t,\overline{X}(t),\bar{u}(t),p(t),q(t))\big]dt\! +Q(t)dW(t), \ \  t\in [0,T),\\
P(T)=-h_{xx}(\overline{X}(T)),
\end{cases}
\end{eqnarray}
where the Hamiltonian $\mathscr{H}$ is defined as\vspace{-1mm}
$$
\begin{array}{ll}\ds
\mathscr{H}(t,x,u,p,q)=\langle p,a(t,x,u)\rangle_H + \langle q,b(t,x,u)\rangle_{\mathcal{L}_2^0}-f(t,x,u),\\
\ns\ds \qq\qq\qq\qq\q (t,x,u,p,q)\in[0,T]\times H\times U\times
H\times \mathcal{L}_2^0.
\end{array}\vspace{-1mm}
$$

Equation \eqref{ad-eq1} represents an $H$-valued BSEE that admits a unique mild solution pair $(p,q) \in L^2_{\mathbb{F}}(\Omega;C([0,T];H)) \times L^2_{\mathbb{F}}(0,T;\mathcal{L}_2^0)$, as established in \cite[Section 4.2]{Lu2021}. The quadruple $(\overline{X}(\cdot), \overline{u}(\cdot), p(\cdot), q(\cdot))$ forms what we call an \emph{optimal 4-tuple} for \textbf{Problem (OP)}.

In the finite-dimensional case where $H = \mathbb{R}^n$, equation \eqref{ad-eq2} can be interpreted as an $\mathbb{R}^{n^2}$-valued backward stochastic evolution equation (BSEE), whose well-posedness follows directly from the standard Hilbert space BSEE theory (see \cite[Section 4.2]{Lu2021}). However, in the infinite-dimensional setting ($\dim H = \infty$), fundamental analytical challenges arise.

No existing stochastic integration or evolution equation theory in general Banach spaces can establish the well-posedness of \eqref{ad-eq2} in the conventional sense. This necessitates the use of a generalized solution concept--the relaxed transposition solution--for \eqref{ad-eq2}, which we briefly recall next.

These limitations necessitate the introduction of an alternative solution concept--the relaxed transposition solution--for the second-order adjoint equation \eqref{ad-eq2}, which we shall recall in what follows.

For notational simplicity, we introduce the following abbreviations:\vspace{-1mm}
$$
\begin{cases}
J(t)=a_x(t,\overline{X}(t),\bar{u}(t)),\q K(t)=b_x(t,\overline{X}(t),\bar{u}(t)), \\
\ns\ds F(t)= -\mathscr{H}_{xx}(t,\overline{X}(t),\bar{u}(t),p(t),q(t)),\q P_T=-h_{xx}(\overline{X}(T)).
\end{cases}\vspace{-1mm}
$$

We define the space of pointwise-defined operators:\vspace{-1mm}
\begin{equation*}
\begin{array}{ll}\ds
\mathcal{L}_{pd}(L_{\mathbb{F}}^2(0,T;L^{4}(\Omega,H));L^2_{\mathbb{F}}(0,T;L^{\frac{4}{3}}(\Omega,H)))\\
\ns\ds \deq \Big\{L\!\in\! \cL\big(L_{\mathbb{F}}^2(0,T;L^{4}(\Omega,H));L^2_{\mathbb{F}}(0,T;L^{\frac{4}{3}}(\Omega,H))\big) \big| \mbox{for }\ae (t,\omega)\in [0,T]\times\Omega, \mbox{there
is }\\
\ns\ds\q \wt L(t,\omega)\!\in\!\mathcal{L}(H)\;  \mbox{such that } \big( L v(\cd)\big)(t,\omega)
= \wt L (t,\omega)v(t,\omega),
\forall\; v(\cd)\in
L_{\mathbb{F}}^2(0,T;L^{4}(\Omega,H))\Big\}.
\end{array}\vspace{-1mm}
\end{equation*}
When no confusion arises, we identify $L\in \mathcal{L}_{pd}(L_{\mathbb{F}}^2(0,T;L^{4}(\Omega,H));L^2_{\mathbb{F}}(0,T; L^{\frac{4}{3}}(\Omega,H)))$ with its pointwise representation $\widetilde{L}(\cdot,\cdot)$.

Define the solution spaces:\vspace{-1mm}
\begin{eqnarray*}
\cP[0,T] \3n& \deq\3n & \big\{P(\cdot,\cdot)\ |\ P(\cdot,\cdot)\in \mathcal{L}_{pd}(L_{\mathbb{F}}^2(0,T;L^{4}(\Omega,H));L^2_{\mathbb{F}}(0,T;L^{\frac{4}{3}}(\Omega,H))),\\
& & \ \ \ P(\cdot,\cdot)\xi \in D_{\mathbb{F}}([t,T];L^{\frac{4}{3}}(\Omega,H))) \ \textup{and} \ |P(\cdot,\cdot)\xi|_{D_{\mathbb{F}}([t,T];L^{\frac{4}{3}}(\Omega,H))}\\
& & \ \ \ \leq \cC |\xi|_{L_{\mathcal{F}_t}^{4}(\Omega;H)} \ \textup{for every}\ t\in [0,T] \ \textup{and} \ \xi \in L_{\mathcal{F}_t}^{4}(\Omega;H)\big\},
\end{eqnarray*}
and\vspace{-4mm}
\begin{eqnarray*}
\mathcal{Q} [0,T]
\3n& \deq\3n &\big\{(Q^{(\cdot)},\widehat{Q}^{(\cdot)})\ |\ Q^{(t)},\widehat{Q}^{(t)}\in \mathcal{L}(\mathcal{H}_t;L_{\mathbb{F}}^2(t,T;L^{\frac{4}{3}}(\Omega;\mathcal{L}_2^0)))\\
& & \hspace{0.8cm} \textup{and}\  Q^{(t)}(0,0,\cdot)^{*}=\widehat{Q}^{(t)}(0,0,\cdot) \ \textup{for any}\  t\in [0,T)\big\}
\end{eqnarray*}
with
\begin{equation*}
\mathcal{H}_t\deq L_{\mathcal{F}_t}^{4}(\Omega;H)\times L_{\mathbb{F}}^2(t,T;L^{4}(\Omega;H))\times L_{\mathbb{F}}^2(t,T;L^{4}(\Omega;\mathcal{L}_2^0)),\ \ \forall \  t\in [0,T).
\end{equation*}

For $j=1,2$ and $t\in [0,T)$, consider the test equation: \vspace{-1mm}
\begin{equation}\label{test-eq1}
\begin{cases}
d\f_j=(A+J)\f_jds+u_jds+K\varphi_jdW(s)+v_jdW(s) &\textup{ in } (t,T],\\
\f_j(t)=\xi_j
\end{cases}
\end{equation}
where $\xi_j \in L_{\mathcal{F}_t}^{4}(\Omega;H)$, $u_j\in L_{\mathbb{F}}^2(t,T;L^{4}(\Omega;H))$, and $v_j \in L_{\mathbb{F}}^2(t,T;L^{4}(\Omega;\mathcal{L}_2^0))$. By standard SEE theory \cite[Section 3.2]{Lu2021}, equation \eqref{test-eq1} admits a unique mild solution $\varphi_j\in C_{\mathbb{F}}([t,T];L^4(\Omega;H))$.

\begin{definition}\label{def2.1} A $3$-tuple $(P(\cdot), Q^{(\cdot)},\widehat{Q}^{(\cd)})\in
\mathcal{P}[0,T]\times \mathcal{Q}[0,T]$ is called a
relaxed transposition solution to the equation \eqref{ad-eq2}  if for any $t\in
[0,T]$, $\xi_j \in L_{\mathcal{F}_t}^{4}(\Omega; H)$, $
u_j(\cdot)\in L_{\mathbb{F}}^2(t,T;L^{4}(\Omega;H))$
and $v_j(\cdot)\in
L_{\mathbb{F}}^2(t,T;L^{4}(\Omega;\mathcal{L}_2^0))$ ($j=1,2$), it holds that\vspace{-1mm}
\begin{eqnarray*}
& & \mathbb{E}\lan P_T\f_1(T),\f_2(T)\ran_H-\mathbb{E}\int_t^T \lan F(s)\f_1(s),\f_2(s)\ran_H ds\\
& & =\mathbb{E}\langle P(t)\xi_1,\xi_2\rangle_H+\mathbb{E}\int_t^T \lan P(s)u_1(s),\f_2(s)\ran_Hds+\mathbb{E}\int_t^T \lan P(s)\f_1(s),u_2(s)\ran_Hds\\
& & \ \ \ +\mathbb{E}\int_t^T \lan P(s)K(s)\f_1(s),v_2(s)\ran_{\mathcal{L}_2^0}ds +\mathbb{E}\int_t^T \lan P(s)v_1(s),K(s)\f_2(s)+v_2(s)\ran_{\mathcal{L}_2^0}ds\\
& & \ \ \ +\mathbb{E}\int_t^T \lan v_1(s),\widehat{Q}^{(t)}(\xi_2,u_2,v_2)(s)\ran_{\mathcal{L}_2^0}ds +\mathbb{E}\int_t^T \lan  Q^{(t)}(\xi_1,u_1,v_1)(s),v_2(s)\ran_{\mathcal{L}_2^0}ds
\end{eqnarray*}
\end{definition}

As an immediate corollary of \cite[Theorem 12.9]{Lu2021}, we have
the following well-posedness result for the equation \eqref{ad-eq2}.

\begin{proposition}\label{prp2.1} The equation \eqref{ad-eq2}
admits a unique relaxed transposition solution $(P(\cdot),
Q^{(\cdot)}, \widehat{Q}^{(\cdot)})$. Furthermore,\vspace{-4mm}
\begin{eqnarray}
& & |P|_{\mathcal{L}(L_{\mathbb{F}}^2(0,T;L^{4}(\Omega,H));L^2_{\mathbb{F}}(0,T;L^{\frac{4}{3}}(\Omega,H)))} + \sup\limits_{t\in[0,T)}|(Q^{(t)},\widehat{Q}^{(t)})|_{ \mathcal{L}(\mathcal{H}_t;L_{\mathbb{F}}^2(t,T;L^{\frac{4}{3}}(\Omega;\mathcal{L}_2^0)))^2}\nonumber\\
& & \leq \cC\big(|F|_{L_{\mathbb{F}}^1(0,T;L^2(\Omega;\mathcal{L}(H)))}+|P_T|_{L_{\mathcal{F}_T}^2(\Omega;\mathcal{L}(H))}\big).
\end{eqnarray}
\end{proposition}
($\overline{X} (\cdot), \bar{u}(\cdot), p(\cdot),
q(\cdot), P(\cdot),
Q^{(\cdot)}, \widehat{Q}^{(\cdot)}$) is called an {\it optimal 7-tuple of} \textbf{Problem (OP)}.

Now we can present the PMP for \textbf{Problem} $\boldsymbol{(S_{x})}$.
\begin{theorem}\label{maximum p2-1}
Suppose that  the assumptions
{\bf (S1)}--{\bf (S3)} hold. Let $(\overline{X} (\cdot), \bar{u}(\cdot), p(\cdot),
q(\cdot),P(\cdot),
Q^{(\cdot)},$ $\widehat{Q}^{(\cdot)})$ be an optimal 7-tuple of \textbf{Problem} $\boldsymbol{(S_{x})}$.
Then,   for  a.e. $(t,\omega)\in [0,T]\times
\Omega$ and for all $\rho \in U$,
\begin{equation*}\label{MP2-eq1-1}
\begin{array}{ll}\ds
\mathscr{H}\big(t,\overline X(t),\bar u(t),p(t),q(t)\big) - \mathscr{H}\big(t,\overline X(t),\rho,p(t),q(t)\big) \\
\ns\ds   - \frac{1}{2}\big\langle
P(t)\big[ b\big(t,\overline X(t),\bar
u(t)\big)-b\big(t,\overline X(t),\rho\big)
\big],  b\big(t,\overline X(t),\bar
u(t)\big)-b\big(t,\overline X(t),\rho\big)
\big\rangle_{\mathcal{L}_2^0} \geq 0.
\end{array}
\end{equation*}
\end{theorem}

\subsection{Relationships between PMP and DPP: smooth case}\label{ssec-smooth}

In this subsection, we provide the relationship between the PMP and DPP under smoothness assumptions on the value function. Our results reveal that the adjoint variables $(p(\cdot), q(\cdot))$ and the value function $V(\cdot,\cdot)$ are fundamentally connected--at least formally.

\begin{theorem}\label{Th5}
	Let {\rm ({\bf S1})--({\bf S3})} hold and fix $x \in H$. Let $(\overline{X}(\cdot), \bar{u}(\cdot), p(\cdot), q(\cdot))$ be an optimal $4$-tuple for \textbf{Problem (OP)}. Assume the value function $V$ and the corresponding stochastic fields $\Gamma$ and $\Phi$ satisfying the same Assumptions as in Proposition $\ref{prop:b1}$,
	%
	%
Then, for almost every $(t,\omega) \in [0,T] \times \Omega$,
	$$\begin{aligned}
		\Gamma(t,\overline{X}(t)) &= \left\langle A^*V_x(t,\overline{X}(t)), \overline{X}(t)\right\rangle_H + \dbH\left(t,\overline{X}(t),\overline{u}(t), V_{x}(t,\overline{X}(t)), \Phi_{x}(t,\overline{X}(t)), V_{xx}(t,\overline{X}(t))\right) \\
		&= \left\langle A^*V_x(t,\overline{X}(t)), \overline{X}(t)\right\rangle_H + \inf_{u \in U} \dbH\left(t,\overline{X}(t), u, V_{x}(t,\overline{X}(t)), \Phi_{x}(t,\overline{X}(t)), V_{xx}(t,\overline{X}(t))\right).
	\end{aligned}$$
	Furthermore, if $V \in \mathbb{C}_{\mathbb{F}}^{0,3}([0, T] \times H),\Phi \in \mathbb{C}_{\mathbb{F}}^{0,2}([0, T] \times H; \wt H)$ and $A^*V_x\in \mathbb{C}_{\mathbb{F}}^{0,1}([0,T]\times H;H)$
then\vspace{-3mm}
\begin{equation}\label{th4.1-eq2}
	\begin{cases}
		V_{x}(t,\overline{X}(t))=-p(t),  \\
		\ns\ds
		V_{xx}(t,\overline{X}(t))b(t,\overline{X}(t),\overline{u}(t))+\Phi_x(t,\overline{X}(t))=-q(t),
	\end{cases}\q \mbox{ a.e. }  (t,\omega)\in [0,T]\times\Omega.\vspace{-1mm}
\end{equation}
\end{theorem}
\begin{proof}
Fix $(s,x) \in [0,T] \times H$. Since $V$ satisfies the stochastic HJB equation \eqref{eqHJB} and admits the representation \eqref{eq611}, we deduce\vspace{-1mm}
$$\begin{aligned}
	\Gamma(s, x) &= \left\langle A^*V_x(s,x), x \right\rangle_H + \inf_{u \in U} \dbH(s, x, u, V_x(s, x), \Phi_x(s, x), V_{xx}(s, x)) \\
	&\leq \left\langle A^*V_x(s,x), x \right\rangle_H + \dbH(s, x, \overline{u}(s), V_x(s, x), \Phi_x(s, x), V_{xx}(s, x)).
\end{aligned}$$
Let $\overline{X}(\cdot) := \overline{X}(\cdot; 0, x, \overline{u}(\cdot))$. Then for all $s \in [0,T]$, we have\vspace{-1mm}
$$\begin{aligned}
	0 &= \left\langle A^*V_x(s,\overline{X}(s)), \overline{X}(s) \right\rangle_H + \dbH(s, \overline{X}(s), \overline{u}(s), V_x(s, \overline{X}(s)), \Phi_x(s, \overline{X}(s)), V_{xx}(s, \overline{X}(s))) - \Gamma(s, \overline{X}(s)) \\
	&\leq \left\langle A^*V_x(s,x), x \right\rangle_H + \dbH(s, x, \overline{u}(s), V_x(s, x), \Phi_x(s, x), V_{xx}(s, x)) - \Gamma(s, x).
\end{aligned}$$
By the additional assumptions on $V,\Gamma$ and $\Phi$, 
the first-order condition yields\vspace{-1mm}
$$\frac{\partial}{\partial x} \Big(\left\langle A^*V_x(s,x), x \right\rangle + \dbH(s, x, \overline{u}(s), V_x(s, x), \Phi_x(s, x), V_{xx}(s, x)) - \Gamma(s, x) \Big) \bigg|_{x = \overline{X}(s)} = 0.\vspace{-1mm}
$$
This implies the following identity:\vspace{-1mm}
$$\begin{aligned}
	0 &= A^*V_{xx}(s,\overline{X}(s))\overline{X}(s) + A^*V_x(s,\overline{X}(s)) \\
	&\quad + a_x(s,\overline{X}(s),\overline{u}(s))^*V_x(s,\overline{X}(s)) + V_{xx}(s,\overline{X}(s))a(s,\overline{X}(s),\overline{u}(s)) \\
	&\quad + b_x(s,\overline{X}(s),\overline{u}(s))^*\Phi_x(s,\overline{X}(s)) + \Phi_{xx}(s,\overline{X}(s))b(s,\overline{X}(s),\overline{u}(s)) \\
	&\quad + b_x(s,\overline{X}(s),\overline{u}(s))^*V_{xx}(s,\overline{X}(s))b(s,\overline{X}(s),\overline{u}(s)) \\
	&\quad + \frac{1}{2}\sum_{j=1}^{\infty} V_{xxx}(s,\overline{X}(s))(b(s,\overline{X}(s),\overline{u}(s))e_j, b(s,\overline{X}(s),\overline{u}(s))e_j) \\
	&\quad + f_x(s,\overline{X}(s),\overline{u}(s)) - \Gamma_x(s,\overline{X}(s)),
\end{aligned}\vspace{-1mm}
$$
where the derivative\vspace{-1mm}
$$\frac{\partial}{\partial x} \left\langle A^*V_x(s,x), x \right\rangle_H \bigg|_{x = \overline{X}(s)} = A^*V_{xx}(s,\overline{X}(s))\overline{X}(s) + A^*V_x(s,\overline{X}(s))\vspace{-1mm}
$$
follows from the assumption $A^*V_x\in \mathbb{C}_{\mathbb{F}}^{0,1}([0,T]\times H;H)$, with $\{e_j\}_{j=1}^{\infty}$ being an orthonormal basis of $H$.

By definition, the spatial derivative satisfies\vspace{-1mm}
$$V_x(t,x) = h_x(x) + \int_t^T \Gamma_x(s,x) ds - \int_t^T \Phi_x(s,x) dW(s), \quad t \in [0,T].\vspace{-1mm}
$$
Applying the It\^o-Kunita formula (see Lemma \ref{Ito-Kunita}) to $-V_x(s,\overline{X}(s))$, we obtain that\vspace{-1mm}
\begin{eqnarray}\label{eq:4.16}
	&& dV_x(s,\overline{X}(s)) \nonumber \\
	&& = \Big( -\Gamma_x(s, \overline{X}(s)) + A^* V_{xx}(s,\overline{X}(s))\overline{X}(s) + V_{xx}(s,\overline{X}(s))a(s,\overline{X}(s),\overline{u}(s)) \nonumber \\
	&& \quad + \frac{1}{2}\sum_{j=1}^{\infty} V_{xxx}(s,\overline{X}(s))(b(s,\overline{X}(s),\overline{u}(s))e_j, b(s,\overline{X}(s),\overline{u}(s))e_j) \nonumber \\
	&& \quad + \Phi_{xx}(s,\overline{X}(s))b(s,\overline{X}(s),\overline{u}(s)) \Big) ds \nonumber \\
	&& \quad + \left( V_{xx}(s,\overline{X}(s))b(s,\overline{X}(s),\overline{u}(s)) + \Phi_x(s, \overline{X}(s)) \right) dW(s), \nonumber \\
	&& = - \Big( A^*V_x(s,\overline{X}(s)) + a_x(s,\overline{X}(s),\overline{u}(s))^*V_x(s,\overline{X}(s)) + f_x(s,\overline{X}(s)),\overline{u}(s)) \nonumber \\
	&& \quad + b_x(s,\overline{X}(s),\overline{u}(s))^* \left[ V_{xx}(s,\overline{X}(s))b(s,\overline{X}(s),\overline{u}(s)) + \Phi_x(s,\overline{X}(s)) \right] \Big) ds \nonumber \\
	&& \quad + \left( V_{xx}(s,\overline{X}(s))b(s,\overline{X}(s),\overline{u}(s)) + \Phi_x(s, \overline{X}(s)) \right) dW(s). \nonumber
\end{eqnarray}
Since $h_x(\overline{X}(T)) = V_x(T,\overline{X}(T))$, the uniqueness of solutions to the BSEE \eqref{ad-eq1} establishes \eqref{th4.1-eq2}.  
\end{proof}

\subsection{Relationships between PMP and DPP: nonsmooth case}\label{ssec-nonsmooth}

In Subsection~\ref{ssec-smooth}, we established the connection between the
PMP and 
DPP under smoothness assumptions on the value function. However, in general settings--particularly for degenerate stochastic systems--the value function frequently fails to maintain such smoothness properties. This observation motivates the central objective of our current work: to weaken the regularity conditions imposed on the value function in Theorem~\ref{Th5}, thereby extending the applicability of these fundamental principles to more general cases.

\subsubsection{Differential in Spatial Variable}

Let  $v \in \mathbb{C}_{\mathbb{F}}^{0,0}([0,T]\times H)$. The second-order parabolic superdifferential of $v$ at $(t,\omega,x) \in [0,T) \times \Omega \times H$ is defined by\vspace{-2mm}
\begin{eqnarray*}
D_{x}^{2,+}v(t,\omega,x) \3n& \deq\3n & \Big\{(p,P)\in  H\times \mathcal{S}(H)\Big| \displaystyle\uplim\limits_{\begin{subarray}{1}
y\to \eta
\end{subarray}}\frac{1}{|x-y|_H^2}\\
& & \Big[v(t,\omega,y)-v(t,\omega,x)-\langle p,y-x\rangle_H-\frac{1}{2}\langle P(y-x),y-x\rangle_H\Big]\leq 0 \Big\}.\vspace{-1mm}
\end{eqnarray*}
Similarly, the second-order parabolic subdifferential of $v$ at $(t,\omega,x)$ is defined by\vspace{-2mm}
\begin{eqnarray*}
D_{x}^{2,-}v(t,\omega,x) \3n& \deq\3n & \Big\{(p,P)\in \times H\times \mathcal{S}(H)\Big| \displaystyle\lowlim\limits_{\begin{subarray}{1}
y\to \eta
\end{subarray}}\frac{1}{|x-y|_H^2}\\
& & \Big[v(t,\omega,y)-v(t,\omega,x)-\langle p,y-x\rangle_H-\frac{1}{2}\langle P(y-x),y-x\rangle_H\Big]\geq 0 \Big\}.
\end{eqnarray*}

For an $S\in \cS(H)$, we denote by\vspace{-1mm}
$$[S,\infty)\deq \{R\in \cS(H)|R-S \text{ is a nonnegative operator }\}$$
and\vspace{-1mm}
$$(-\infty,S]\deq \{R\in \cS(H)|S-R \text{ is a nonnegative operator }\}.$$
Then the following relationship holds between PMP and DPP in the spatial variable:
\begin{theorem}\label{th5.1}  Suppose Assumptions  {\rm ({\bf S1})--({\bf S3})} hold. For a fixed initial state $\eta \in H$, let $(\overline{X}(\cdot), \bar{u}(\cdot),p(\cdot)$, $q(\cdot), P(\cdot), Q^{(\cdot)}, \widehat{Q}^{(\cdot)})$ be an optimal $7$-tuple for \textbf{Problem (OP)}, with $V \in \mathbb{C}_{\mathbb{F}}^{0,0}([0,T]\times H)$ being the corresponding value function. Then the following differential inclusions hold:\vspace{-1mm}
\begin{equation}\label{th5.1-eq1}
\{-p(t,\omega)\} \times [-P(t,\omega), \infty) \subset D_x^{2,+}V(t,\omega,\overline{X}(t,\omega)), \quad \forall t \in [0,T],\ \mathbb{P}\text{-a.s.}
\end{equation}
and\vspace{-1mm}
\begin{equation}\label{Th71}
D_x^{2,-}V(t,\omega,\overline{X}(t,\omega)) \subset \{-p(t,\omega)\} \times (-\infty, -P(t,\omega)], \quad \forall t \in [0,T],\ \mathbb{P}\text{-a.s.}
\end{equation}
\end{theorem}
\begin{proof} 
Our proof adapts techniques from \cite{Yong1999,ChenLu} and is organized in six steps.

{\bf Step  1}.  Fix $t \in [0,T]$. For any initial state $z \in H$, consider the  SEE:\vspace{-1mm}
\begin{equation}\label{12.13-eq56}
\begin{cases}
\ds    dX^z(r)=\big(AX^z(r)+a(r,X^z(r),\bar{u}(r))\big)dr+b(r,X^z(r),\bar{u}(r))dW(r),\ \ \ \ r\in(t,T],\\
\ns\ds  X^z(t)=z.
\end{cases}\vspace{-1mm}
\end{equation}
Define the deviation process $\xi^z(r):=X^z(r)-\overline{X}(r)$. 

Interpreting \eqref{12.13-eq56} on the filtered probability space $(\Omega, \mathcal{F}, \mathbf{F}, \mathbb{P}(\cdot|\mathcal{F}_t)(\omega))$ for $\mathbb{P}$-a.e. $\omega$, we obtain the following continuous dependence estimate for any integer $k \geq 1$:  \vspace{-1mm}
\begin{equation}
\mathbb{E}\big(\sup\limits_{t \leq r \leq T} |\xi^z(r)|_H^{2k}\big|
\mathcal{F}_t\big)\leq \cC|z-\overline{X}(t)|_H^{2k},\q \mathbb{P}\mbox{-a.s.}\vspace{-1mm}
\end{equation}

The deviation process $\xi^z(\cdot)$ admits two distinct Taylor-type expansions. The first oeder one is \vspace{-1mm} 
\begin{equation}\label{12.13-eq58}
\!\begin{cases}
\ds  d\xi^z(r)\! =\! \big(A\xi^z(r)\!+\! \bar{a}_{x}(r)\xi^z(r)\big) dr\! + \! \bar{b}_{x}(r)\xi^z(r) dW(r)\! +\!\epsilon_{z,a}(r)dr\!+ \! \epsilon_{z,b}(r)dW(r),
\; r\in(t,T],\\
\ns\ds  \xi^z(t)=z-\overline{X}(t),
\end{cases}\vspace{-1mm}
\end{equation}
and the second-order one is\vspace{-1mm}
\begin{equation}\label{12.13-eq59}
\!\begin{cases}\ds
d\xi^z(r)=\(A\xi^z(r)+ \bar{a}_{x}(r)\xi^z(r)+\frac{1}{2}  \bar{a}_{xx}(r)\big(\xi^z(r),\xi^z(r)\big) \) dr\\
\ns\ds  \ \ \ \ \ \ \ \ \ \  \ + \(  \bar{b}_{x}(r) \xi^z(r)\!+\!\frac{1}{2} \bar{b}_{xx}(r)\big(\xi^z(r),\xi^z(r)\big)\)dW(r)\! + \!\tilde\epsilon_{z,a}(r)dr\!+ \! \tilde\epsilon_{z,b}(r)dW(r),
\; r\in(t,T],\\
\ns\ds  \xi^z(t)=z-\overline{X}(t),
\end{cases}\vspace{-1mm}
\end{equation}
where for $\f=a,b$\vspace{-1mm}
\begin{equation*}
\bar{\f}_{x}(r):=\f_{x}(r,\overline{X}(r),\bar{u}(r)),\q 
\bar{\f}_{xx}(r):=\f_{xx}(r,\overline{X}(r),\bar{u}(r)), \vspace{-1mm}
\end{equation*}
and\vspace{-3mm}
\begin{equation*}
\begin{cases}\ds
\epsilon_{z,\f}(r)=\int_0^1 \big(\f_{x}(r,\overline{X}(r)+\theta \xi^{z}(r),\bar{u}(r))-\bar{\f}_{x}(r)\big)\xi^{z}(r)d\theta, \\
\ns\ds  \wt \epsilon_{z,\f}(r)=\int_0^1 (1-\theta)   \big(\f_{xx}(r,\overline{X}(r)+\theta \xi^{z}(r),\bar{u}(r))-\bar{\f}_{xx}(r)\big)\big(\xi^{z}(r),\xi^{z}(r)\big) d\theta.
\end{cases}\vspace{-1mm}
\end{equation*}
\ss

{\bf Step  2}.  
In this step, we prove the existence of a deterministic, continuous, and strictly increasing function 
$\delta \colon [0,\infty) \to [0,\infty)$, independent of $z \in H$, with $\delta(r) = o(r)$ as $r \to 0^+$, 
such that for any integer $k \geq 1$, the following moment estimates hold: \vspace{-1mm}
\begin{equation}\label{12.13-eq62}
\mathbb{E}\(\int_t^T |\epsilon_{z,a}(r)|_H^{2k}dr\Big|\mathcal{F}_t\)(\omega)+\mathbb{E}\(\int_t^T |\epsilon_{z,b}(r)|_{\cL_2^0}^{2k}dr\Big|\mathcal{F}_t\)(\omega)\leq \delta \big(|z-\overline{X}(t,\omega)|_H^{2k}\big), \hspace{1.06cm} \mathbb{P}\mbox{-a.s. } \omega,\vspace{-1mm}
\end{equation}
\begin{equation}\label{12.13-eq63}
\mathbb{E}\(\int_t^T |\tilde\epsilon_{z,a}(r)|_H^{k}dr\Big|\mathcal{F}_t\)(\omega)+\mathbb{E}\(\int_t^T |\tilde\epsilon_{z,b}(r)|_{\cL_2^0}^{k}dr\Big|\mathcal{F}_t\)(\omega)\leq \delta \big(|z-\overline{X}(t,\omega)|_H^{2k}\big), \hspace{1.06cm} \mathbb{P}\mbox{-a.s. }\omega.\vspace{-1mm}
\end{equation}

For notational convenience, denote $\f_x(r,\theta) := \f_x(r,\overline{X}(r) + \theta \xi^z(r))$. Under Assumption ({\bf S3}), we establish the first moment estimate:\vspace{-1mm}
\begin{eqnarray*}
\mathbb{E}\(\int_t^T |\epsilon_{z,a}(r)|_H^{2k}dr\Big|\mathcal{F}_t\)
\3n  & \leq\3n &  \int_t^T \mathbb{E} \Big(\int_0^1 |a_{x}(r,\theta)-\bar{a}_{x}(r)|_{\mathcal{L}(H)}^{2k}d\theta |\xi^z (r)|_H^{2k}\Big|\mathcal{F}_t\Big)dr\\
& \leq\3n & \cC\int_t^T \mathbb{E} \big(|\xi^z (r)|_H^{4k}\big|\mathcal{F}_t\big)dr\leq \cC|z-\overline{X}(t,\omega)|_H^{4k}.
\end{eqnarray*}
This proves \eqref{12.13-eq62} with $\delta(x) = \cC x^2$.

For the second-order terms, denote $\f_{xx}(r,\theta) := \f_{xx}(r,\overline{X}(r) + \theta \xi^z(r))$ and compute \vspace{-1mm}
\begin{eqnarray*}
\mathbb{E}\(\int_t^T |\tilde\epsilon_{z,a}(r)|_H^{k}dr\Big|\mathcal{F}_t\)
&\3n \leq \3n&  \int_t^T \mathbb{E} \(\int_0^1 |a_{xx}(r,\theta)-\bar{a}_{xx}(r)|_{\mathcal{L}(H,H;H)}^{k}d\theta |\xi^z (r)|_H^{2k}\Big|\mathcal{F}_t\)dr\\
&\3n \leq \3n& \int_t^T \big\{\mathbb{E} \big[\bar{\omega}\big(|\xi^z (r)|_H\big)^{2k}\big|\mathcal{F}_t\big]\big\}^{1/2}\big( \mathbb{E} |\xi^z (r)|_H^{4k}\big|\mathcal{F}_t\big)^{1/2}dr\\
&\3n \leq \3n& \cC|z-\overline{X}(t,\omega)|_H^{2k}\int_t^T \big\{ \mathbb{E}^{t}[\bar{\omega}(|\xi^z (r)|_H)^{2k}\big|\mathcal{F}_t]\big\}^{1/2}dr.
\end{eqnarray*}
The modulus of continuity $\bar{\omega}$ ensures the existence of a suitable $\delta(\cdot)$ satisfying \eqref{12.13-eq63}. 

By taking the pointwise supremum over all such admissible functions, we obtain a maximal $\delta(\cdot)$ that simultaneously satisfies both \eqref{12.13-eq62} and \eqref{12.13-eq63}.

\ss

{\bf Step  3.} Define the first-order derivative $\bar{f}_x(r) := f_x(r,\overline{X}(r),\bar{u}(r))$. Using the adjoint equation \ref{ad-eq1}, we obtain that\vspace{-1mm}
\begin{eqnarray}\label{eq79}
& & \mathbb{E}\Big(\int_t^T \langle \bar{f}_{x}(r), \xi^z (r)\rangle_H dr+\langle h_{x}(\overline{X}(T)),\xi^z (T)\rangle_H \Big| \mathcal{F}_t \Big) \nonumber\\
& & =\langle -p(t),\xi^z (t)\rangle_H  -\frac{1}{2}\mathbb{E}\Big[\int_t^T \big(\langle p(r), \xi^z (r)^*\bar{a}_{xx}(r)\xi^z(r) \rangle_H  + \langle q(r), \xi^z (r)^*\bar{b}_{xx}(r)\xi^z(r) \rangle_{\mathcal{L}_2^0} \big)dr\nonumber\\
& & \ \ \ -\int_t^T \big(\langle p(r),\tilde\epsilon_{z,a}(r)\rangle_H + \langle q(r), \tilde\epsilon_{z,b}(r)\rangle_{\mathcal{L}_2^0} \big)dr\big|\mathcal{F}_t \Big], \ \ \mathbb{P}\mbox{-a.s.}
\end{eqnarray}
Working on the conditioned probability space $(\Omega, \mathcal{F}, \mathbf{F}, \mathbb{P}(\cdot|\mathcal{F}_t)(\omega))$ with conditional expectation $\mathbb{E}^t_\omega := \mathbb{E}(\cdot|\mathcal{F}_t)(\omega)$, Definition \ref{def2.1} yields\vspace{-1mm}
\begin{eqnarray}\label{eq710}
& & \mathbb{E}^t_\omega\Big(\int_t^T \langle \mathbb{H}_{xx}(r)\xi^z(r),\xi^z(r)\rangle_H dr-\langle h_{xx}(T)\xi^z(T),\xi^z(T)\rangle_H \Big)\nonumber\\
& & = \langle P(t)\xi^z(t),\xi^z(t)\rangle_H + \mathbb{E}^t_\omega \int_t^T \langle P(r)\epsilon_{z,a}(r),\xi^z(r)\rangle_H dr \nonumber \\
& &\ \ \  + \mathbb{E}^t_\omega \int_t^T \big(\langle P(r)\xi^z(r),\epsilon_{z,a}(r)\rangle_H + \langle P(r)\bar{b}_{x}(r)\xi^z(r),\epsilon_{z,b}\rangle_{\mathcal{L}_2^0}\big)dr \nonumber \\
& & \ \ \ +\mathbb{E}^t_\omega \int_t^T \langle P(r)\epsilon_{z,b}(r),\bar{b}_{x}(r)\xi^z(r)+\epsilon_{z,b}(r)\rangle_{\mathcal{L}_2^0}dr \\
& & \ \ \ +\mathbb{E}^t_\omega\int_t^T \big(\langle \epsilon_{z,b}(r),\hat{Q}^{(t)}(r)\rangle_{\mathcal{L}_2^0}+ \langle Q^{(t)}(r),
\epsilon_{z,b}(r)\rangle_{\mathcal{L}_2^0}\big)dr,\qquad \mathbb{P}\mbox{-a.s.}
\nonumber
\end{eqnarray}

\ss

{\bf Step  4}. In this step, we compute  $V(t,\omega,z)-V(t,\omega,\overline{X}(t,\omega))$.

Let $M$ be a countable dense subset of $H$. We calim that there exists a full-measure set $\Omega_0 \subset \Omega$ ($\mathbb{P}(\Omega_0)=1$) such that for all $\omega_0 \in \Omega_0$, it holds \vspace{-1mm}
\begin{equation*}
\begin{cases}\ds
V(t,\omega_0,\overline{X}(t,\omega_0))=\mathbb{E}\(\int_t^T f(r,\overline{X}(r),\bar{u}(r))dr+h(\overline{X}(T))\Big|\mathcal{F}_t\)(\omega_0), \\
\ns\ds    \textup{\eqref{12.13-eq56}, \eqref{12.13-eq62}--\eqref{eq710} hold for any} \ z\in M,\\
\ns\ds \sup \limits_{s\leq r\leq T}|p(r,\omega_0)|<+\infty,\\
\ns\ds      P(t,\omega_0)\in \mathcal{L}(H),\ P(\cdot,\omega_0)\xi\in L^2(r,T), \ \forall \xi \in L_{\mathcal{F}_r}^2(\Omega;H), \ \forall r\in [t,T].
\end{cases}\vspace{-1mm}
\end{equation*}

The first equality follows from Theorem \ref{thm-sdy1}, the second from the regularity $p \in L_{\mathbb{F}}^2(\Omega,C([0,T];H))$: \vspace{-1mm}
\begin{equation}
\mathbb{E}\sup_{0 \leq r \leq T} |p(r)|_H^2  < +\infty,\vspace{-1mm}
\end{equation}
and the third from $P(\cdot,\cdot) \in \mathcal{P}[0,T]$. 

For fixed $\omega_0 \in \Omega_0$ and $z \in M$, the value function difference satisfies \vspace{-1mm}
\begin{eqnarray}\label{eq712}
& & V(t,\omega_0,z)-V(t,\omega_0,\overline{X}(t,\omega_{0})) \nonumber\\
& & \leq \mathbb{E}^{t}_{\omega_0}\Big[\int_t^T \big(f(r,X^{z}(r),\bar{u}(r))-\bar{f}(r)\big)dr+h(^{z}(T))-h(\overline{X}(T))\Big]\nonumber\\
& & =\mathbb{E}^{t}_{\omega_0}\Big(\int_t^T \langle \bar{f}_{x}(r), \xi ^{z}(r)\rangle_H dr+\langle h_{x}(\overline{X}(T)),\xi^{z}(T)\rangle_H \Big)  \\
& & \ \ \  + \frac{1}{2}\mathbb{E}^{t}_{\omega_0} \Big(\int_t^T \langle \bar{f}_{xx}(r)\xi^{z}(r),\xi^{z}(r) \rangle_H dr+\langle h_{xx}(\overline{X}(T))\xi^{z}(T), \xi^{z}(T)\rangle_H \Big) +  o\big(|z-\overline{X}(t,\omega_{0})|_H^2\big). \nonumber 
\end{eqnarray}

Combining with \eqref{eq79}, \eqref{eq710} and the Hamiltonian definition yields \vspace{-1mm}
\begin{eqnarray}\label{eq713}
& &\3n\3n\3n V(t,\omega_0,z)-V(t,\omega_0,\overline{X}(t,\omega_0))\nonumber\\
& &\3n\3n\3n \leq \!-\langle p(t,\omega_0),\xi^z(t,\omega_0)\rangle_H \!-\! \frac{1}{2}\mathbb{E}^t_{\omega_0}\!\Big[\!\int_t^T\!\!\big(\lan p(r),  \bar{a}_{xx}(r)\big(\xi^z\!(r),\xi^z\!(r)\big) \ran_H \! + \!\lan q(r),  \bar{b}_{xx}(r)\big(\xi^z\!(r),\xi^z\!(r)\big) \ran_{\mathcal{L}_2^0} \big)dr\nonumber\\
& &   -\int_t^T \big(\langle p(r),\tilde \epsilon_{z,a}(r)\rangle_H + \langle q(r), \tilde \epsilon_{z,b}(r)\rangle_{\mathcal{L}_2^0} \big)dr\Big] \nonumber \\
& &   +\ \frac{1}{2}\mathbb{E}^{t}_{\omega_0} \Big(\int_t^T \langle \bar{f}_{xx}(r)\xi^{z}(r),\xi^{z}(r) \rangle_H dr+\langle h_{xx}(\overline{X}(T))\xi^{z}(T), \xi^{z}(T)\rangle_H \Big) + \ o(|z-\overline{X}(t,\omega_{0})|_H^2)\nonumber\\
& &\3n\3n\3n  = -\langle p(t,\omega_0),\xi^z(t,\omega_0)\rangle_H -\frac{1}{2}\mathbb{E}^t_{\omega_0}\Big(\int_t^T \langle \mathbb{H}_{xx}(r)\xi^z(r),\xi^z(r)\rangle_H dr - \langle h_{xx}(\overline X(T))\xi^z(T),\xi^z(T)\rangle_H \Big)\nonumber\\
& &  -\mathbb{E}^t_{\omega_0} \int_t^T \big(\langle p(r),\tilde \epsilon_{z,a}(r)\rangle_H + \langle q(r), \tilde \epsilon_{z,b}(r)\rangle_{\mathcal{L}_2^0} \big)dr + o(|z-\overline{X}(t,\omega_0)|_H^2)\\
& & \3n\3n\3n = - \langle p(t,\omega_0),\xi^z(t,\omega_0)\rangle_H - \frac{1}{2}\langle P(t,\omega_0)\xi^z(t,\omega_0),\xi^z(t,\omega_0)\rangle_H\nonumber\\
& &   -\mathbb{E}^t_{\omega_0} \int_t^T \big(\langle p(r),\tilde \epsilon_{z,a}(r)\rangle_H + \langle q(r), \tilde \epsilon_{z,b}(r)\rangle_{\mathcal{L}_2^0} \big)dr  -\frac{1}{2}\mathbb{E}^t_{\omega_0} \int_t^T \langle P(r)\epsilon_{z,a}(r),\xi^z(r)\rangle_H dr \nonumber\\
& & -\mathbb{E}^t_{\omega_0} \int_t^T \big(\langle P(r)\xi^z(r),\epsilon_{z,a}(r)\rangle_H + \langle P(r)\bar{b}_{x}(r)\xi^z(r),\epsilon_{z,b}(r)\rangle_{\mathcal{L}_2^0}\big)dr  \nonumber \\
& &  -\mathbb{E}^t_{\omega_0} \int_t^T \langle P(r)\epsilon_{z,b}(r),\bar{b}_{x}(r)\xi^z(r)+\epsilon_{z,b}(r)\rangle_{\mathcal{L}_2^0}dr \nonumber\\
& &  -\mathbb{E}^t_{\omega_0} \int_t^T \big(\langle \epsilon_{z,b}(r),\hat{Q}^{(t)}(r)\rangle_{\mathcal{L}_2^0}+ \langle Q^{(t)}(r), \epsilon_{z,b}(r)\rangle_{\mathcal{L}_2^0}\big)dr + \ o(|z-\overline{X}(t,\omega_0)|_H^2).\nonumber
\end{eqnarray}

\ss

{\bf Step  5}. In this step, we estimate the  remainder term in equality \eqref{eq713}. 

First, using the estimates \eqref{12.13-eq62} and \eqref{12.13-eq63}, we bound the integral terms involving the adjoint processes:\vspace{-4mm}
\begin{eqnarray}\label{eqpez}
& & \Big|\mathbb{E}^t_{\omega_0} \int_t^T \big( \langle p(r),\tilde \epsilon_{z,a}(r)\rangle_H + \langle q(r), \tilde \epsilon_{z,b}(r)\rangle_{\mathcal{L}_2^0} \big)dr\Big|\nonumber\\
& \leq \3n& \int_t^T \mathbb{E}^t_{\omega_0}\big(|p(r)|_H|\tilde \epsilon_{z,a}(r)|_H+|q(r)|_{\mathcal{L}_2^0}|\tilde \epsilon_{z,b}(r)|_{\mathcal{L}_2^0}\big)dr \\
& \leq \3n& |p |_{L^2_{\mathbb{F}}(\Omega,C([t,T];H))}\Big(\mathbb{E}^t_{\omega_0}\int_t^T |\tilde \epsilon_{z,a}(r)|_H^2dr\Big)^{1/2} + |q |_{L_{\mathbb{F}}^2(0,T;\mathcal{L}_2^0)} \Big(\mathbb{E}^t_{\omega_0}\int_t^T |\tilde \epsilon_{z,b}(r)|_{\mathcal{L}_2^0}^2dr\Big)^{1/2}\nonumber\\
& \leq\3n & \cC\delta(|z-\overline{X}(t,\omega_0)|_H^2)= o(|z-\overline{X}(t,\omega_0)|_H^2).\nonumber
\end{eqnarray}
Next, applying Proposition \ref{prp2.1} and the regularity $P \in \mathcal{P}[0,T]$, we get that \vspace{-2mm}
\begin{eqnarray}\label{eqpez1}
& &  \Big|\mathbb{E}^t_{\omega_0} \int_t^T \langle P(r)\epsilon_{z,a}(r), \xi^z(r)\rangle_H dr\Big| \nonumber\\
& &  \leq\Big[\int_t^T \big(\mathbb{E}^t_{\omega_0}|P(r)\epsilon_{z,a}(r)|_{H}^{4/3}\big)^{3/2}dr\Big]^{1/2}\Big(\int_t^T \big(\mathbb{E}^t_{\omega_0}|\xi^z(r)|_{H}^4\big)^{1/2}dr\Big)^{1/2}
\\
& & \leq |P |_{\mathcal{L}(L_\mathbb{F}^2(0,T;L^4(\Omega;H));L_{\mathbb{F}}^2(0,T;L^{4/3}(\Omega;H)))}\Big(\mathbb{E}^t_{\omega_0} \int_t^T |\epsilon_{z,a}(r)|_H^4dr \Big)^{1/4} \Big(\sup\limits_{t\leq r\leq T}\mathbb{E}^t_{\omega_0} |\xi^z(r)|_H^4 \Big)^{1/4}\nonumber\\
& &  =  o(|z-\overline{X}(t,\omega_0)|_H)O(|z-\overline{X}(t,\omega_0 )|_H)= o(|z-\overline{X}(t,\omega_0)|_H^2).\nonumber
\end{eqnarray}
Similarly, we obtain\vspace{-2mm}
\begin{eqnarray}\label{eqpez2}
& &  \Big|\mathbb{E}^t_{\omega_0} \int_t^T \lan P(r)\xi^z(r),\epsilon_{z,a}(r)\ran_H dr \Big|= o\big(|z-\overline{X}(t,\omega_0)|_H^2\big).
\end{eqnarray}
Using the boundedness of $\bar{b}_x$ and $P \in \mathcal{P}[0,T]$, we get that
\begin{eqnarray}\label{eqpez3}
& & \3n\3n  \Big|\mathbb{E}^t_{\omega_0}\int_t^T \langle P(r)\bar{b}_x(r)\xi^z(r),\epsilon_{z,b}(r)\rangle_{\mathcal{L}_2^0} dr\Big|\nonumber\\
& &\3n\3n \leq \Big[\int_t^T \big(\mathbb{E}^t_{\omega_0}|P(r)\bar{b}_x(r)\xi^z(r)|_{\mathcal{L}_2^0}^{4/3}\big)^{3/2}dr\Big]^{1/2}\Big(\int_t^T \big(\mathbb{E}^t_{\omega_0}|\epsilon_{z,b}(r)|_{\mathcal{L}_2^0}^4\big)^{1/2}dr\Big)^{1/2}\nonumber\\
& & \3n\3n\leq\! |P |_{\mathcal{L}(L_\mathbb{F}^2(0,T;L^4(\Omega;H));L_{\mathbb{F}}^2(0,T;L^{4/3}(\Omega;H)))} |\bar{b}_x \xi^z |_{L_\mathbb{F}^2(0,T;L^4(\Omega;\mathcal{L}_2^0))}\Big[\!\int_t^T\!\!\big(\mathbb{E}^t_{\omega_0}|\epsilon_{z,b}(r)|_{\mathcal{L}_2^0}^4\big)^{1/2}dr\Big]^{1/2}\nonumber\\
& &\3n\3n \leq \cC|\xi^z |_{L_\mathbb{F}^2(0,T;L^4(\Omega;H))}\Big(\int_t^T \big(\mathbb{E}^t_{\omega_0}|\epsilon_{z,b}(r)|_{\mathcal{L}_2^0}^4\big)^{1/2}dr\Big)^{1/2} \\
& &\3n\3n \leq \cC(|z-\overline{X}(t,\omega_0)|_H^4)^{1/4}\delta(|z-\overline{X}(t,\omega_0)|_H^4)^{1/4}= o(|z-\overline{X}(t,\omega_0)|_H^2).\nonumber
\end{eqnarray}
The remaining terms are estimated similarly:\vspace{-1mm}
\begin{eqnarray}\label{eqpez4}
& &  \Big|\mathbb{E}^t_{\omega_0} \int_t^T \langle P(r)\epsilon_{z,b}(r),\bar{b}_x(r)\xi^z(r)+\epsilon_{z,b}(r)\rangle_{\mathcal{L}_2^0} dr\Big|=  o(|z-\overline{X}(t,\omega_0)|_H^2).
\end{eqnarray}
Finally, using \eqref{12.13-eq62} and the properties of the relaxed transposition solution $\widehat{Q}^{(t)}$ to \eqref{ad-eq2}, we obtain \vspace{-1mm}
\begin{eqnarray}\label{eqpez5}
& &  \Big|\mathbb{E}^t_{\omega_0} \int_t^T \langle \epsilon_{z,b}(r), \hat{Q}^{(t)}(r)\rangle_{\mathcal{L}_2^0}dr\Big|\nonumber\\
& & \leq \big|\hat{Q}^{(t)}(0,0,\epsilon_{z,b}(\cdot))\big|_{L_\mathbb{F}^2(0,T;L^{4/3}(\Omega;\mathcal{L}_2^0))}\Big[\mathbb{E}^t_{\omega_0}\(\int_t^T |\epsilon_{z,b}(r)|_{\mathcal{L}_2^0}^4dr\)^{1/2}\Big]^{1/2}\nonumber\\
& & \leq \cC|\epsilon_{z,b} |_{L_\mathbb{F}^2(0,T;L^4(\Omega,\mathcal{L}_2^0))} \Big[\mathbb{E}^t_{\omega_0}\(\int_t^T |\epsilon_{z,b}(r)|_{\mathcal{L}_2^0}^4dr\)^{1/2}\Big]^{1/2}\\
& & \leq \cC \Big[\mathbb{E}^t_{\omega_0}\(\int_t^T |\epsilon_{z,b}(r)|_{\mathcal{L}_2^0}^4dr\)\Big]^{1/2}\leq \big(\delta(|z-\overline{X}(t,\omega_0)|_H^4)\big)^{1/2}\nonumber\\
& & =o\big(|z-\overline{X}(t,\omega_0)|_H^2\big).\nonumber
\end{eqnarray}
A similar argument implies
\begin{eqnarray}\label{eqpez6}
\Big|\mathbb{E}^t_{\omega_0} \int_t^T \langle Q^{(t)}(r), \epsilon_{z,b}(r)\rangle_{\mathcal{
L}_2^0}dr\Big|=o\big(|z-\overline{X}(t,\omega_0)|_H^2\big).
\end{eqnarray}

\ss

{\bf Step  6}. We complete the proof in this step.

Building upon the previous steps, we have established the following inequality for the value function:
\begin{eqnarray}\label{12.13-eq75}
& &\3n\3n\3n V(t,\omega_0,z)-V(t,\omega_0,\overline{X}(t,\omega_0))\nonumber\\
& &\3n\3n\3n \leq - \langle p(t,\omega_0),\xi^z(t,\omega_0)\rangle_H - \frac{1}{2}\langle P(t,\omega_0)\xi^z(t,\omega_0),\xi^z(t,\omega_0)\rangle_H +  o(|z-\overline{X}(t,\omega_0)|_H^2).
\end{eqnarray}
From \eqref{eqpez}--\eqref{eqpez6}, it follows that all remainder terms are of order $o(|z - \overline{X}(t,\omega_0)|_H^2)$  uniformly in $z$. Together with the continuity of $V(t,\omega_0,\cdot)$, this implies that inequality \eqref{12.13-eq75} can be extended to all $z \in H$. Consequently, we get the superdifferential inclusion: \vspace{-1mm}
$$
(-p(t,\omega_0),-P(t,\omega_0))\in D_x^{2,+}V(t,\omega_0,\overline{X}(t,\omega_0)),
$$
which proves the first claim \eqref{th5.1-eq1} by definition of the parabolic superdifferential.

To establish the subdifferential inclusion, fix $\omega_0 \in \Omega_0$ such that \eqref{12.13-eq75} holds for all $z \in H$. For any $(\widetilde{p}, \widetilde{P}) \in D_x^{2,-}V(t,\omega_0,\overline{X}(t,\omega_0))$, the definition of the subdifferential gives 
\begin{eqnarray*}
0 \3n&  \leq \3n & \liminf_{z\to \overline{X}(t,\omega_0)}\frac{V(t,\omega_0,z)\!-\!V(t,\omega_0,\overline{X}(t,\omega_0))\!-\!\langle \wt p,z\!-\!\overline{X}(t,\omega_0)\rangle_H\! -\!\frac{1}{2}\langle \wt P(z\!-\overline{X}(t,\omega_0)),z\!-\overline{X}(t,\omega_0)\rangle_H}{|z-\overline{X}(t,\omega_0)|_H^2}\\
& \leq\3n & \liminf_{z\to \overline{X}(t,\omega_0)}\frac{-\langle \wt p+p(t),z-\overline{X}(t,\omega_0)\rangle_H -\frac{1}{2}\langle (\wt P+P(t,\omega_0))(z-\overline{X}(t,\omega_0)),z-\overline{X}(t,\omega_0)\rangle_H}{|z-\overline{X}(t,\omega_0)|_H^2},
\end{eqnarray*}
which implies\vspace{-2mm}
$$\wt p=-p(t),\hspace{1cm} \wt P\leq P(t,\omega_0),$$
thus completing the proof of the second inclusion \eqref{Th71} and establishing Theorem \ref{th5.1} in full. 
\end{proof} 

\subsubsection{Differentials in the time variable}\label{sec-nonsmooth2}

In this section, we study the superdifferential and subdifferential of the value function in the time variable along an optimal trajectory. 

First, we introduce the following definitions for   the super- and subdifferential of $v\in \mathbb{C}_{\mathbb{F}}^{0,0}([0,T]\times H)$ with respect to time variable $t$. Fix $(t,\omega,x) \in (0,T) \times \Omega \times H$:
\begin{eqnarray*}
D_{t,+}^{1,+}v(t,\omega,x) \3n& \deq\3n & \Big\{r\in \dbR\Big| \displaystyle\uplim\limits_{\begin{subarray}{1}
s\to t,s\in (t,T)
\end{subarray}}\frac{1}{|t-s|}\Big[\mathbb{E}_{\omega}^t[v(s,x)-v(t,\omega,x)]-r(s-t)\Big]\leq 0 \Big\},
\end{eqnarray*}
\begin{eqnarray*}
D_{t,+}^{1,-}v(t,\omega,x) \3n& \deq\3n & \Big\{r\in \dbR\Big| \displaystyle\lowlim\limits_{\begin{subarray}{1}
s\to t,s\in (t,T)
\end{subarray}}\frac{1}{|t-s|} \Big[\mathbb{E}_{\omega}^t[v(s,x)-v(t,\omega,x)]-r(s-t)\Big]\geq 0 \Big\}.
\end{eqnarray*}

Then, the relationship between PMP and DPP with respect to the time variable is given as follows:

\begin{theorem}\label{th6.1} 
Under the assumptions of Theorem \ref{th5.1}, for almost every $t \in [0,T]$ satisfying either $\overline{X}(t) \in D(A)$ or $p(t) \in D(A^*)$, the following temporal superdifferential inclusion holds:\vspace{-1mm}
$$\langle\!\langle  A\overline{X}(t,\omega),p(t,\omega)\rangle\!\rangle + \mathcal{H}(t,\overline{X}(t,\omega),\bar{u}(t,\omega))\in D_{t+}^{1,+}V(t,\omega,\overline{X}(t,\omega)),\quad \mathbb{P}\mbox{-a.s.}\vspace{-1mm}$$
where the duality pairing $\langle\!\langle \cdot, \cdot \rangle\!\rangle$ is defined by\vspace{-1mm}
$$\langle\!\langle  A\overline{X}(t),p(t)\rangle\!\rangle = \begin{cases}\langle  A\overline{X}(t),p(t)\rangle_H, & \mbox{ if }\;\overline{X}(t)\in D(A),\\
\ns\ds \langle \overline{X}(t), A^*p(t)\rangle_H, & \mbox{ if }\;p(t)\in D(A^*),
\end{cases}\vspace{-1mm}
$$
and
$\mathcal{H}(t,x,u)\deq -\mathbb{H}(t,x,u,-p(t),-q(t),P(t))$, with $\mathbb{H}$ defined in \eqref{eq:4.1}.
\end{theorem}
\begin{proof} 
Fix $t \in (0,T)$ and let $\tau \in (t,T]$. Consider the solution $X_\tau$ to the SEE on $[\tau,T]$: \vspace{-1mm}
\begin{equation}
\begin{cases}
\ds   dX_{\tau}(r)=AX_{\tau}(r)dr+a(r,X_{\tau}(r),\bar{u}(r))dr+b(r,X_{\tau}(r),\bar{u}(r))dW(r),\ \ r\in (\tau, T],\\
\ns\ds X_{\tau}(\tau)=\overline{X}(t).
\end{cases}\vspace{-1mm}
\end{equation}
Define the deviation process $\xi_{\tau}(r)=X_{\tau}(r)-\overline{X}(r)$ for $r\in[\tau,T]$.
Working under the conditional measure $\mathbb{P}(\cdot|\mathcal{F}_\tau)(\omega)$ for $\mathbb{P}\text{-a.e. } \omega$, we obtain that\vspace{-1mm}
\begin{equation}
\mathbb{E}\(\sup\limits_{\tau \leq r \leq T} |\xi_{\tau}(r)|_H^{2k}| \mathcal{F}_{\tau}\)\leq \cC\big|\overline{X}(\tau)-\overline{X}(t)\big|_H^{2k} ,\ \ \ \mathbb{P}\mbox{-a.s.}\vspace{-1mm}
\end{equation}
Taking conditional expectation $\mathbb{E}(\cdot|\mathcal{F}_t)$ and using the filtration property $\mathcal{F}_t \subset \mathcal{F}_\tau$ yields \vspace{-1mm}
\begin{equation}\label{12.13-eq6.25}
\mathbb{E}\(\sup\limits_{\tau \leq r \leq T} |\xi_{\tau}(r)|_H^{2k}| \mathcal{F}_t\)\leq \cC|\tau-t|^k\big(|A\overline{X}(t)|_H+1+|\overline{X}(t)|_H\big)^k\leq \cC|\tau-t|^k,\ \ \ \mathbb{P}\mbox{-a.s.}\vspace{-1mm}
\end{equation}

The deviation process satisfies two Taylor expansions analogous to \eqref{12.13-eq58} and \eqref{12.13-eq59}:

The first-order counterpart is\vspace{-1mm}
\begin{eqnarray}\label{12.14-eq1}
\3n\!\!\begin{cases}
d\xi_{\tau}(r)\! =\! \big(A\xi_{\tau}(r)\!+\! \bar{a}_{x}(r)\xi_{\tau}(r)\big) dr\! +\!  \bar{b}_{x}(r) \xi_{\tau}(r) dW\!(r) \!+\!\epsilon_{\tau,a}(r)dr\!+ \! \epsilon_{\tau,b}(r)dW\!(r),
\,\; r\in(\tau,T],\\
\ns\ds  \xi_{\tau}(\tau)=-[S(\tau-t)-I]\overline{X}(t)-\int_t^{\tau} S(\tau-r)\bar{a}(r)dr-\int_t^{\tau} S(\tau-r)\bar{b}(r)dW(r),
\end{cases}
\end{eqnarray}
and the second-order counterpart is\vspace{-1mm}
\begin{eqnarray}\label{12.14-eq2}
\3n\!\begin{cases}\ds
d\xi_{\tau}(r)=\(A\xi_{\tau}(r)+ \bar{a}_{x}(r)\xi_{\tau}(r)+\frac{1}{2}  \bar{a}_{xx}(r)\big(\xi_{\tau}(r),\xi_{\tau}(r) \big)  \) dr\\
\ns\ds \ \ \ \ \ \ \ \ \ \ \   +  \(  \bar{b}_{x}(r) \xi_{\tau}(r)\!+\!\frac{1}{2} \bar{b}_{xx}(r)\big(\xi_{\tau}(r),\xi_{\tau}(r) \big)\)dW\!(r)\! +\!   \tilde\epsilon_{\tau,a}(r)d\!+\!  \tilde\epsilon_{\tau,b}(r)dW\!(r), 
\,\;r\in(\tau,T],\\
\ns\ds  \xi_{\tau}(\tau)=-[S(\tau-t)-I]\overline{X}(t)-\int_t^{\tau} S(\tau-r)\bar{a}(r)dr-\int_t^{\tau} S(\tau-r)\bar{b}(r)dW(r).
\end{cases}
\end{eqnarray}
Here, for $\f = a,b$, the remainder terms are given by\vspace{-1mm}
\begin{equation}\label{12.13-eq6.29}
\begin{cases}\ds
\epsilon_{\tau,\f}(r)=\int_0^1 \big(\f_{x}(r,\overline{X}(r)+\theta \xi_{\tau}(r),\bar{u}(r))-\bar{\f}_{x}(r)\big)\xi_{\tau}(r)d\theta, \\
\ns\ds  \tilde \epsilon_{\tau,\f}(r)=\int_0^1 (1-\theta)\xi_{\tau}(r)^*  \big(\f_{xx}(r,\overline{X}(r)+\theta \xi_{\tau}(r),\bar{u}(r))-\bar{\f}_{xx}(r)\big)\xi_{\tau}(r) d\theta.
\end{cases}\vspace{-1mm}
\end{equation}
Similar to \eqref{12.13-eq62} and \eqref{12.13-eq63}, the following inequalities hold for any $k\ge 1$,\vspace{-1mm}
\begin{equation}
\begin{cases}\ds
\mathbb{E}\(\int_{\tau}^T |\epsilon_{\tau,a}(r)|_H^{2k}dr|\mathcal{F}_t \)+\mathbb{E}\(\int_{\tau}^T |\epsilon_{\tau,b}(r)|_{\cL_2^0}^{2k}dr|\mathcal{F}_t \)  \leq \delta (|\tau-t|^k), \qquad \mathbb{P}\mbox{-a.s.},\\
\ns\ds    \mathbb{E}\(\int_{\tau}^T |\tilde \epsilon_{\tau,a}(r)|_H^{k}dr|\mathcal{F}_t \)+\mathbb{E}\(\int_{\tau}^T |\tilde \epsilon_{\tau,b}(r)|_{\cL_2^0}^{k}dr|\mathcal{F}_t \) \leq \delta (|\tau-t|^k), \qquad \mathbb{P}\mbox{-a.s.}, 
\end{cases}\vspace{-1mm}
\end{equation}
where $\delta:[0,\infty)\to [0,\infty)$ is a deterministic modulus of continuity satisfying $\delta(r)/r \to 0$ as $r \to 0^+$. 

The value function $V$ satisfies the fundamental inequality:\vspace{-1mm}
\begin{equation}\label{12.13-eq83}
V(\tau,\overline{X}(t))\leq \mathbb{E}\Big(\int_{\tau}^T f(r,X_{\tau}(r),\bar{u}(r))dr+h(X_{\tau}(T))\Big|\mathcal{F}_{\tau} \Big),\ \ \ \mathbb{P}\mbox{-a.s.}\vspace{-1mm}
\end{equation}
Applying the tower property of conditional expectation yields\vspace{-1mm}
\begin{equation}\label{12.13-eq6.31}
\mathbb{E}(V(\tau,\overline{X}(t))|\cF_t)\leq \mathbb{E}\Big(\int_{\tau}^T f(r,X_{\tau}(r),\bar{u}(r))dr+h(X_{\tau}(T))\Big|\mathcal{F}_t \Big),\ \ \ \mathbb{P}\mbox{-a.s.}\vspace{-1mm}
\end{equation}
From inequality \eqref{eq:3.16} and the regularity $\overline{X}(t) \in L_{\mathcal{F}_t}^2(\Omega,H)$, we deduce that\vspace{-1mm}
\begin{eqnarray*}
&&\Big|\mathbb{E}\big[V(\tau,\overline{X}(t))\big|\cF_t\big](\wt{\omega})-\mathbb{E}\big[V(\tau,\overline{X}(t,\wt{\omega}))\big|\cF_t\big](\wt{\omega})\Big|\\
&\le &\cC \mathbb{E}\big[\big(|\overline{X}(t)|_H+|\overline{X}(t,\wt{\omega})|_H\big)|\overline{X}(t)-\overline{X}(t,\wt{\omega})|_H\big|\cF_t\big](\wt{\omega})=0
\end{eqnarray*}
for almost every $\wt{\omega}\in \Omega$, which implies that\vspace{-1mm}
\begin{equation}\label{eq:732}
\mathbb{E}[V(\tau,\overline{X}(t)|\cF_t](\wt{\omega})=\mathbb{E}[V(\tau,\overline{X}(t,\wt{\omega}))|\cF_t](\wt{\omega}),\qquad\mathbb{P}\mbox{-a.s.}\vspace{-1mm}
\end{equation}

We select a measurable set $\Omega_0 \subset \Omega$ with full measure ($\mathbb{P}(\Omega_0)=1$) such that for every $\omega_0 \in \Omega_0$, the following properties hold:\vspace{-1mm}
\begin{equation}
\begin{cases}\ds
V(t,\omega_0,\overline{X}(t,\omega_0))=\mathbb{E}\Big(\int_t^T f(r,\overline{X}(r),\bar{u}(r))dr+h(\overline{X}(T))\Big| \cF_t\Big)(\omega_0),\nonumber\\
\textup{\eqref{12.13-eq6.25}, \eqref{12.13-eq6.29},\eqref{12.13-eq6.31} and \eqref{eq:732} are satisfied for any rational}\  \tau>t,\nonumber\\
\sup \limits_{s\leq r\leq T}|p(r,\omega_0)|<+\infty,\\
P(t,\omega_0)\in \cL(H),\ P(\cdot,\omega_0)\xi\in L^2(r,T;H), \ \forall \xi \in L_{\mathcal{F}_r}^2(\Omega;H), \ \forall r\in [0,T].
\end{cases}\vspace{-1mm}
\end{equation}
For a fixed $\omega_0 \in \Omega_0$, we define the conditional expectation $\mathbb{E}^t_{\omega_0}[\cdot] := \mathbb{E}(\cdot|\mathcal{F}_t)(\omega_0)$. Then for any rational $\tau > t$, we derive the following key estimate:\vspace{-1mm} 
\begin{eqnarray}
& & \mathbb{E}_{\omega_0}^tV(\tau,\overline{X}(t))-V(t,\omega_0,\overline{X}(t,\omega_0))\nonumber\\
& \leq\3n & \mathbb{E}^t_{\omega_0}\Big(-\int_t^{\tau}\bar{f}(r)dr+\int_{\tau}^T\big[f(r,X_{\tau}(r),\bar{u}(r))-\bar{f}(r)\big]dr+h(X_{\tau}(T))-h(\overline{X}(T))\Big)\nonumber\\
& =\3n & \mathbb{E}^t_{\omega_0}\Big(-\int_t^{\tau}\bar{f}(r)dr+\int_{\tau}^T \langle \bar{f}_x(r),\xi_{\tau}(r)\rangle_H dr+\langle h_x(\overline{X}(T)),\xi_{\tau}(T)\rangle_H  \\
& & \qq+\frac{1}{2}\int_{\tau}^T \langle \bar{f}_{xx}(r)\xi_{\tau}(r),\xi_{\tau}(r)\rangle_H dr+\frac{1}{2}\langle h_{xx}(\overline{X}(T))\xi_{\tau}(T),\xi_{\tau}(T)\rangle_H \Big)+o(|\tau-t|)\nonumber.
\end{eqnarray}

Similar to inequality \eqref{eq712}, and by \eqref{eq:732}, the following inequality holds:\vspace{-1mm}
\begin{eqnarray}\label{12.13-eq86}
& & \mathbb{E}_{\omega_0}^tV(\tau,\overline{X}(t,\omega_0))-V(t,\omega_0,\overline{X}(t,\omega_0))\nonumber\\
& \leq\3n   & -\mathbb{E}^t_{\omega_0}\int_t^{\tau}\bar{f}(r)dr-\mathbb{E}^t_{\omega_0}\(\langle p(\tau),\xi_{\tau}(\tau)\rangle_H - \frac{1}{2}\langle P(\tau)\xi_{\tau}(\tau),\xi_{\tau}(\tau)\rangle_H\)\nonumber\\
& &     -\mathbb{E}^t_{\omega_0} \int_{\tau}^T \big(\langle p(r),\tilde\epsilon_{\tau,a}(r)\rangle_H + \langle q(r), \tilde\epsilon_{\tau,b}(r)\rangle_{\mathcal{L}_2^0} \big)dr  -\frac{1}{2}\mathbb{E}^t_{\omega_0} \int_{\tau}^T \langle P(r)\epsilon_{\tau,a}(r),\xi_{\tau}(r)\rangle_H dr  \nonumber\\
& & -\mathbb{E}^t_{\omega_0} \int_{\tau}^T \big(\langle P(r)\xi_{\tau}(r),\epsilon_{\tau,a}(r)\rangle_H + \langle P(r)\bar{b}_{x}(r)\xi_{\tau}(r),\epsilon_{\tau,b}(r)\rangle_{\mathcal{L}_2^0}\big)dr\nonumber\\
& & -\mathbb{E}^t_{\omega_0} \int_{\tau}^T \langle P(r)\epsilon_{\tau,b}(r),\bar{b}_{x}(r)\xi_{\tau}(r)+\epsilon_{\tau,b}(r)\rangle_{\mathcal{L}_2^0}dr  \nonumber \\
& & -\mathbb{E}^t_{\omega_0} \int_{\tau}^T \big(\langle \epsilon_{\tau,b}(r),\hat{Q}^{(\tau)}(r)\rangle_{\mathcal{L}_2^0}+ \langle Q^{(\tau)}(r), \epsilon_{\tau,b}(r)\rangle_{\mathcal{L}_2^0}\big)dr + \ o(|\tau-t|)\nonumber\\
& =\3n   & -\mathbb{E}^t_{\omega_0}\int_t^{\tau}\bar{f}(r)dr-\mathbb{E}^t_{\omega_0}\( \langle p(\tau),\xi_{\tau}(\tau)\rangle_H - \frac{1}{2}\langle P(\tau)\xi_{\tau}(\tau),\xi_{\tau}(\tau)\rangle_H\) +o(|\tau-t|).
\end{eqnarray}

We now analyze the right-hand side of \eqref{12.13-eq86} term by term. First, for any $\phi,\phi'\in L_{\mathbb{F}}^2(0,T;H)$ and $\psi\in L_{\mathbb{F}}^2(0,T;\mathcal{L}_2^0)$, we have \vspace{-1mm}
\begin{eqnarray}\label{12.13-eq87}
\mathbb{E}^t_{\omega_0}\Big\langle \int_t^{\tau} \phi(r)dr,\int_t^{\tau} \phi'(r)dr\Big\rangle_H\3n &\leq\3n &  \Big(\mathbb{E}^t_{\omega_0}\Big|\int_t^{\tau} \phi(r)dr\Big|_H^2\Big)^{\frac{1}{2}} \Big(\mathbb{E}^t_{\omega_0}\Big| \int_t^{\tau} \phi'(r)dr\Big|_H^2\Big)^{\frac{1}{2}}\nonumber\\
& =\3n &(\tau-t)\Big(\int_t^{\tau}\mathbb{E}^t_{\omega_0}|\phi(r)|_H^2dr\int_t^{\tau}\mathbb{E}^t_{\omega_0}|\phi'(r)|_H^2dr\Big)^{\frac{1}{2}}\\
& =\3n & o(|\tau-t|), \hspace{1cm}\textup{as}\ \tau\downarrow t,\ \ \forall \ t\in [0,T),\ \ \mathbb{P}\mbox{-a.s}.\nonumber
\end{eqnarray}
Moreover, by the Lebesgue differentiation theorem and continuity of the filtration $t\mapsto \mathcal{F}_t$:\vspace{-1mm}
\begin{eqnarray}\label{12.13-eq88}
\mathbb{E}^t_{\omega_0}\Big\langle\! \int_t^{\tau}\! \phi(r)dr,\int_t^{\tau}\! \psi(r)dW(r)\Big\rangle_H \3n&\leq\3n&  \Big(\mathbb{E}^t_{\omega_0}\Big|\int_t^{\tau} \phi(r)dr\Big|_H^2\Big)^{\frac{1}{2}} \Big(\mathbb{E}^t_{\omega_0}\Big|\int_t^{\tau} \psi(r)dW(r)\Big|_{\mathcal{L}_2^0}^2\Big)^{\frac{1}{2}}\nonumber\\
& =\3n &(\tau-t)^{\frac{1}{2}}\Big(\!\int_t^{\tau}\!\mathbb{E}^t_{\omega_0}|\phi(r)|_H^2dr\!\int_t^{\tau}\!\mathbb{E}^t_{\omega_0}|\psi(r)|_{\mathcal{L}_2^0}^2dr\Big)^{\frac{1}{2}}\\
& =\3n & o(|\tau-t|), \hspace{1cm}\textup{as}\ \tau\downarrow t,\ \ a.e.\  t\in [0,T),\ \ \mathbb{P}\mbox{-a.s}.\nonumber
\end{eqnarray}
Combining \eqref{12.13-eq87} and \eqref{12.13-eq88}, we analyze the adjoint term: \vspace{-1mm}
\begin{eqnarray}\label{12.13-eq88-1}
& &\3n\3n \3n\3n\3n\3n   \mathbb{E}^t_{\omega_0} \langle p(\tau),\xi_{\tau}(\tau) \rangle_H =\mathbb{E}^t_{\omega_0}\big(\langle p(t),\xi_{\tau}(\tau)\rangle_H+\langle p(\tau)-p(t),\xi_{\tau}(\tau)\rangle_H\big)\nonumber\\
&  &\3n\3n \3n\3n\3n\3n   =  \mathbb{E}^t_{\omega_0}\Big\{\Big\langle p(t), -[S(\tau-t)-I]\overline{X}(t)-\int_t^{\tau} S(\tau-r)\bar{a}(r)dr-\int_t^{\tau} S(\tau-r)\bar{b}(r)dW(r)\Big\rangle_H\nonumber\\
& & +\Big\langle [S(\tau-t)-I]S(T-\tau)h_{x}(\overline{X}(T))-\int_t^{\tau}S(r-t)\big[\bar{a}_x(r)^*p(r)+\bar{b}_x(r)^*q(r)-\bar{f}_x(r)\big]dr\nonumber\\
& &   -\int_{\tau}^T[S(\tau-t)-I]S(r-\tau)\big[\bar{a}_x(r)^*p(r) +\bar{b}_x(r)^*q(r) -\bar{f}_x(r)\big]dr\\
& &  
+ \int_t^{\tau}S(r-t)q(r)dW(r)-\int_{\tau}^T[S(\tau-t)-I]S(r-\tau)q(r)dW(r),\nonumber\\
& &   -[S(\tau-t)-I]\overline{X}(t)-\int_t^{\tau} S(\tau-r)\bar{a}(r)dr
-\int_t^{\tau} S(\tau-r)\bar{b}(r)dW(r)\Big\rangle_H\Big\}\nonumber\\
&  &\3n\3n \3n\3n \3n\3n  =\mathbb{E}^t_{\omega_0}\[-\Big\langle A^*p(t),(\tau-t) \overline{X}(t)- \Big\langle p(t),\int_t^{\tau} S(\tau-r)\bar{a}(r)dr\Big\rangle_H \nonumber\\
&  & -\int_t^{\tau} \lan S(r-t)q(r), S(\tau-r)\bar{b}(r)\ran_{\cL_2^0}dr\] +o(|\tau-t|).\nonumber 
\end{eqnarray}
Finally, using the definition of $\xi_{\tau}(\tau)$, we estimate the second-order term:\vspace{-1mm}
\begin{eqnarray}\label{12.13-eq90}
& & \mathbb{E}^t_{\omega_0}\langle P(\tau)\xi_{\tau}(\tau),\xi_{\tau}(\tau)\rangle_H\nonumber\\
& =\3n &  \mathbb{E}^t_{\omega_0}\Big\langle P(\tau)\Big\{[S(\tau-r)-I]\overline{X}(t)-\int_t^{\tau}S(\tau-r)\bar{a}(r)dr-\int_t^{\tau}S(\tau-r)\bar{b}(r)dW(r)\Big\},\nonumber\\
& & \q\; [S(\tau-r)-I]\overline{X}(t)-\int_t^{\tau}S(\tau-r)\bar{a}(r)dr-\int_t^{\tau}S(\tau-r)\bar{b}(r)dW(r)\Big\rangle_H\nonumber\\
& =\3n & \mathbb{E}^t_{\omega_0}\int_t^{\tau} \langle P(\tau)\bar{b}(r),\bar{b}(r)\rangle_{\mathcal{L}_2^0} dr+o(|\tau-t|)\\
& =\3n &\mathbb{E}^t_{\omega_0}\int_t^{\tau} \langle P(t)\bar{b}(t),\bar{b}(t)\rangle_{\mathcal{L}_2^0} dr + \mathbb{E}^t_{\omega_0}\int_t^{\tau}  \langle P(\tau)\bar{b}(r),\bar{b}(r)-\bar{b}(t)\rangle_{\mathcal{L}_2^0} dr\nonumber\\
& & +\mathbb{E}^t_{\omega_0}\int_t^{\tau} \big\langle P(\tau)\big(\bar{b}(r)-\bar{b}(t)\big),\bar{b}(t)\big\rangle_{\mathcal{L}_2^0} dr+  \mathbb{E}^t_{\omega_0}\int_t^{\tau} \big\langle \big(P(\tau)-P(t)\big)\bar{b}(t),\bar{b}(t)\big\rangle_{\mathcal{L}_2^0} dr+o(|\tau-t|).\nonumber
\end{eqnarray}

We now analyze the asymptotic behavior of terms in \eqref{12.13-eq90} as $\tau \downarrow t$. 

First, since $P(\cdot)\bar{b}(r) \in D_{\mathbb{F}}([r,T]; L^{4/3}(\Omega;\mathcal{L}_2^0))$, we have \vspace{-1mm}
\begin{eqnarray}\label{12.13-eq91}
& & \mathbb{E}^t_{\omega_0}\int_t^{\tau}  \langle P(\tau)\bar{b}(r),\bar{b}(r)-\bar{b}(t)\rangle_{\mathcal{L}_2^0} dr\nonumber\\
& \leq\3n & |P(\cdot)\bar{b}(\cdot)|_{D_{\mathbb{F}}([t,T];L^{4/3}(\Omega;\mathcal{L}_2^0))}\int_t^{\tau}\big(\mathbb{E}^t_{\omega_0}|\bar{b}(r)-\bar{b}(t)|_{\mathcal{L}_2^0}^4\big)^{\frac{1}{4}}dr\\
& =\3n & o(|\tau-t|), \hspace{1cm}  \textup{as}\ \tau\downarrow t,\ \ a.e.\  t\in [0,T).\nonumber
\end{eqnarray}
Similarly, we obtain that\vspace{-4mm}
\begin{eqnarray}\label{12.13-eq92}
& & \mathbb{E}^t_{\omega_0}\int_t^{\tau} \big\langle P(\tau)\big(\bar{b}(r)-\bar{b}(t)\big),\bar{b}(t)\big\rangle_{\mathcal{L}_2^0} dr\nonumber\\
& \leq\3n & \int_t^\tau \big[\mathbb{E}^t_{\omega_0}\big|P(\tau)\big(\bar{b}(r)-\bar{b}(t)\big)\big|_{\mathcal{L}_2^0}^{\frac{4}{3}}\big]^{\frac{3}{4}}dr|\bar{b}(t)|_{L^4_{\mathcal{F}_t}(\Omega;\mathcal{L}_2^0)}\\
& =\3n & o(|\tau-t|), \hspace{1cm}  \textup{as}\ \tau\downarrow t,\ \ a.e.\  t\in [0,T).\nonumber
\end{eqnarray}
By the definition of $\mathcal{P}[0,T]$, for any $\xi \in L_{\mathcal{F}_t}^4(\Omega;H)$, the mapping $P(\cdot)\xi$ is right continuous in $L_{\mathcal{F}_t}^4(\Omega;H)$. This implies that\vspace{-4mm}
\begin{eqnarray}\label{12.13-eq93}
& & \mathbb{E}^t_{\omega_0}\int_t^{\tau} \big\langle \big(P(\tau)-P(t)\big)\bar{b}(t),\bar{b}(t)\big\rangle_{\mathcal{L}_2^0} dr\nonumber\\
& \leq\3n & (\tau-t)\big|(P(\tau)-P(t)\big)\bar{b}(t)\big|_{L^{4/3}_{\mathcal{F}_t}(\Omega;\mathcal{L}_2^0)}\big|\bar{b}(t)|_{L^4_{\mathcal{F}_t}(\Omega;\mathcal{L}_2^0)}\nonumber\\
& =\3n & o(|\tau-t|), \q\textup{as}\ \tau\downarrow t, \nonumber
\end{eqnarray}
Combining estimates  \eqref{12.13-eq90}--\eqref{12.13-eq93} yields\vspace{-1mm}
\begin{eqnarray}\label{12.13-eq93-1}
&&\mathbb{E}^t_{\omega_0}\langle P(\tau)\xi_{\tau}(\tau),\xi_{\tau}(\tau)\rangle_H 
\nonumber \\
& =\3n & \mathbb{E}^t_{\omega_0}\int_t^{\tau} \langle P(t)\bar{b}(t),\bar{b}(t)\rangle_{\mathcal{L}_2^0} dr +o(|\tau-t|)\\
& = \3n& (\tau\!-t)\langle P(t)\bar{b}(t),\bar{b}(t)\rangle_{\mathcal{L}_2^0} +o(|\tau-t|),  \q\textup{as}\ \tau\downarrow t,\ \ a.e.\  t\in [0,T).\nonumber
\end{eqnarray}
From \eqref{12.13-eq86}, \eqref{12.13-eq88-1} and \eqref{12.13-eq93-1}, we conclude that for any rational $\tau > t$ and $\omega_0\in \Omega_0$:
\begin{eqnarray*}
& &\3n\3n\3n\3n\3n\3n \mathbb{E}_{\omega_0}^t V(\tau,\overline{X}(t,\omega_0))-V(t,\omega_0,\overline{X}(t,\omega_0))\nonumber\\
&&\3n\3n\3n\3n\3n\3n \leq \mathbb{E}^t_{\omega_0}\Big(\big\langle p(t),\big[S(\tau-t)-I\big]\overline{X}(t)\big\rangle_H+\Big\langle p(t),\int_t^{\tau}S(\tau-r)\bar{a}(r)dr\Big\rangle_H\nonumber\\
& & +\int_t^{\tau}\! \lan S(r\!-\!t)q(r), S(\tau\!-\!r)\bar{b}(r)\ran_{\cL_2^0}dr-\frac{1}{2}(\tau\!-t)\langle P(t)\bar{b}(t),\bar{b}(t)\rangle_{\mathcal{L}_2^0} -\int_t^{\tau}\!\bar{f}(r)dr\Big) +o(|\tau\!-t|)\nonumber\\
&&\3n\3n\3n\3n\3n\3n \leq (\tau-t)\Big\{\langle\!\langle A\overline{X}(t,\omega_0),p(t,\omega_0)\rangle\!\rangle-\mathbb{H}(t,\overline{X}(t,\omega_0),\overline{u}(t,\omega_0),-p(t,\omega_0),-q(t,\omega_0),P(t,\omega_0))\Big\}+o(|\tau-t|),
\end{eqnarray*}
which completes the proof of the theorem.
\end{proof}
%


\section{Illustrative examples }\label{sec-exam}

In this section, we present two illustrative examples which fulfill
the assumptions in Theorems \ref{thm-sdy1}, \ref{th5.1}   and/or \ref{th6.1}.

Consider a complete filtered probability space $(\widehat\Omega, \widehat{\mathcal{F}}, \{\widehat{\mathcal{F}}_t\}_{t\in [0,T]}, \widehat{\mathbb{P}})$, on which a standard one dimensional Brownian motion
$\widehat W(\cdot)$ is defined and $\widehat{\mathbf{F}}\deq \{\widehat{\cF}_t\}_{t\in [0,T]}$ is
the natural filtration generated by $\widehat W(\cdot)$.

Let $G\subset\dbR^n$ be a bounded domain with the smooth boundary
$\pa G$.  Let $H=L^2(G)$ and $U$  be a bounded closed subset of
$L^2(G)$. Consider the following stochastic parabolic equation:
\begin{equation}\label{system3}
\begin{cases}
\ds   dy =\big(\Delta y + \hat a(t,y,u)\big) dt+\hat b(t,y,u) d\widehat W(t) &\textup{in }  (0,T]\times G,\\
\ns\ds   y =0 &\textup{on }   (0,T]\times \pa G,\\
\ns\ds   y(0)=y_0 &\textup{in }  G,
\end{cases}
\end{equation}
where $y_0 \in L^2(G)$, $u(\cdot)\in \cU[0,T]$, and $\hat a$ and
$\hat b$  satisfy the following condition:

\ss

\no{\bf (B1)} {\it For $\f=\hat a,\hat b$, suppose that
$\f(\cd,\cd,\cd,\cd):[0,T]\times\widehat\Om\times \dbR\times \dbR\to
\dbR$ satisfies : i) For any $(r,u)\in
\dbR\times \dbR$, the function
$\f(\cd,\cd,r,u):[0,T] \to \dbR$ is $\widehat{\mathbf{F}}$-adapted; ii) For any $r\in
\dbR$ and a.e. $\hat\om\in\widehat\Om$, the function $\f(\cdot,\hat\om,r,\cd):[0,T]\times\dbR\to \dbR$
is continuous; and iii) For all
$(t,r_1,r_2,u_1,u_2)\in [0,T]\times \dbR\times
\dbR\times \dbR$,\vspace{-1mm}
\begin{equation}\label{ab0}
\left\{
\begin{array}{ll}\ds
	|\f(t,r_1,u_1) - \f(t,r_2,u_2)|   \leq
	\cC (|r_1-r_2| + |u_1-u_2|), \q \wt{\dbP}\mbox{-a.s.},\\
	\ns\ds |\f(t,0,u)| \leq \cC(1+|u|), \q \wt{\dbP}\mbox{-a.s.};
\end{array}
\right.
\end{equation}
iv)  For all
$(t,u)\in [0,T]\times  \dbR$ and a.e. $\hat\om\in\widehat\Om$,
$\f(t,\hat\om,\cd,u)$ is $C^2$, and for any $(r,u)\in
\dbR\times \dbR$ and a.e. $(t,\hat\om)\in [0,T]\times\widehat\Om$,\vspace{-1mm}
\begin{equation*}\label{ab1}
|\f_r(t,\hat\om,r,u)| + |\f_{rr}(t,\hat\om,r,u)|
\leq \cC.
\end{equation*}}

Consider the following
cost functional:\vspace{-1mm}
\begin{equation}\label{cost3}
\mathcal{J}(y_0;u(\cdot))= \mathbb{E}\Big(\int_0^T\int_G \hat
f(t,y(t),u(t))dxdt+\int_G \hat h(y(T))dx \Big),
\end{equation}
where $\hat f$ and $\hat h$ satisfy the following condition:

\ss

\no{\bf (B2)} {\it For all $(r,u)\in\dbR\times\dbR$ and a.e. $\hat\om\in\widehat\Om$, $\hat f(\cdot,\hat\om,r,u)$ is continuous;  for all $(t,u)\in [0,T]\times\dbR$ and a.e. $\hat\om \in\widehat\Om$,  $\hat
f(t,\hat\om,\cd,u)$ and $\hat h(\hat\om,\cd)$ are $C^2$, such that $\hat
f_r(t,\hat\om, r,\cd)$ and $\hat f_{rr}(t,\hat\om,r,\cd)$ are continuous, and for
any $(r,u)\in \dbR\times \dbR$ and a.e. $(t,\hat\om)\in [0,T]\times\widehat\Om$,
\begin{equation*}\label{ab1-1}
|\hat f_r(t,\hat\om,r,u)| +|\hat h_r(\hat\om,r) |+|\hat f_{rr}(t,\hat\om,r,u)|
+|\hat h_{rr}(\hat\om,r)| \leq \cC.
\end{equation*}
}

Under {\bf (B1)} and {\bf (B2)}, it is easy to see that ({\bf
S1})--({\bf S3}) hold. Then we know that all assumptions in Theorem
\ref{th5.1} are fulfilled.  By the regularity of backward stochastic
parabolic equations (e.g.,\cite{Du2012}), we know that $A^*p(t)\in
L^2(G)$ for a.e. $(t,\hat\om)\in (0,T)\times \widehat\Om$. Hence,  all
assumptions in Theorems \ref{thm-sdy1} and \ref{th6.1} are fulfilled.

\ss

Next, let $H=H_0^1(G)\times L^2(G)$ and $U$  be a bounded closed
subset of $L^2(G)$. Consider the following stochastic hyperbolic
equation:
\begin{equation}\label{system4}
\begin{cases}
\ds   dy_t =\big(\Delta y + \hat a(t,y,u)\big) dt+\hat b(t,y,u) dW(t) &\textup{in }  (0,T]\times G,\\
\ns\ds   y =0 &\textup{on }   (0,T]\times \pa G,\\
\ns\ds   y(0)=y_0,\; y_t(0)=y_1 &\textup{in }  G.
\end{cases}
\end{equation}
where $ (y_0,y_1)\in H_0^1(G)\times L^2(G)$, $u(\cdot)\in \cU[0,T]$, and $\hat a$ and $\hat b$ satisfy {\bf (B1)}.

Consider the following cost functional: 
\begin{equation}\label{cost4}
\mathcal{J}(y_0,y_1;u(\cdot))= \mathbb{E}\Big(\int_0^T\int_G
\hat f(t,y(t),u(t))dxdt+\int_G \hat h(y(T))dx \Big),
\end{equation}
where $\hat f$ and $\hat h$ satisfy  {\bf (B2)}. Under {\bf (B1)} and {\bf (B2)}, it is easy to see that ({\bf S1})--({\bf S3}) hold. Then we know that all assumptions in Theorems \ref{thm-sdy1} and \ref{th5.1} are fulfilled.


\section{Appendix}

In this section, we derive an It\^o-Kunita formula for 
$H$-valued It\^o processes, extending the classical It\^o-Kunita formula for  
$\dbR^n$-valued It\^o processes. In addition to proving Proposition \ref{prop:b1}, we highlight the intrinsic theoretical significance of this result.

As preparation, we first give the following proposition similar to \cite[Proposition 1.132]{Fabbri2017} under the assumption that the generator $A$ generates a generalized contraction semigroup (introduced in Section 1).

\begin{proposition}\label{prop51}
Let $p\ge 2$ and $n\in \mathbb{N}$, let $\wt{\xi} \in L_{\mathcal{F}_0}^{p}(\Omega;H)$, $\wt{a} \in L_{\mathbb{F}}^{p}(\Omega;L^{\infty}(0,T;H))$ and $\wt{b} \in L_{\mathbb{F}}^p(\Omega; L^{\infty}(0,T; \mathcal{L}_2^0))$. Let $\wt{X}(\cdot)$ be the mild solution of\vspace{-1mm}
\begin{equation}\label{eeq1}
\begin{cases}
d\wt{X}(s) = (A\wt{X}(s) + \wt{a}(s))\,ds + \wt{b}(s)dW(s),\qq s\in (0,T], \\ 
\ns\ds \wt{X}(0) = \xi,
\end{cases}\vspace{-1mm}
\end{equation}
and $\wt{X}^{n}(\cdot)$ be the solution of\vspace{-1mm}
\begin{equation}\label{eeq2}
\begin{cases}\ds
d\wt{X}^{n}(s) = (A_{n}\wt{X}^{n}(s) + \wt{a}(s))\,ds + \wt{b}(s)dW(s),\qq s\in (0,T], \\ 
\ns\ds\wt{X}^{n}(0) = \xi,
\end{cases} \vspace{-1mm}
\end{equation}
where $A_{n} \deq nA(nI-A)^{-1}$ is the Yosida approximation of $A$. Then $X_n\rightarrow X$ in $L_{\mathbb{F}}^p(\Omega;C([0,T];H))$.
\end{proposition}

\begin{proof}
By assumptions on $\wt{\xi}$, $\wt{a}$ and $\wt{b}$, the mild solution of $X$ is well defined and satisfies\vspace{-1mm}
$$X(t) = S(t)\xi + \int_{0}^{t} S(t-s)\wt{a}(s)ds + \int_{0}^{t} S(t-s)\wt{b}(s)dW(s), \qquad t \in [0, T].\vspace{-1mm}$$
The same is true for the mild solution of \eqref{eeq2} (which is also a strong solution). Fix $t \in [0, T]$,\vspace{-1mm}
$$
\begin{aligned}
X^n(t) - X(t) &= \big[S_n(t) - S(t)\big] \wt{\xi} + \int_0^t \big[S_n(t-s) - S(t-s)]\wt{a}(s) \, ds \\
&\quad + \int_0^t \big[S_n(t-s) - S(t-s)\big] \wt{b}(s)dW(s) \\
&=: I_1^n(t) + I_2^n(t) + I_3^n(t).
\end{aligned}\vspace{-1mm}
$$

It suffices to prove $\lim\limits_{n \to \infty} \mathbb{E}\big[ \sup_{s \in [0,T]} \left| I_{i}^{n}(s) \right|_H^{p} \big] = 0$ for $i = 1, 2, 3$.

Since $A$ generates a generalized contractive $C_0$-semigroup, it follows from \cite[Proposition 1.112]{Fabbri2017} that\vspace{-1mm}
$$
\lim_{n \to \infty} \mathbb{E}\Big[ \sup_{s \in [0,T]} \left| I_{3}^{n}(s) \right|_H^{p} \Big] = 0.\vspace{-1mm}
$$
Meanwhile, for $i = 1, 2$, the result follows similarly, as the proof parallels that of \cite[Proposition 1.132]{Fabbri2017}, thus completing the proof. 
\end{proof}

Next, we fix the regularity of $\wt{a},\wt{b}$, let $\widetilde{a} \in L_{\mathbb{F}}^2(\Omega; L^{\infty}(0,T; H))$ and $\widetilde{b} \in L_{\mathbb{F}}^4(\Omega; L^{\infty}(0,T; \mathcal{L}_2^0))$. Denote by $\wt X$ the mild solution of the following SEE:\vspace{-1mm}
\begin{equation}\label{8.15-eq1}
\begin{cases}\ds d\wt X(t) = \big( A\wt X(t) + \widetilde{a}(t) \big) dt + \widetilde{b}(t)  dW(t), \quad t \in [0,T],\\
\ns\ds \wt X(0)=X_0\in H.
\end{cases}\vspace{-1mm}
\end{equation}

Assume the stochastic field $F$  defined on $[0,T]\times\Om \times H$ satisfies the SEE\vspace{-1mm}
$$
dF(t, x) = \Gamma(t, x)  dt + \Phi(t, x)  dW(t), \quad (t,x) \in [0,T] \times H,\vspace{-1mm}
$$
where $\Gamma \in \mathbb{S}_{\mathbb{F}}^{1}([0,T] \times H)$ and $\Phi \in \mathbb{S}_{\mathbb{F}}^{1}([0,T] \times H, \widetilde{H})$ are continuously differentiable in the spatial variable. Furthermore, assume that the field $F$ itself belongs to $\mathbb{C}_{\mathbb{F}}^{0,2}([0,T] \times H)$, indicating it is twice continuously differentiable in the spatial variable and continuous in the time variable, and the adjoint composition $A^*F_x \in \mathbb{S}_{\mathbb{F}}^{0}([0,T] \times H)$ establishes an important relationship between the field's spatial derivatives and the system's generator $A$. This specification ensures the necessary regularity for applying stochastic calculus in infinite-dimensional spaces, with the evolution equation capturing both the deterministic drift through $\G$ and random fluctuations through $\Phi$ as the field evolves over time and space within the domain $[0,T] \times H$.
 
Further, we assume  that there exist a positive-valued process $L(\cd) \in L_{\mathbb{F}}^4(\Omega,L^2(0,T;\mathbb{R}))$ and $k \in \mathbb{N}$ such that for almost every $(t,\omega) \in [0,T] \times \Omega$, the following growth conditions hold:\vspace{-1mm}
$$
\begin{aligned}
&|\Gamma(t,x)| + |\Phi(t,x)|_{\widetilde{H}} + |\Gamma_x(t,x)|_H + |\Phi_x(t,x)|_{\mathcal{L}_2^0} \leq L(t)(1 + |x|_H^k), \\
&|F(t,x)| + |F_x(t,x)|_H + |A^*F_x(t,x)|_H + |F_{xx}(t,x)|_{\mathcal{L}(H)}  \leq L(t)(1 + |x|_H^k).
\end{aligned}\vspace{-1mm}
$$
\begin{lemma}\label{Ito-Kunita} Under the above assumptions, for all $t \in [0,T]$, 
\begin{eqnarray}\label{iw}
F(t, \wt X(t))\3n &=&\3nF(0, \wt X(0))+ \int_0^t\Big[ \Gamma(s, \wt X(s)) + \langle A^* F_x(s, \wt X(s)),\wt X(s)\rangle_{H}+\langle F_x(s, \wt X(s)),\wt{a}(s)\rangle_{H}\nonumber\\
& &\3n+ \frac{1}{2} \left<F_{xx}(t,\wt X(s))\wt{b}(s),\wt{b}(s)\right>_{\mathcal{L}_2^0}+ \wt{b}(s)^*\Phi_x(s, \wt X(s))\Big] ds \nonumber\\
& &\3n+\Big[\Phi(s,\wt  X(s)) + \wt{b}(s)^* F_x(s, \wt X(s))\Big]dW(s),\quad a.s.
\end{eqnarray}
\end{lemma}

\begin{proof} We divide the proof  into three steps.
	
\textbf{Step 1.} Following the approaches in \cite{Krylov2011}, we first establish the It\^o-Kunita formula \eqref{iw} for the case $A=0$ and $H=\mathbb{R}^n$ (for any fixed $n\in \mathbb{N}$) driven by cylindrical Brownian motion. 

Let $\phi \in C_c^\infty(\mathbb{R}^n,\mathbb{R})$ be a nonnegative, compactly supported function with $\operatorname{supp}(\phi) \subseteq B_{\mathbb{R}^n}(0,1)$ and $\int_{\mathbb{R}^n} \phi(x) dx = 1$. For $\epsilon > 0$, define the mollifier $\phi_\epsilon(x) := \epsilon^{-n} \phi(x/\epsilon)$. Applying It\^o's formula to $\phi_\epsilon(X(t)-x)$ gives 
$$
\begin{aligned}
\phi_\epsilon(\wt X(t)-x)=&\phi_\epsilon(\wt X(0)-x)+\int_0^t\langle\partial_x\phi_\epsilon(\wt X(s)-x),\wt{a}(s)\rangle_H
    ds+\int_0^t\wt{b}(s)^*\partial_x\phi_\epsilon(\wt X(s)-x)dW(s)\\
    &+\frac{1}{2}\int_0^t\langle\partial_x^2\phi_\epsilon(\wt X(s)-x)\wt{b}(s),\wt{b}(s)\rangle_{\cL_2^0} ds.
\end{aligned}
$$
Next, applying It\^{o}'s formula to $F(t,x)\phi_\epsilon(\wt X(t)-x)$ yields for all $t\in[0,T]$, a.s.: 
\begin{eqnarray}\label{c6}
      &&\quad F(t,x)\phi_\epsilon(\wt X(t)-x)\nonumber \\
      &&=F(0,x)\phi_\epsilon(\wt X_0-x) + \int_0^t \phi_\epsilon(\wt X(s)-x) \Gamma(s,x)  ds + \int_0^t \phi_\epsilon(\wt X(s)-x) \Phi(s,x)  dW(s) \nonumber \\
      &&\quad + \frac{1}{2} \int_0^t F(s,x) \langle \partial_x^2 \phi_\epsilon (\wt X(s)-x) \wt{b}(s), \wt{b}(s) \rangle_{\mathcal{L}_2^0}  ds + \int_0^t \Phi(s,x) \wt{b}(s)^* \partial_x \phi_\epsilon (\wt X(s)-x)  ds \nonumber \\
      &&\quad+ \int_0^t F(s,x) \wt{b}(s)^* \partial_x \phi_\epsilon (\wt X(s)-x)  dW(s)+\int_0^t F(s,x) \langle \partial_x \phi_\epsilon (\wt X(s)-x), \wt{a}(s) \rangle_H  ds. 
\end{eqnarray}
Under our assumptions on $\wt{a}$ and $\wt{b}$, we have 
$$ |\wt X(t)|_{C([0,T]; H)},\; |\wt{a}|_{L^{\infty}(0,T; H)},\;|\wt{b}|_{L^{\infty}(0,T; \cL_2^0)}<\infty,\quad \mbox{a.s.,}$$
ensuring all terms in \eqref{c6} are almost surely finite. For $r\in\mathbb{N}$, set $B_r:=\{x\in\mathbb{R}^n:|x|_H<r\}$. The growth conditions imply
\begin{eqnarray}\label{cc7}
&&\quad\int_0^T\int_{B_r}|\phi_\epsilon(\wt X(s)-x)\Phi(s,x)|_H^2+|F(s,x)\wt{b}(s)^*\partial_x\phi_\epsilon(\wt X(s)-x)|_H^2dxds\nonumber\\
&&\leq \int_0^T\int_{\mathbb{R}^n}\Big[|\phi_\epsilon(\wt X(s)-x)|_H^2+|\wt{b}(s)|_{\cL_2^0}^2|\partial_x\phi_\epsilon(\wt X(s)-x)|_H^2\Big]L(s)^2\big(1+|x|_H^k\big)^2dxds\nonumber\\
&&\leq \[\sup_{s\in [0,T]}\big(1+|\wt X(s)|_H^k+\epsilon\big)^2\int_0^TL(s)^2ds\]\int_{\mathbb{R}^n}|\phi_\epsilon(x)|^2dx\nonumber\\
&&\quad+\[\sup_{s\in [0,T]}\big(1+|\wt X(s)|_H^k+\epsilon\big)^2|\wt{b}|_{L^{\infty}(0,T; \cL_2^0)}^2\int_0^TL(s)^2ds\]\int_{\mathbb{R}^n}|\partial_x\phi_\epsilon(x)|_H^2dx\nonumber\\
&&<\infty,\ \ \text{a.s.} 
\end{eqnarray}
%

Integrating \eqref{c6} over $B_r$ and applying Fubini's theorem and the stochastic Fubini theorem (Theorem 2.141 in \cite{Lu2021}) yields  
\begin{eqnarray}\label{c7}
&&\quad\int_{B_r} F(t,x) \phi_\epsilon (\wt X(t) - x)  dx\nonumber\\
&&=\int_{B_r} F(0,x) \phi_\epsilon (\wt X_0 - x)  dx + \int_0^t \int_{B_r} \phi_\epsilon (\wt X(s) - x) \Gamma (s,x)  dx  ds  \nonumber\\
&&\quad + \int_0^t \int_{B_r} \phi_\epsilon (\wt X(s) - x) \Phi (s,x)  dx  dW(s)+ \frac{1}{2} \int_0^t \int_{B_r} F(s,x) \Big\langle \partial_x^2 \phi_\epsilon (\wt X(s) - x) \wt{b}(s) , \wt{b}(s) \Big\rangle_{\mathcal{L}_2^0}  dx  ds  \nonumber\\
&&\quad+ \int_0^t \int_{B_r} \Phi (s,x) \wt{b}(s)^{*} \partial_x \phi_\epsilon (\wt X(s) - x)  dx  ds+\int_0^t \int_{\mathbb{R}^n} F(s,x) \wt{b}(s)^{*} \partial_x \phi_\epsilon (\wt X(s) - x)  dx  dW(s) \nonumber\\
&&\quad+ \int_0^t \Big\langle \int_{B_r} F(s,x) \partial_x \phi_\epsilon (\wt X(s) - x)  dx , \wt{a}(s) \Big\rangle_H ds.
\end{eqnarray}

From inequality \eqref{cc7} and the  Dominated Convergence Theorem for $H$-valued process (see \cite[Theorem 2.18]{Lu2021}), as $r\rightarrow\infty$:
$$\int_0^T\Big|\int_{B_r}\phi_\epsilon(\wt X(s)-x)\Phi(s,x)dx-\int_{\mathbb{R}^n}\phi_\epsilon(\wt X(s)-x)\Phi(s,x)dx\Big|_H^2ds\rightarrow 0,\text{ a.s.,}$$
thus converges in probability, which implies
$$\int_0^T\int_{B_r}\phi_\epsilon(\wt X(s)-x)\Phi(s,x)dxdW(s)\rightarrow\int_0^T\int_{\mathbb{R}^n}\phi_\epsilon(\wt X(s)-x)\Phi(s,x)dxdW(s)$$
in probability. Following similar arguments to take $r\to\infty$ in \eqref{c7}, we obtain that
\begin{eqnarray}\label{c7-1}
&&\quad\int_{\mathbb{R}^n} F(t,x) \phi_\epsilon (\wt X(t) - x)  dx\nonumber\\
&&=\int_{\mathbb{R}^n} F(0,x) \phi_\epsilon (\wt X_0 - x)  dx + \int_0^t \int_{\mathbb{R}^n} \phi_\epsilon (\wt X(s) - x) \Gamma (s,x)  dx  ds  \nonumber\\
&&\quad + \int_0^t \int_{\mathbb{R}^n} \phi_\epsilon (\wt X(s) - x) \Phi (s,x)  dx  dW(s)+ \frac{1}{2} \int_0^t \int_{\mathbb{R}^n} F(s,x) \Big\langle \partial_x^2 \phi_\epsilon (\wt X(s) - x) \wt{b}(s) , \wt{b}(s) \Big\rangle_{\mathcal{L}_2^0}  dx  ds  \nonumber\\
&&\quad+ \int_0^t \int_{\mathbb{R}^n} \Phi (s,x) \wt{b}(s)^{*} \partial_x \phi_\epsilon (\wt X(s) - x)  dx  ds+\int_0^t \int_{\mathbb{R}^n} F(s,x) \wt{b}(s)^{*} \partial_x \phi_\epsilon (\wt X(s) - x)  dx  dW(s) \nonumber\\
&&\quad+ \int_0^t \Big\langle \int_{\mathbb{R}^n} F(s,x) \partial_x \phi_\epsilon (\wt X(s) - x)  dx , \wt{a}(s) \Big\rangle_H ds.
\end{eqnarray}

Integration by parts gives
\begin{eqnarray}\label{c7-2}
&&
\int_{\mathbb{R}^n}F(s,x)\langle\partial_x^2\phi_\epsilon(\wt X(s)-x)\wt{b}(s),\wt{b}(s)\rangle_{\cL_2^0}
dx =\int_{\mathbb{R}^n}\phi_\epsilon(\wt X(s)-x)\langle\partial_x^2F(s,x)\wt{b}(s),\wt{b}(s)\rangle_{\cL_2^0} dx,\nonumber\\
&&\int_{\mathbb{R}^n}\Phi(s,x)\wt{b}(s)^{*}\partial_x\phi_\epsilon(\wt X(s)-x)dx =\int_{\mathbb{R}^n}\phi_\epsilon(\wt X(s)-x)\wt{b}(s)^*\partial_x
\Phi(s,x)dx,\\
&&\int_{\mathbb{R}^n}
F(s,x)\partial_x\phi_\epsilon(\wt X(s)-x)dx =\int_{\mathbb{R}^n}\phi_\epsilon(\wt X(s)-x)\partial_x
F(s,x)dx. \nonumber
\end{eqnarray}
Finally, taking the limit $\epsilon\rightarrow0$ in \eqref{c7-1} and noting \eqref{c7-2}, we obtain \eqref{iw}.

\ms

\textbf{Step 2.} In this step, we establish \eqref{iw} for $A=0$ in general separable Hilbert spaces through finite-dimensional projections. Let $\{e_i\}_{i=1}^\infty$ be a complete orthonormal basis of $H$, and define for each $n\in\mathbb{N}$ the subspace $H_n := \text{span}\{e_1,...,e_n\}$ with projection $P_n:H\to H_n$. The projected process satisfies:
$$
\wt X^{n}(t) =P_n\wt X_0+\int_0^t P_n \wt{a}(s)ds + \int_0^tP_n \wt{b}(s)dW(s).
$$
Define stopping times:
$$\tau^R:=\inf\{s\in [0,T]:|\wt X(s)|_H>R\},\quad \tau_n^R:=\inf\{s\in [0,T]:|\wt X^n(s)|_H>R+1,|X(s)|_H>R\}.$$ 
There exists a subsequence (still denoted $X^n$) converging uniformly to $X$ on $\tilde{\Omega}\subset\Omega$ with $\mathbb{P}(\tilde{\Omega})=1$, yielding 
$$\lim_{n\rightarrow \infty}\tau_{n}^R=\tau^R\q \mbox{ in } \wt{\Omega},$$
with pointwise convergence:
$$\lim_{n\rightarrow \infty}\mathbf{1}_{[0,s\wedge \tau_n^R]}=\mathbf{1}_{[0,s\wedge \tau^R]}\q \text{ for } s\in [0,T].$$
Define the projected stochastic field:
$$
F(t, P_n x) \deq F(0, P_n x) + \int_0^t \Gamma(s, P_n x)ds + \int_0^t \Phi(s, P_n x)dW(s).
$$
Applying the finite-dimensional It\^{o}-Kunita formula yields
\begin{eqnarray}\label{FPn}
F(t\wedge \tau_n^R,\wt X^{n}(t\wedge \tau_n^R)) &\3n=\3n& F(0,\wt X^{n}(0)) + \int_0^t\mathbf{1}_{[0,s\wedge \tau_n^R]} \bigg[ \Gamma(s, \wt X^{n}(s)) + \langle P_n\wt{a}(s), F_x(s, \wt X^{n}(s)) \rangle_{H} \nonumber\\
&\3n\3n&+ \frac{1}{2} \langle F_{xx}(s, \wt X^{n}(s))P_n\wt{b}(s), P_n\wt{b}(s) \rangle_{\mathcal{L}_2^0}
+ \langle \Phi_x(s, \wt X^{n}(s)), P_n\wt{b}(s) \rangle_{\mathcal{L}_2^0} \bigg] ds \nonumber\\
&\3n\3n&+ \int_0^t\mathbf{1}_{[0,s\wedge \tau_n^R]} \bigg[ \Phi(s, \wt X^{n}(s)) + (P_n\wt{b}(s))^* F_x(s, \wt X^{n}(s)) \bigg] dW(s).
\end{eqnarray}
By continuity of $F$,
$$
F(t\wedge \tau_n^R,\wt X^{n}(t\wedge \tau_n^R)) \to F(t\wedge \tau^R,\wt X(t\wedge \tau^R)),\quad F(0,\wt X^{n}_0) \to F(0,\wt X_0), \q\text{ a.s.}
$$

The growth conditions on $\Gamma$ yield the uniform bound: 
$$
\mathbf{1}_{[0,s\wedge \tau_n^R]}|\Gamma(s, \wt X^n(s))| \leq [1+(R+1)^k]L(s) \in L_{\mathbb{F}}^1(0,T).
$$
By continuity of $\Gamma$ in the spatial variable and the  Dominated Convergence Theorem (real-valued version, which we omit in the content below), we obtain 
$$
\begin{aligned}
\int_0^t \mathbf{1}_{[0,s\wedge \tau_n^R]}\Gamma(s, \wt X^{n}(s))ds \to \int_0^t\mathbf{1}_{[0,s\wedge \tau^R]} \Gamma(s, \wt X(s))ds,\quad  \text{ a.s.}
\end{aligned}
$$
For the remaining deterministic terms, by assumptions on $a,b,L$, we establish similar bounds
$$
\begin{aligned}
&\mathbf{1}_{[0,s\wedge \tau_n^R]}|\langle P_n\wt{a}(s),F_x(s, \wt X^{n}(s)) \rangle_{H}|\le L(s)|\wt{a}(s)|_H [1+(R+1)^k]\in  L_{\mathbb{F}}^1(0,T),\\
&\mathbf{1}_{[0,s\wedge \tau_n^R]}|\big<F_{xx}(s,  \wt X^{n}(s))P_n\wt{b}(s), P_n\wt{b}(s)\big>_{\mathcal{L}_2}|\le L(s)|\wt{b}(s)|_{\cL_2^0}|\wt{b}(s)|_{\cL_2^0}[1+(R+1)^k]\in L_{\mathbb{F}}^1(0,T),\\
&\mathbf{1}_{[0,s\wedge \tau_n^R]}|\big<\Phi_x(s,  \wt X^{n}(s)),P_n\wt{b}(s)\big>_{\mathcal{L}_2^0}|\le L(s)|\wt{b}(s)|_{\mathcal{L}_2^0} [1+(R+1)^k] \in L_{\mathbb{F}}^1(0,T). 
\end{aligned}
$$
Thus, by the Dominated Convergence Theorem, taking the limit $n \to \infty$, these deterministic integral terms in \eqref{FPn} converge to the corresponding terms.

The convergence of the stochastic integral terms is established through the following argument. By assumptions on $F_x$ and $\wt{b}$,  
$$\mathbf{1}_{[0,s\wedge \tau_n^R]}|P_n\wt{b}(s)^*F_x(s, \wt X^{n}(s))|_{\wt{H}}+\mathbf{1}_{[0,s\wedge \tau^R]}|\wt{b}(s)^*|F_x(s, \wt X(s))|_{\wt{H}}\le |b(s)|_{\cL_2^0}L(s)[1+(R+1)^k]\in L_{\mathbb{F}}^2(0,T).$$
Applying the Dominated Convergence Theorem (see \cite[Theorem 2.18]{Lu2021}) for stochastic integrals terms yields 
\begin{eqnarray*}
&&\mathbb{E}\Big|\int_0^t \mathbf{1}_{[0,s\wedge \tau_n^R]}(P_n\wt{b}(s))^* F_x(s, \wt X^{n}(s))dW(s)-\int_0^t\mathbf{1}_{[0,s\wedge \tau^R]}\wt{b}(s)^* F_x(s, \wt X(s))dW(s)\Big|^2\\
&&\le \int_0^t\mathbb{E}\Big|\mathbf{1}_{[0,s\wedge \tau_n^R]}(P_n\wt{b}(s))^* F_x(s, \wt X^{n}(s))-\mathbf{1}_{[0,s\wedge \tau^R]}\wt{b}(s)^* F_x(s, \wt X(s))\Big|_H^2ds\rightarrow 0,\quad\text{ as }n\rightarrow \infty.
\end{eqnarray*}
Similarly, 
\begin{eqnarray*}
&&\mathbb{E}\Big|\int_0^t \mathbf{1}_{[0,s\wedge \tau_n^R]}\Phi(s, \wt X^{n}(s))dW(s)-\int_0^t\mathbf{1}_{[0,s\wedge \tau^R]}\Phi(s, \wt X(s))dW(s)\Big|^2\\
&&\le \int_0^t\mathbb{E}\Big|\mathbf{1}_{[0,s\wedge \tau_n^R]}\Phi(s, \wt X^{n}(s))-\mathbf{1}_{[0,s\wedge \tau^R]}\Phi(s, \wt X(s))\Big|_H^2ds\rightarrow 0,\quad\text{ as }n\rightarrow \infty.
\end{eqnarray*}

Therefore, up to a subsequence, the stochastic integral terms converge $\mathbb{P}$-a.s. Then, letting $R\rightarrow \infty$, the equality \eqref{iw} is proved for $A=0$.

\textbf{Step 3.} In this step, we now establish the result for the general case.

Consider the approximating equation: 
$$
\wt X^{n}(t) = X_0 + \int_0^t\left( A_{n} \wt X^{n}(s) + \wt{a}(s) \right) ds + \int_0^t\wt{b}(s)  dW(s), \quad t \in [0,T].
$$

By using Proposition \ref{prop51}, we have
$$\wt X^{n} \to \wt X\text{ in }L^2(\Omega; C([0,T]; H)),\quad \text{as } n\rightarrow\infty.$$
This allows extraction of a subsequence (still denoted $\{\wt X^{n}\}_{n=1}^\infty$) converging uniformly to $\wt X$ on $\wt{\Omega}\subset\Om$ with $\mathbb{P}(\wt{\Omega})=1$. Define $\tau^R:=\inf\{s\in [0,T]:|\wt X(s)|_H>R\}$ and $\tau_n^R:=\inf\{s\in [0,T]:|\wt X^n(s)|_H>R+1,|\wt X(s)|_H>R\}$. On $\wt{\Omega}$, we have  
$$\lim_{n\rightarrow \infty}\tau_{n}^R=\tau^R,$$
and
$$\lim_{n\rightarrow \infty}\mathbf{1}_{[0,s\wedge \tau_n^R]}=\mathbf{1}_{[0,s\wedge \tau^R]},\text{ pointwise on } [0,T].$$ 
Since $A_{n}$ is bounded, we apply the established result for $A = 0$ to obtain
\begin{eqnarray}\label{Flambda}
&&F(t\wedge \tau_n^R, \wt X^{n}(t\wedge \tau_n^R))\nonumber\\ && =  F(0,\wt X_0)  + \int_0^t \mathbf{1}_{[0,s\wedge \tau_n^R]}(s)\Big[ \Gamma(s, \wt X^{n}(s)) + \langle A_{n}^* F_x(s, \wt X^{n}(s)), \wt X^{n}(s) \rangle_{H} \nonumber\\
&& \q+ \langle F_x(s, \wt X^{n}(s)), \wt{a}(s) \rangle_{H} + \frac{1}{2} \left\langle F_{xx}(s,\wt X^{n}(s)) \wt{b}(s), \wt{b}(s) \right\rangle_{\mathcal{L}_2^0}+ \wt{b}(s)^* \Phi_x(s, \wt X^{n}(s)) \Big] ds \nonumber\\
&&\q + \int_0^t \mathbf{1}_{[0,s\wedge \tau_n^R]}(s)\Big[ \Phi(s, \wt X^{n}(s)) + \wt{b}(s)^* F_x(s, \wt X^{n}(s)) \Big] dW(s).
\end{eqnarray}
Similarly, by standard properties of the Yosida approximation:
\begin{itemize}
\item $A_{n} x \to A x$ for $x \in D(A)$, and $A_{n}^* y \to A^* y$ for $y \in D(A^*)$;
\item there exists a constant $\cC$ independent of $n$ for sufficient big $n$ such that  $|n (nI-A)^{-1}|_{\cL(H)}\le \cC$.
\end{itemize}
Combining the condition $A^*F_x \in \mathbb{S}_{\mathbb{F}}^0([0,T] \times H)$ with the growth condition on $A^*F_x$ yields
\begin{eqnarray}
\mathbf{1}_{[0,s\wedge \tau_n^R]}(s)|\langle A_n^*F_x(t,x),\wt X^n(s)\rangle_H| &\3n\le\3n& |n (nI-A)^{-1}A^*F_x(t,\wt X^n(s))|_H|\wt X^n(s)|_H\nonumber\\
&\3n\le\3n& \cC L(s) [1 + (R+1)^k] (R+1)\in L_{\mathbb{F}}^1(0,T).
\end{eqnarray}
Thus, the Dominated Convergence Theorem yields:
$$
\int_0^t \mathbf{1}_{[0,s\wedge \tau_n^R]}(s)\langle A_{n}^* F_x(s, \wt X^{n}(s)), \wt X^{n}(s) \rangle_{H} ds \to \int_0^t \mathbf{1}_{[0,s\wedge \tau^R]}(s) \langle A^* F_x(s, \wt X(s)), \wt X(s) \rangle_{H} ds \quad \text{a.s.}
$$

Similar to the convergence arguments in Step 2, by taking the limit as $n \to \infty$ in equation \eqref{Flambda}, and by the growth conditions and continuity of $F$, $\Gamma$, $\Phi$, and their derivatives, those integrands in \eqref{Flambda} converge to the corresponding terms.

Taking the limit as $R \to \infty$ in equation \eqref{Flambda} completes the proof.
\end{proof}

\end{document}